\documentclass[11pt]{article}
\usepackage[T1]{fontenc}
\usepackage{graphicx,amsmath,amsthm,amsfonts,amssymb,microtype,thmtools,underscore,mathtools,anyfontsize,thm-restate,MnSymbol}
\usepackage[shortlabels]{enumitem}
\setlist[itemize]{topsep=0ex,itemsep=0ex,parsep=0.4ex}
\setlist[enumerate]{topsep=0ex,itemsep=0ex,parsep=0.4ex}
\usepackage[usenames,dvipsnames,svgnames,table]{xcolor}
\usepackage[unicode=true]{hyperref}
\hypersetup{
colorlinks=true,
breaklinks=true,
linkcolor={black},
citecolor={black},
urlcolor={blue!60!black},
pdftitle={Product Structure and Tree-Decompositions},
pdfauthor={Chun-Hung Liu, Sergey Norin, David~R.~Wood}}
\usepackage[noabbrev,capitalise]{cleveref}
\crefname{lem}{Lemma}{Lemmas}
\crefname{thm}{Theorem}{Theorems}
\crefname{cor}{Corollary}{Corollaries}
\crefname{prop}{Proposition}{Propositions}
\crefname{conj}{Conjecture}{Conjectures}
\crefname{openproblem}{Open Problem}{Open Problems}
\crefname{claim}{Claim}{Claims}
\crefformat{equation}{(#2#1#3)}
\Crefformat{equation}{Equation #2(#1)#3}

\newcommand{\defn}[1]{\textcolor{Maroon}{\emph{#1}}}
\newcommand{\mathdefn}[1]{\textcolor{Maroon}{#1}}
\newcommand{\TT}{\mathcal{T}}

\newcommand{\JJ}{\mathcal{J}}
\newcommand{\GG}{\mathcal{G}}
\newcommand{\AAA}{\mathcal{A}}
\newcommand{\BB}{\mathcal{B}}
\newcommand{\DD}{\mathcal{D}}

\newcommand{\FF}{\mathcal{F}}
\usepackage[longnamesfirst,numbers,sort&compress]{natbib}
\makeatletter
\def\NAT@spacechar{~}
\makeatother
\usepackage[margin=30mm]{geometry}
\newcommand{\half}{\ensuremath{\protect\tfrac{1}{2}}}
\DeclarePairedDelimiter{\floor}{\lfloor}{\rfloor}
\DeclarePairedDelimiter{\ceil}{\lceil}{\rceil}
\renewcommand{\geq}{\geqslant}
\renewcommand{\leq}{\leqslant}
\newcommand{\subsetsim}{
\mathrel{\substack{\textstyle\subset\\[-0.6ex]\textstyle\sim\\[-0.4ex]}}}
\newcommand{\undirected}[1]{\overlinesegment{#1}}
\newcommand{\CartProd}{\mathbin{\square}}
\newcommand{\StrongProd}{\mathbin{\boxtimes}}
\newcommand{\DirectedStrongProd}{\mathbin{\boxslash}}
\DeclareMathOperator{\indeg}{\Delta^-}
\DeclareMathOperator{\dist}{dist}
\DeclareMathOperator{\diam}{diam}

\DeclareMathOperator{\tw}{tw}
\DeclareMathOperator{\pw}{pw}
\DeclareMathOperator{\bw}{bw}
\DeclareMathOperator{\fvn}{fvn}
\DeclareMathOperator{\utw}{utw}
\DeclareMathOperator{\talpha}{tree-\alpha}
\DeclareMathOperator{\tchi}{tree-\chi}
\DeclareMathOperator{\ttw}{tree-tw}
\DeclareMathOperator{\ttd}{tree-td}
\DeclareMathOperator{\td}{td}
\DeclareMathOperator{\ltw}{ltw}
\DeclareMathOperator{\rtw}{rtw}
\DeclareMathOperator{\lpw}{lpw}
\newcommand{\treef}{\text{tree-}f}
\DeclareMathOperator{\tpw}{tree-pw}
\DeclareMathOperator{\tbw}{tree-bw}
\DeclareMathOperator{\twtwa}{(\tw\StrongProd\tw\StrongProd\,\ast)}
\DeclareMathOperator{\pwpwa}{(\pw\StrongProd\pw\StrongProd\,\ast)}
\DeclareMathOperator{\twpwa}{(\tw\StrongProd\pw\StrongProd\,\ast)}

\DeclareMathOperator{\twtw}{(\tw\StrongProd\tw)}
\DeclareMathOperator{\pwpw}{(\pw\StrongProd\pw)}
\DeclareMathOperator{\twpw}{(\tw\StrongProd\pw)}

\DeclareMathOperator{\tdtd}{(\td\StrongProd\td)}
\DeclareMathOperator{\TwIntTw}{(\tw\cap\tw)}
\DeclareMathOperator{\TwIntPw}{(\tw\cap\pw)}

\DeclareMathOperator{\PwIntPw}{(\pw\cap\pw)}

\renewcommand{\thefootnote}{\fnsymbol{footnote}}
\numberwithin{equation}{section}
\theoremstyle{plain}
\newtheorem{thm}[equation]{Theorem}
\newtheorem{lem}[equation]{Lemma}
\newtheorem{cor}[equation]{Corollary}
\newtheorem{prop}[equation]{Proposition}
\newtheorem{obs}[equation]{Observation}
\newtheorem{open}[equation]{Open Problem}

\theoremstyle{definition}

\newcommand{\PP}{\mathcal{P}}
\newcommand{\LL}{\mathcal{L}}

\newcommand{\NN}{\mathbb{N}}
\newcommand{\ZZ}{\mathbb{Z}}
\newcommand{\RR}{\mathbb{R}}
\newcommand{\eps}{\varepsilon}
 \newcommand{\mc}[1]{\mathcal{#1}}
 \newcommand{\bb}[1]{\mathbb{#1}}
 \newcommand{\brm}[1]{\operatorname{#1}}
\begin{document}

\author{Chun-Hung Liu\,\footnotemark[3]
\qquad Sergey Norin\,\footnotemark[4]
\qquad David~R.~Wood\,\footnotemark[2]}

\footnotetext[3]{Department of Mathematics, Texas A\&M University, USA (\texttt{chliu@tamu.edu}). Partially supported by NSF under CAREER award DMS-2144042.}

\footnotetext[4]{Department of Mathematics and Statistics, McGill University, Montr\'eal, Canada (\texttt{sergey.norin@mcgill.ca}). Research supported by an NSERC Discovery grant.}

\footnotetext[2]{School of Mathematics, Monash   University, Melbourne, Australia  (\texttt{david.wood@monash.edu}). Research supported by the Australian Research Council.}

\sloppy

\title{\bf\boldmath Product Structure and Tree-Decompositions}

\maketitle
% \thanks{\textbf{MSC Classification}: ???}

\begin{abstract}
This paper explores the structure of graphs defined by an excluded minor or an excluded odd minor through the lens of graph products and tree-decompositions. We prove that every graph excluding a fixed odd minor is contained in the strong product of two graphs each with bounded treewidth. For graphs excluding a fixed minor, we strengthen the result by showing that every such graph is contained in the strong product of two digraphs with bounded indegree and with bounded treewidth (ignoring the edge orientation). This result has the advantage that the product now has bounded degeneracy. 

In the setting of 3-term products, we show the following more precise quantitative result: every $K_t$-minor-free graph is contained in $H_1\StrongProd H_2 \StrongProd K_{c(t)}$ where $\tw(H_i)\leq t-2$. This treewidth bound is close to tight: in any such result with $\tw(H_i)$ bounded, both $H_1$ and $H_2$ can be forced to contain any graph of treewidth $t-5$, implying $\tw(H_1)\geq t-5$ and $\tw(H_2)\geq t-5$. Analogous lower and upper bounds are shown for any excluded minor, where the minimum possible bound on $\tw(H_i)$ is tied to the treedepth of the excluded minor.

Subgraphs of the product of two graphs with bounded treewidth have two tree-decompositions where any bag from the first decomposition intersects any bag from the second decomposition in a bounded number, $k$, of vertices, so called $k$-orthogonal tree-decompositions. We show that graphs excluding a fixed odd-minor have a tree-decomposition and a path-decomposition that are $O(1)$-orthogonal. This result does not generalise for classes admitting strongly sublinear separators, since we show (using topological methods) that triangulations of 3-dimensional grids do not have a pair of $O(1)$-orthogonal tree-decompositions. 

Relaxing the notion of orthogonal tree-decompositions, we study tree-decompositions in which each bag has a given parameter bounded. We show that graphs excluding a fixed odd-minor have a tree-decomposition in which each bag has bounded pathwidth. This result is best possible in that  `pathwidth' cannot be replaced by `bandwidth' or `treedepth'. Moreover, we characterize the minor-closed classes that have a tree-decomposition in which each bag has bounded bandwidth, or each bag has bounded treedepth. 

None of our results can be generalised for bounded degree graphs, since we prove that for fixed $\eps>0$ there exists $d$ such that in every tree-decomposition of a random $d$-regular $n$-vertex graph some bag has treewidth at least $(1-\eps)n$ asymptotically almost surely.
\end{abstract}

\renewcommand{\thefootnote}
{\arabic{footnote}}

\newpage
\tableofcontents
\newpage

\section{\boldmath Introduction}
\label{Introduction}

This paper explores the structure of graphs defined by an excluded minor or an excluded odd minor through the lens of graph products and tree-decompositions. The focus is on the following  four related structures:
\begin{itemize}
    \item strong products of two bounded treewidth graphs (see \cref{BoundedTreewidthProducts}); 
    \item strong products of two bounded treewidth graphs and a bounded size complete graph  (see \cref{3Dproducts}); 
    \item two tree-decompositions in which the intersection of any bag from the first tree-decomposition and any bag from the second tree-decomposition has bounded size (see \cref{OrthogonalTreeDecompositions}); and 
    \item tree-decompositions in which every bag induces a subgraph with bounded treewidth  (see \cref{TreeTreewidth}).
\end{itemize}
We start by providing background on the first two structures.

Graph product structure theory describes graphs in complicated graph classes as subgraphs of products of graphs in simpler graph classes, typically with bounded treewidth or bounded pathwidth. As defined in \cref{Definitions}, the treewidth of a graph $G$, denoted by \defn{$\tw(G)$}, is the standard measure of how similar $G$ is to a tree, and the pathwidth of a graph $G$, denoted by \defn{$\pw(G)$}, is the standard measure of how similar $G$ is to a path.

As illustrated in \cref{ProductExample}, the \defn{Cartesian product} $A \CartProd B$ of graphs $A$ and $B$ has vertex-set $V(A) \times V(B)$, where distinct vertices $(v,x), (w,y)$ are adjacent if $v = w$ and $xy \in E(B)$, or $x = y$ and $vw \in E(A)$. 
The \defn{direct product} $A \times B$ of graphs $A$ and $B$ has vertex-set $V(A) \times V(B)$, where distinct vertices $(v,x), (w,y)$ are adjacent if $vw\in E(A)$ and $xy \in E(B)$. 
The \defn{strong product} $A \StrongProd B$ is the union of $A\CartProd B$ and $A\times B$. 
%By definition, $A\CartProd B\subseteq A\StrongProd B$. 

\begin{figure}[!h]
\centering
\includegraphics{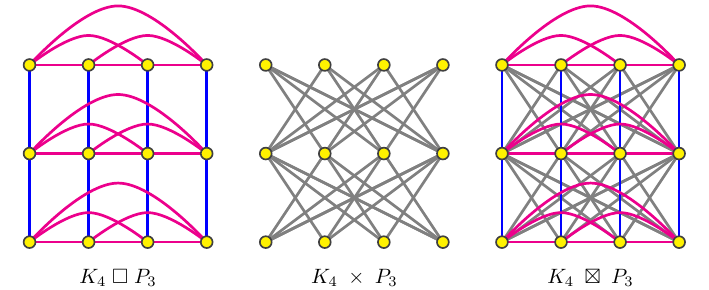}
\caption{\label{ProductExample} Examples of graph products}
\end{figure}

The following Planar Graph Product Structure Theorem\footnote{\citet{DJMMUW20} first proved \cref{PGPST}(a) with $\tw(H)\leq 8$. In follow-up work, \citet{UWY22} improved the bound to $\tw(H)\leq 6$. Part (b) is due to Dujmovi\'c, and is presented in \citep{UWY22}. Part (c) is in the original paper of  \citet{DJMMUW20}. \citet{ISW} proved a general result (see \cref{JJstMain}) that leads to a new proof of part (c).} is the classical example of a graph product structure theorem. Here, a graph $H$ is \defn{contained} in a graph $G$ if $H$ is isomorphic to a subgraph of $G$, written \defn{$H \subsetsim G$}. 

\begin{thm}
\label{PGPST} 
For every planar graph $G$:
\begin{enumerate}[(a)]
\item $G\subsetsim H \StrongProd P$ for some graph $H$ with $\tw(H)\leq 6$ and path $P$ \textup{\citep{UWY22}},
\item $G\subsetsim H \StrongProd P \StrongProd K_2$ for some graph $H$ with $\tw(H)\leq 4$ and path $P$ \textup{\citep{UWY22}},
\item $G\subsetsim H \StrongProd P \StrongProd K_3$ for some graph $H$ with $\tw(H)\leq 3$ and path $P$ \textup{\citep{DJMMUW20}}.
\end{enumerate}
\end{thm}

\cref{PGPST} provides a powerful tool for studying questions about planar graphs, by a reduction to bounded treewidth graphs. Indeed, this result has been the key for resolving several open problems including queue layouts~\cite{DJMMUW20}, nonrepetitive colourings~\cite{DEJWW20}, centred colourings~\cite{DFMS21}, adjacency labelling schemes~\cite{BGP22,EJM23,DEGJMM21}, twin-width~\cite{BKW,JP22}, infinite graphs~\cite{HMSTW}, and comparable box dimension~\cite{DGLTU22}. In several of these applications, because the dependence on $\tw(H)$ is often exponential, the best bounds are obtained by applying the 3-term product in \cref{PGPST}(c). The $\tw(H)\leq 3$ bound here is best possible in any result of the form $G\subsetsim H \StrongProd P \StrongProd K_c$ (see \citep{DJMMUW20}). 

Motivated by \cref{PGPST}, \citet{BDJMW22} defined the \defn{row treewidth} of a graph $G$, denoted by \defn{$\rtw(G)$}, to be the minimum $c\in\NN$ such that $G\subsetsim H\StrongProd P$ for some path $P$ and graph $H$ with $\tw(H)\leq c$. \cref{PGPST}(a) says that planar graphs have row treewidth at most 6. Other classes with bounded row treewidth include graphs embeddable on any fixed surface~\cite{DJMMUW20,DHHW22}. More generally, \citet{DJMMUW20} established the following characterisation of minor-closed classes with bounded row treewidth. A graph $X$ is \defn{apex} if $V(X)=\emptyset$ or $X-a$ is planar for some vertex $a\in V(X)$. 

\begin{thm}[\citep{DJMMUW20}]
\label{ApexMinorFreeProduct}
A minor-closed class $\GG$ has bounded row treewidth if and only if some apex graph is not in $\GG$.
\end{thm}

%Analogous product structure theorems have been established for more general classes, including graphs of bounded genus~\citep{DJMMUW20,DHHW22}, apex-minor-free graphs~\citep{DJMMUW20,DHHJLMMRW,ISW} (see \cref{UsingUTW} for more details), and for various non-minor-closed classes~\citep{DMW23,HW24,BDHK22,DHSW}. 

See \cref{UsingUTW} for recent developments related to \cref{ApexMinorFreeProduct}, and see \citep{DMW23,HW24,DHHJLMMRW,ISW,BDHK22,DHSW,UTW,HJMW24,BDJMW22,DDEHJMMSW24,DJMMW24,DHJMMW24,HJ24} for more results in graph product structure theory.

\subsection{\boldmath Products of Bounded Treewidth Graphs}
\label{BoundedTreewidthProducts}

Now consider the structure of graphs in an arbitrary proper minor-closed class. Building on the Graph Minor Structure Theorem of \citet{RS-XVI}, the Graph Minor Product Structure Theorem of \citet{DJMMUW20} describes graphs in any proper minor-closed class in terms of a tree-decomposition in which each torso has a product structure including apex vertices (see \cref{ProductStructureXMinorFree} below). It is natural to seek a product structure theorem simply in terms of the product of two graphs (as in \cref{PGPST,ApexMinorFreeProduct}). The following definition and theorem achieve this goal. Define the \defn{$\twtw$-number} of a graph $G$, denoted by \defn{$\twtw(G)$}, to be the minimum integer $k$ such that $G$ is contained in $H_1\StrongProd H_2$, for some graphs $H_1$ and $H_2$ with $\tw(H_1)\leq k$ and $\tw(H_2)\leq k$. Define \defn{$\twpw(G)$} and \defn{$\pwpw(G)$} analogously. 

\begin{thm}
\label{NoKtMinorProduct}
For any $t\in\NN$ there exists $c\in\NN$ such that 
$\twtw(G)\leq c$ for every $K_t$-minor-free graph $G$.
\end{thm}

\cref{NoKtMinorProduct} is best possible in multiple ways. First, we show that $K_6$-minor-free graphs have unbounded $(\tw\StrongProd \pw)$-number (see \cref{twpw-LowerBound}), and $K_4$-minor-free graphs have unbounded $(\pw\StrongProd \pw)$-number (see \cref{pwpw-LowerBound}). So $\twtw$ cannot be replaced by $\twpw$ or $\pwpw$ in \cref{NoKtMinorProduct}. We also show that if every $K_t$-minor-free graph is contained in $H_1\boxtimes H_2$ for some graphs $H_1$ and $H_2$ with bounded treewidth, then both graphs $H_1$ and $H_2$ must have treewidth increasing with $t$. In fact, both $H_1$ and $H_2$ can be forced to contain any graph with treewidth $t-5$, implying $\tw(H_1)\geq t-5$ and $\tw(H_2)\geq t-5$ (see \cref{ForceH-Kt}). In particular, $\twtw(G)\geq t-5$ for some $K_t$-minor-free graph $G$. 

Note that $H_1\StrongProd H_2$ may be dense even when both $H_1$ and $H_2$ have bounded treewidth. For example, as illustrated in  \cref{StarStarProduct}, the strong product of two $n$-leaf stars, $K_{1,n}\StrongProd K_{1,n}$, contains the complete bipartite graph $K_{n,n}$ and in fact contains $K_{1,n,n}$. Since $\pw(K_{1,n})\leq 1$, for every bipartite graph $G$,
\begin{equation}
\label{twtw-bipartite}
\twtw(G)\leq \twpw(G)\leq \pwpw(G) \leq 1.
\end{equation}

\begin{figure}[ht]
\centering
\includegraphics{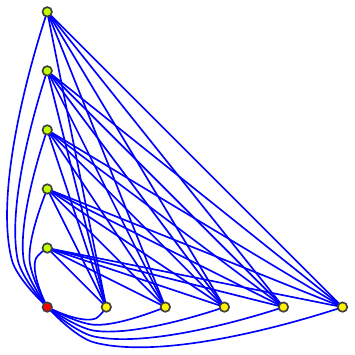}
\caption{\label{StarStarProduct} $K_{1,m,n}$ is contained in $K_{1,m}\StrongProd K_{1,n}$.}
\end{figure}

We therefore seek results analogous to \cref{NoKtMinorProduct} for more general graph classes that include dense graphs. The class of graphs excluding a fixed odd minor, which includes all bipartite graphs, are such a class. We prove the following product structure theorem in this setting, which is stronger than \cref{NoKtMinorProduct}. 

\begin{thm}
\label{NoOddKtMinorProduct}
For any $t\in\NN$ there exists $c\in\NN$ such that $\twtw(G)\leq c$ for every odd-$K_t$-minor-free graph $G$.
\end{thm}

\cref{NoOddKtMinorProduct} and a recent result of \citet{HJ24} are the first product structure theorems accommodating dense graphs.

A weak point of \cref{NoKtMinorProduct} and \cref{KtMinorFreeThreeDimProduct}, stated in the next subsection, is they embed sparse graphs (excluding a fixed minor) in dense graphs (products of bounded treewidth graphs). We therefore introduce the following definition, which leads to a sparse strengthening of \cref{NoKtMinorProduct}. As illustrated in \cref{DirectedGrid}, the strong product \defn{$D_1\DirectedStrongProd D_2$} of two digraphs $D_1$ and $D_2$ is the digraph with vertex set $V(D_1)\times V(D_2)$, 
where $(x,y)(x',y')$ is a directed edge of $D_1\DirectedStrongProd D_2$ if and only if 
(a) $x=x'$ or $xx'\in E(D_1)$, and
(b) $y=y'$ or $yy'\in E(D_2)$, and  
(c) $x\neq x'$ or $y\neq y'$. 
Note that a `digraph' allows both $vw$ and $wv$ to be directed edges\footnote{In this paper, a directed edge $(x,x')$ with tail $x$ and head $x'$ is denoted by $xx'$ to avoid confusion with the notation for vertices in a strong product.}.

\begin{figure}[ht]
    \centering
    \includegraphics{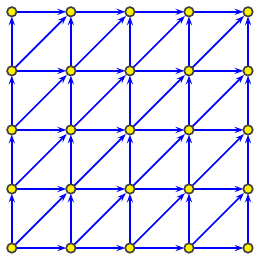}
    \caption{    \label{DirectedGrid}
Strong product of two directed paths.}
\end{figure}

For a digraph $G$, let \defn{$\undirected{G}$} be the undirected graph underlying $G$. By definition, 
$\undirected{D_1\DirectedStrongProd D_2} \subseteq \undirected{D_1}\StrongProd\undirected{D_2} $, and if $D_1$ (or $D_2$) has all edges bidirected, then $\undirected{D_1\DirectedStrongProd D_2} = \undirected{D_1} \StrongProd \undirected{D_2}$. So $D_1\DirectedStrongProd D_2$ is a natural generalisation of $D_1\StrongProd D_2$. The difference is that if $v_iw_i$ is a directed edge in $D_i$ for $i=1,2$, then $(v_1,w_2)$ is not necessarily adjacent to $(w_1,v_2)$ in $D_1\DirectedStrongProd D_2$. For example, the $K_{m,n}$ subgraph in $K_{1,m}\StrongProd K_{1,n}$ does not appear in 
$\undirected{K_{1,m}\DirectedStrongProd K_{1,n}}$ if the edges of $K_{1,m}$ and $K_{1,n}$ are oriented away from the centre of the star. Let \defn{$\indeg(D)$} be the maximum indegree of a vertex in a digraph $D$. If $\indeg(D_1)\leq d_1$ and $\indeg(D_2)\leq d_2$, then $\indeg(D_1\DirectedStrongProd D_2) \leq d_1d_2+d_1+d_2$. So if $d_1$ and $d_2$ are bounded, then 
$D_1\DirectedStrongProd D_2$ has bounded indegree, implying that 
$\undirected{D_1\DirectedStrongProd D_2}$ has 
bounded degeneracy. In contrast, $H_1\StrongProd H_2$ has unbounded degeneracy whenever $H_1$ and $H_2$ are undirected graphs with unbounded maximum degree. We prove the following product structure theorem, where an undirected graph $H$ is \defn{contained} in a digraph $G$ if $H$ is contained in $\undirected{G}$. 

\begin{restatable}{thm}{DirectedProduct}
\label{DirectedProduct}
For any $t\in\NN$ there exist $k,d\in\NN$ such that every $K_t$-minor-free graph $G$ is contained in $D_1\DirectedStrongProd D_2$, for some digraphs $D_1$ and $D_2$ with $\tw(\undirected{D_i})\leq k$ and $\indeg(D_i)\leq d$ for each $i\in\{1,2\}$. 
\end{restatable}

%\begin{restatable}{thm}{InducedProductStructureXMinorFree} \label{InducedProductStructureXMinorFree} For any $t\in\NN$ there exist $c,d\in\NN$ such that every $K_t$-minor-free graph is contained in a $d$-degenerate induced subgraph of $H_1\StrongProd H_2$, for some graphs $H_1$ and $H_2$ with $\tw(H_1)\leq c$ and $\tw(H_2)\leq c$. \end{restatable}

Note that \cref{DirectedProduct} implies and strengthens \cref{NoKtMinorProduct} since $\undirected{D_1\DirectedStrongProd D_2} \subseteq \undirected{D_1}\StrongProd\undirected{D_2} $. The big difference is that $\undirected{D_1\DirectedStrongProd D_2}$ in \cref{DirectedProduct} has bounded degeneracy. It turns out that \cref{NoKtMinorProduct,NoOddKtMinorProduct} are straightforward consequences of the `low-treewidth' colouring result of \citet{DDOSRSV04} and \citet{DHK-SODA10} (see \cref{Joins}). However, the proof of %\cref{InducedProductStructureXMinorFree} 
\cref{DirectedProduct} does not use low-treewidth colourings and requires more sophisticated methods (see \cref{SparseDirectedProducts}). 

\subsection{\boldmath 3-Term Products}
\label{3Dproducts}

Inspired by \cref{PGPST}(b) and (c) for planar graphs, we consider products of the form $H_1\StrongProd H_2 \StrongProd K_c$ where $H_1$ and $H_2$ have bounded treewidth. Here the goal is to minimise $\tw(H_i)$ while keeping $c$ constant. In this setting, one can show  more precise bounds on $\tw(H_i)$ and it makes sense to consider excluded minors other than complete graphs. Our first result here is the following product structure theorem for $K_t$-minor-free graphs, which immediately implies \cref{NoKtMinorProduct} since $\tw(H_2 \StrongProd K_c)\leq (\tw(H_2)+1)c-1$.

\begin{restatable}{thm}{KtMinorFreeThreeDimProduct}
\label{KtMinorFreeThreeDimProduct}
For any $t\in\NN$ there exists $c\in\NN$ such that every $K_t$-minor-free graph is contained in $H_1\StrongProd H_2 \StrongProd K_{c}$, for some graphs $H_1$ and $H_2$ with $\tw(H_1)\leq t-2$ and $\tw(H_2)\leq t-2$. 
\end{restatable}

We show that the bounds on $\tw(H_i)$ in \cref{KtMinorFreeThreeDimProduct} are almost tight. In particular, for every $c$, there is a $K_t$-minor-free graph $G$ such that if $G$ is contained in $H_1\StrongProd H_2 \StrongProd K_c$, then $H_1$ or $H_2$ has treewidth at least $t-5$ (see \cref{twtw-LowerBound-Kt}). 
As with the case of 2-term products, we actually show a stronger lower bound, which essentially says that both $H_1$ and $H_2$ can be forced to contain any given graph with treewidth $t-5$ (see \cref{ForceH-Kt}). 

\cref{KtMinorFreeThreeDimProduct} is in fact a special case of a more general result for arbitrary proper minor-closed classes, where the bounds on $\tw(H_i)$ depend on the structure of the excluded minor (see \cref{JstConstructThreeDimProduct}). For a graph class $\GG$, let \defn{$\twtwa(\GG)$} be the minimum $k\in\NN$ such that for some $c\in\NN$ every graph in $\GG$ is contained in $H_1\StrongProd H_2 \StrongProd K_c$ for some graphs $H_1$ and $H_2$ with $\tw(H_1)\leq k$ and $\tw(H_2)\leq k$. For a graph $X$, let \defn{$\GG_X$} be the class of $X$-minor-free graphs. The above results show that $\twtwa(\GG_{K_t})\in\{t-5,t-4,t-3,t-2\}$. Define $\twpwa(\GG)$ and $\pwpwa(\GG)$ analogously. 
We show that $\twtwa(\GG_X)$ is tied to the treedepth \defn{$\td(X)$} of $X$. In particular, 
\begin{equation}
\label{kX-td}
    \tfrac12 (\td(X)-2) \leq \twtwa(\GG_X) \leq 2^{\td(X)+1} -3.
\end{equation}
See \cref{XMinorFree-twtw-td} and \cref{XMinorFree-ThreeDimProduct-Treedepth} for the lower and upper bounds respectively. We prove other bounds on $\twtwa(\GG_X)$ that in some cases are better than the above treedepth bound. Let the \defn{apex-number} of a graph $X$ be \[a(X):=\min\{|A|:A\subseteq V(X),\,X-A\text{ is planar}\}.\] 
For a lower bound, we show that if $X$ is $a(X)$-connected, then $ \twtwa(\GG_X)\geq a(X)-1$ (see \cref{twtw-LowerBound}). Let the \defn{vertex-cover-number} of $X$ be \[\tau(X):=\min\{|S|:S\subseteq V(G),\,E(G-S)=\emptyset\}.\] 
For an upper bound, we show that $\twtwa(\GG_X) \leq\tau(X)$ (see \cref{XMinorFreeThreeDimProduct}). For example, together these bounds show that $\twtwa(\GG_{K_{s,t}})\in\{s-3,s-2,s-1,s\}$ for $t\geq s\geq 1$ (see \cref{twtw-LowerBound-Kst,KstMinorFreeThreeDimProduct}). 

\subsection{\boldmath Orthogonal Tree-Decompositions}
\label{OrthogonalTreeDecompositions}

Two tree-decompositions $(B_x:x\in V(S))$ and $(C_y:y\in V(T))$ of a graph $G$ are \defn{$k$-orthogonal} if $|B_x\cap C_y|\leq k$ for each $x\in V(S)$ and $y\in V(T)$. As a simple example, as illustrated in \cref{OrthogonalDecompsGrid}, any planar grid graph $G$ has a path-decomposition in which each bag is the union of two consecutive rows, and $G$ has a path-decomposition in which each bag is the union of two consecutive columns. 
The intersection of two columns and two rows has four vertices. So these two path-decompositions are 4-orthogonal. 

\begin{figure}[ht]
\centering
\includegraphics{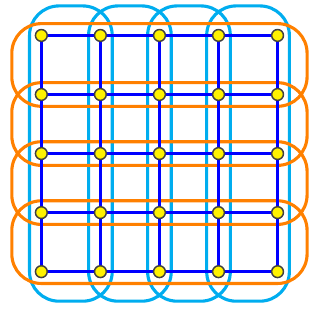}
\caption{\label{OrthogonalDecompsGrid}
Two 4-orthogonal path-decompositions of the grid graph.}
\end{figure}

The \defn{$\TwIntTw$-number} of a graph $G$, denoted \defn{$\TwIntTw(G)$}, is the minimum integer $k$ such that $G$ has a pair of $k$-orthogonal tree-decompositions. Similarly, let \defn{$\TwIntPw(G)$} be the minimum integer $k$ such that $G$ has a tree-decomposition and a path-decomposition that are $k$-orthogonal. Let \defn{$\PwIntPw(G)$} be the minimum integer $k$ such that $G$ has a pair of $k$-orthogonal path-decompositions. These concepts were independently introduced  by 
\citet{Stav15,Stav16} and \citet{DJMNW18}. 
\citet{Stav15,Stav16} called $\TwIntTw$-number the \defn{2-medianwidth}, and \citet{DJMNW18} called $\TwIntTw$-number the \defn{2-dimensional treewidth}. We use $\TwIntTw$-number for consistency with $\TwIntPw$-number and $\PwIntPw$-number. 

As an example, for any bipartite graph $G$ with bipartition $\{\{v_1,\dots,v_n\},\{w_1,\dots,w_m\}\}$, the path-decompositions $(\{v_1,w_1,\dots,w_m\},\{v_2,w_1,\dots,w_m\},\dots,\{v_n,w_1,\dots,w_m\})$
and 
$(\{w_1,v_1,\dots,v_n\},\{w_2,v_1,\dots,v_n\},\dots,\{w_m,v_1,\dots,v_n\})$ of $G$ are 2-orthogonal. Thus 
$$\TwIntTw(G) \leq \TwIntPw(G) \leq \PwIntPw(G)\leq 2.$$

These definitions are related to graph products, since it is easily seen (see \cref{TwIntTW-twtw}) that for every graph $G$, 
\begin{align*}
\TwIntTw(G) & \leq (\twtw(G)+1)^2\\
\TwIntPw(G) & \leq (\twpw(G)+1)^2\\
\PwIntPw(G) & \leq (\pwpw(G)+1)^2.
\end{align*}
On the other hand, we now show that each of $\twtw$-number, $\twpw$-number and $\pwpw$-number are separated\footnote{A \defn{graph parameter} is a function $f$ such that $f(G)\in\NN_0$ for every graph $G$, and $f(G_1)=f(G_2)$ for all isomorphic graphs $G_1$ and $G_2$. Graph parameters include minimum degree $\delta(G)$, maximum degree $\Delta(G)$, chromatic number $\chi(G)$, independence number $\alpha(G)$, clique-number $\omega(G)$, treewidth $\tw(G)$ and pathwidth $\pw(G)$. Two graphs parameters $f_1$ and $f_2$ are \defn{tied} if there is a function $g$ such that for every graph $G$, $f_1(G)\leq g(f_2(G))$ and $f_2(G)\leq g(f_1(G))$; otherwise $f_1$ and $f_2$ are \defn{separated}. A graph parameter $f$ is \defn{bounded} on a graph class $\GG$ if there exists $c\in\NN_0$ such that $f(G)\leq c$ for every $G\in\GG$.} from each of $\TwIntTw$-number and $\TwIntPw$-number. \citet{FJMTW18} showed that Burling graphs $G$ satisfy $\TwIntTw(G)\leq \TwIntPw(G) \leq 2$, and \citet{Burling65} showed that $\chi(G)\to\infty$. On the other hand, a product colouring shows that for any graph $H$, \[\chi(H)\leq(\twtw(H)+1)^2 \leq (\twpw(H)+1)^2 \leq (\pwpw(H)+1)^2.\] Thus, for Burling graphs $G$, \[\pwpw(G)\geq \twpw(G)\geq \twtw(G)\to\infty.\]
Graphs with bounded $\PwIntPw$-number do have bounded chromatic number~\citep{DJMNW18}, so
chromatic number does not separate 
$\PwIntPw$-number and any of $\twtw$-number, $\twpw$-number and $\pwpw$-number. Indeed, it is open whether there is a function $f$ such that $\twtw(G)\leq f( \PwIntPw(G))$ for every graph $G$. 

\citet{DJMNW18} proved the following result for graphs excluding a fixed minor; here we prove it in the more general setting of graphs excluding a fixed odd minor.

\begin{restatable}{thm}{MainTWIntPW}
\label{OddMinorFree-TreePathDecompsOrtho}
    For any graph $X$ there exists $k\in\NN$ such that for every odd-$X$-minor-free graph $G$,
    $$\TwIntPw(G) \leq k.$$
\end{restatable}

\cref{OddMinorFree-TreePathDecompsOrtho} is in stark contrast to the lower bounds discussed above, which show that even $K_6$-minor-free graphs have unbounded $\twpw$-number.

Our final result regarding $\TwIntTw$-number shows that triangulations of $n\times n\times n$ grids (despite admitting strongly sublinear balanced separators) have $\TwIntTw$-number $\Theta(n)$ (see \cref{3Dgrid}). 

\subsection{\boldmath Tree-Treewidth}
\label{TreeTreewidth}

Graphs with bounded $\TwIntTw$-number (which is implied if $\twtw$-number is bounded) have tree-decompositions in which each bag induces a subgraph with bounded treewidth (see \cref{ttw-otw}). This property is of independent interest, and we prove several results in this direction. We in fact work in a more general setting that considers tree-decompositions where each bag induces a subgraph where a given graph parameter is bounded. For a graph parameter $f$ and graph $G$, let \defn{$\treef(G)$} be the minimum integer $k$ such that $G$ has a tree-decomposition $(B_x:x\in V(T))$ such that $f(G[B_x])\leq k$ for each node $x\in V(T)$.

This concept in the case of $\tchi$ was introduced by \citet{Seymour16}, and has since been %widely  
studied in~\citep{HK17,BFMMSTT19,HRWY21}. The parameter $\talpha$ was introduced by \citet{DMS21,DMS24a,DMS24b} and 
has also attracted substantial interest~\citep{DFGKM24,DKKMMSW24,MR22,AACHSV24}, especially for classes defined by an excluded induced subgraph. Note that $\tw(G)=\treef(G)$ where $f(G):=|V(G)|-1$. 
Several papers recently studied tree-diameter~\citep{BS24,DG07}. 
%connected treewidth \citep{DM18,HW16a}

Our focus is on $\ttw$, and to a lesser extent $\tpw$, $\tbw$ and $\ttd$, where \defn{$\bw$} denotes the bandwidth. %, where $\pw(G)$, $\bw(G)$ and $\td(G)$ are the pathwidth, bandwidth and treedepth  of $G$ respectively. 
As an example, consider any bipartite graph $G$ with bipartition $\{ \{v_1,\dots,v_m\},\{w_1,\dots,w_n\}\}$. Let $B_i:= \{v_i,w_1,\dots,w_n\}$. Then $(B_1,B_2,\dots,B_m)$ is a path-decomposition of $G$, where each bag induces a star plus isolated vertices, which has treedepth 2. 
Thus,
\begin{equation}
\label{Bipartite-ttd-ttw-tpw}
    \ttd(G)\leq 2\text{ and }\ttw(G)\leq \tpw(G) \leq 1.
\end{equation}

The following basic inequalities hold for any graph $G$. Since $\chi(G)\leq\tw(G)+1$, $$\tchi(G)\leq\ttw(G)+1.$$
Since $\tw(G)\leq\pw(G)\leq\bw(G)$, 
$$\ttw(G)\leq\tpw(G)\leq\tbw(G).$$
Since $\tw(G) \leq \pw(G) \leq \td(G)-1$,
$$\ttw(G)\leq\tpw(G)\leq\ttd(G)-1.$$
Now compare tree-treewidth and $\TwIntTw$-number. It is easily seen  (see \cref{ttw-otw}) that for every graph $G$, 
\begin{align}\label{e:tpw-twintpw}
    \ttw(G) &\leq \TwIntTw(G)-1 \text{ and} \notag\\
      \tpw(G) &\leq \TwIntPw(G)-1.
    \end{align}
On the other hand, $\TwIntTw$-number is not bounded by any function of $\ttw$. In particular, we describe a class of graphs $G$ with $\ttd(G)\leq 4$ and thus 
$\ttw(G)\leq\tpw(G)\leq 3$ and with unbounded $\TwIntTw$-number (see \cref{Separating}). This says that $\TwIntTw$-number is sandwiched between $\ttw$ and $\twtw$-number, and is separated from both. 

\cref{OddMinorFree-TreePathDecompsOrtho} and \eqref{e:tpw-twintpw} imply that  graphs excluding any fixed odd-minor have bounded  tree-pathwidth. We show that this result is best possible in the sense that tree-pathwidth cannot be replaced by tree-bandwidth (even for $K_6$-minor-free graphs; see \cref{Pyramid}) or by tree-treedepth (even for planar graphs; see \cref{ApplyHex}). Moreover, we characterize the minor-closed classes with bounded tree-bandwidth as those excluding an apex graph (see \cref{tbw-characterisation}), and we characterize the minor-closed classes with bounded tree-treedepth as those with bounded treewidth (see \cref{ttd-Characterisation}). The bounded tree-bandwidth result includes planar graphs and graphs of bounded Euler genus. For planar graphs $G$ we show that $\ttw(G)\leq \tpw(G) \leq \tbw(G) \leq 3$ (see \cref{ttw-tpw-planar}). This bound is best possible for $G=K_4$. For graphs of Euler genus $g$, we show that $\ttw(G) \leq \tpw(G) \leq \tbw(G) \in O(g)$, and that this bound is best possible up to the constant factor. 

Our final contributions are for random regular graphs. We show that for fixed $\eps>0$ there exists $d$ such that a random $d$-regular $n$-vertex graph $G$ asymptotically almost surely has $\ttw(G) \geq (1-\eps)n$, which implies that $\TwIntTw(G)\geq (1-\eps)n$ and $\twtw(G)\geq \sqrt{(1-\eps)n}$ (see \cref{ProductRandomRegular}). These results justify our focus on minor-closed classes. Without such an assumption, even in the setting of
bounded degree graphs, bounded tree-treewidth is impossible, the orthogonal tree-decomposition structure in \cref{OddMinorFree-TreePathDecompsOrtho} is impossible, and the bounded treewidth product structure in \cref{NoKtMinorProduct,KtMinorFreeThreeDimProduct,DirectedProduct,NoOddKtMinorProduct} is impossible.

To conclude this introduction, we give a rough intuitive description of the above structures. The most rigid structure that we consider describes a graph $G$ as being contained in $A\StrongProd B$ with $A$ and $B$ having bounded treewidth. Here each vertex of $G$ appears only once in $A\StrongProd B$. Tree-decompositions with bounded treewidth bags (that is, $\ttw$) is the most flexible structure, since each vertex of $G$ can appear in many bags of the tree-decomposition, and there is no relationship between the tree-decompositions of different bags. Orthogonal tree-decompositions are in between. They are flexible in the sense that each vertex can appear many times in either tree-decomposition, but are somewhat rigid in that the second tree-decomposition governs behaviour within bags of the first tree-decomposition. 

This paper is organised as follows. \cref{Definitions} contains definitions and basic observations. \cref{UpperBounds} establishes upper bounds on the above parameters, where the subsections of \cref{UpperBounds} are loosely organised by the method employed. \cref{LowerBounds} establishes lower bounds on the above parameters, first for $\twtw$-number, then for $\twpw$-number, then for $\TwIntTw$-number, and finally for random $d$-regular graphs. Several open problems are presented. 

%%%%%%%%%%%%%%%%%%%%%%%%%%%%%
\section{\boldmath Basics}
\label{Definitions}

\subsection{Definitions}

We consider finite simple undirected graphs $G$ with vertex-set $V(G)$ and edge-set $E(G)$. For a vertex $v\in V(G)$, let $N_G(v):=\{w\in V(G): vw\in E(G)\}$ and $N_G[v]:=N_G(v)\cup\{v\}$.

A graph $H$ is a \defn{minor} of a graph $G$ if a graph isomorphic to $H$ can be obtained from $G$ by deleting edges, deleting vertices, and contracting edges. For a graph $X$, a graph $G$ is \defn{$X$-minor-free} if $X$ is not a minor of $G$. A \defn{class} is a collection of graphs closed under isomorphism. A class $\GG$ is \defn{proper} if some graph is not in $\GG$. A class is \defn{minor-closed} if for every $G\in\GG$ every minor of $G$ is in $\GG$. 

A \defn{surface} is a 2-dimensional manifold without boundary. The class of graphs embeddable on a fixed surface is proper minor-closed (see \citep{MoharThom} for background on graph embeddings). The surface obtained from the sphere by adding $h$ handles and $c$ crosscaps has \defn{Euler genus} $2h+c$. The \defn{Euler genus} of a graph $G$ is the minimum Euler genus of a surface in which $G$ has an embedding with no edge crossing. For each $g\in\NN$ the class of graphs with Euler genus at most $g$ is proper minor-closed. 

A \defn{model} of a graph $X$ in a graph $G$ is a function $\phi$ such that:
\begin{itemize}
\item $\phi(x)$ is a non-empty subtree of $G$ for each $x\in V(X)$, 
\item $V(\phi(x))\cap V(\phi(y))=\emptyset$ for distinct $x,y\in V(X)$, 
\item for each edge $xy\in E(X)$, $\phi(xy)$ is an edge of $G$ with endpoints in $V(\phi(x))$ and $V(\phi(y))$. 
\end{itemize}
Observe that $X$ is a minor of $G$ if and only if there is a model of $X$ in $G$. A model $\phi$ of $X$ in $G$ is \defn{odd} if there is a $2$-colouring of $\bigcup_{x\in V(X)}V(\phi(x))$ such that $\phi(x)$ is properly coloured for each $x\in V(X)$, and the ends of 
 $\phi(xy)$ receive the same color for each $xy\in E(X)$. For a graph $X$, a graph is \defn{odd-$X$-minor-free} if there is no odd model of $X$ in $G$. 

A graph $G$ has an odd $K_3$-model if and only if $G$ has an odd cycle. Thus $G$ is odd-$K_3$-minor-free if and only if $G$ is bipartite. 

For a tree $T$ with $V(T)\neq\emptyset$, a \defn{$T$-decomposition} of a graph $G$ is a collection $(B_x:x \in V(T))$ such that:
\begin{itemize}
    \item $B_x\subseteq V(G)$ for each $x\in V(T)$, 
    \item for every edge ${vw \in E(G)}$, there exists a node ${x \in V(T)}$ with ${v,w \in B_x}$, and 
    \item for every vertex ${v \in V(G)}$, the set $\{ x \in V(T) : v \in B_x \}$ induces a non-empty (connected) subtree of $T$. 
\end{itemize}
A \defn{tree-decomposition} is a $T$-decomposition for any tree $T$. 
A \defn{path-decomposition} is a $P$-decomposition for any path $P$, denoted by the corresponding sequence of bags. 
The \defn{width} of a $T$-decomposition $(B_x:x \in V(T))$ is ${\max\{ |B_x| : x \in V(T) \}-1}$. The \defn{treewidth} of a graph $G$, denoted \defn{$\tw(G)$}, is the minimum width of a tree-decomposition of $G$. 
The \defn{pathwidth} of a graph $G$, denoted \defn{$\pw(G)$}, is the minimum width of a path-decomposition of $G$. 
Treewidth is the standard measure of how similar a graph is to a tree. Indeed, a connected graph has treewidth at most 1 if and only if it is a tree. It is an important parameter in structural graph theory, especially Robertson and Seymour's graph minor theory, and also in algorithmic graph theory, since many NP-complete problems are solvable in linear time on graphs with bounded treewidth. See \citep{HW17,Bodlaender98,Reed97} for surveys on treewidth. 

%We will use the following treewidth separator lemma of \citet{RS-II}.

%\begin{lem}[{\protect\citet[(2.5)]{RS-II}}] \label{Separator} For every graph $G$ there exist $X \subseteq V(G)$ with $|X| \leq \tw(G)+1$ such that every component of $G-X$ has at most $\frac{n-|X|}{2}$ vertices. \end{lem}

Consider a tree-decomposition  $(B_x:x\in V(T))$ of a graph $G$. For each edge $xy\in E(T)$, the set $B_x\cap B_y$ is called an \defn{adhesion set}. The \defn{adhesion} of $(B_x:x\in V(T))$ is $\max\{|B_x\cap B_y| : xy\in E(T)\}$. We say $(B_x:x\in V(T))$ is \defn{taut} if $B_x\cap B_y$ is a clique for each $xy\in E(T)$. For each node $x\in V(T)$, the \defn{torso} at $x$ is the graph \defn{$G\langle{B_x}\rangle$} obtained from $G[B_x]$ by adding edges so that the adhesion set $B_x\cap B_y$ is a clique for each edge $xy\in E(T)$. 
Of course, $(B_x:x\in V(T))$ is a taut tree-decomposition of $\bigcup_{x\in V(T)} G\langle{B_x}\rangle$. 

The \defn{span} of an edge $v_iv_j\in E(G)$ with respect to a vertex ordering $(v_1,\dots,v_n)$ of a graph $G$ is $|i-j|$. The \defn{bandwidth} of a graph $G$, denoted \defn{$\bw(G)$}, is the minimum integer $k$ such that there is a linear ordering of $V(G)$ in which every edge of $G$ has span at most $k$. In this case, $(\{v_1,\dots,v_{k+1}\},\{v_2,\dots,v_{k+2}\},\dots)$ is a path-decomposition of $G$, so $\pw(G)\leq\bw(G)$. 

The \defn{closure} of a rooted tree $T$ is the graph $G$ with $V(G):=V(T)$ where $vw\in E(G)$ if and only if $v$ is an ancestor of $w$ (or vice versa). The \defn{height} of a rooted tree $T$ is the maximum distance between a leaf of $T$ and the root of $T$. The \defn{treedepth} of a connected graph $G$, denoted \defn{$\td(G)$}, is the minimum height of a rooted tree $T$ plus 1, such that $G$ is a subgraph of the closure of $T$. The \defn{treedepth} of a disconnected graph, denoted \defn{$\td(G)$}, is the maximum treedepth of its connected components. It is well-known that $\pw(G)\leq \td(G)-1$ for every graph $G$. To see this, introduce one bag for each root--leaf path, ordered left-to-right according to a plane drawing of $T$, producing a path-decomposition of $G$. 

A graph $G$ is \defn{$k$-degenerate} if every subgraph of $G$ has minimum degree at most $k$. Note that a graph is $k$-degenerate if and only if $G$ has an acyclic edge orientation in which each vertex has indegree at most $k$. Also note that if a graph $G$ has a (not necessarily acyclic) edge orientation in which each vertex has indegree at most $k$, then every subgraph $H$ of $G$ satisfies $|E(H)|\leq k|V(H)|$, implying $H$ has minimum degree at most $2k$, and $G$ is $2k$-degenerate. 

\subsection{Elementary Bounds}
\label{Basics}

%
%\textbf{Grid and Hex Graphs:} Consider the $n\times n$ grid graph $G=P_n\CartProd P_n$ with vertex set $\{1,\dots,n\}^2$. So $G$ is planar with treewidth $n$ (see \citep{HW17}). Since $G$ is bipartite, 
%$\ttw(G)= \tpw(G)=\ptw(G)=\ppw(G)=1$. 
%$\ttw(G)= \tpw(G)=1$. 
%This says that planar graphs with large treewidth can have small $\tpw$. But bipartiteness is not essential here. 
%\subsection{Elementary Bounds}
%\textbf{Cycles:} Consider a cycle $C_n=\{v_1,\dots,v_n\}$ of length at least $n\geq 4$. Then $\{v_1,v_2,v_3\},\{v_1,v_3,v_4\},\{v_1,v_4,v_5\},\dots,\{v_1,v_{n-1},v_n\}$ is a path-decomposition of $G$ in which each bag is a forest. Thus $\ttw(C_n)=\ptw(C_n)=1$.

Note the following connection between $\TwIntTw$-number and products: 

\begin{lem}
\label{TwIntTw-Product}
For any graphs $H_1$ and $H_2$ and any graph $G\subsetsim H_1\StrongProd H_2$,
\begin{equation*}
\begin{aligned}
& \textup{(a)} \quad &  \TwIntTw(G) & \leq (\tw(H_1)+1)(\tw(H_2)+1),\\
& \textup{(b)} \quad & \TwIntPw(G) & \leq (\tw(H_1)+1)(\pw(H_2)+1),\\
& \textup{(c)} \quad &  \PwIntPw(G) & \leq (\pw(H_1)+1)(\pw(H_2)+1).
\end{aligned}
\end{equation*}
\end{lem}

\begin{proof} 
We first prove (a). Let $(A_x)_{x\in V(S)}$ be a tree-decomposition of $H_1$ with width $\tw(H_1)$. Let $(B_y:y\in V(T)$ be a tree-decomposition of $H_2$ with width $\tw(H_2)$. 
For each node $x\in V(S)$, let $A'_x$ be the set of vertices $v\in V(G)$ mapped to $(a,b)\in V(H_1\StrongProd H_2)$ with $a\in A_x$. Then $(A'_x)_{x\in V(S)}$ is a tree-decomposition of $G$. Similarly, for each node $y\in V(T)$, let $B'_y$ be the set of vertices $v\in V(G)$ mapped to $(a,b)\in V(H_1\StrongProd H_2)$ with $b\in B_y$. Then $(B'_y:y\in V(T))$ is a tree-decomposition of $G$. Say $v\in A'_x\cap B'_y$ where $v$ is mapped to $(a,b)\in V(H_1\StrongProd H_2)$. By construction, $a\in A_x$ and $b\in B_y$. Since $|A_x|\leq \tw(H_1)+1$ and $|B_y|\leq \tw(H_2)+1$, we have $|A'_x\cap B'_y|\leq (\tw(H_1)+1)(\tw(H_2)+1)$. Thus 
$\TwIntTw(G)\leq (\tw(H_1)+1)(\tw(H_2)+1)$. Parts (b) and (c) are proved analogously. 
\end{proof}

\cref{TwIntTw-Product} implies: 

\begin{cor}
\label{TwIntTW-twtw} 
\label{TwIntPW-twpw}
\label{PwIntPW-pwpw}
For any graph $G$, 
\begin{equation*}
\begin{aligned}
&\textup{(a)} \quad & \TwIntTw(G) & \leq  (\twtw(G)+1)^2,\\
&\textup{(b)} \quad & \TwIntPw(G) & \leq  (\twpw(G)+1)^2,\\
&\textup{(c)} \quad & \PwIntPw(G) & \leq  (\pwpw(G)+1)^2.
\end{aligned}
\end{equation*}
\end{cor}

% \begin{proof}
% Let $k:=\twtw(G)$. So $G\subsetsim A\StrongProd B$ for some graphs $A$ and $B$ with $\tw(A)\leq k$ and $\tw(B)\leq k$. Let $(A_x)_{x\in V(S)}$ be a tree-decomposition of $A$ with width at most $k$. Let $(B_y:y\in V(T))$ be a tree-decomposition of $B$ with width at most $k$. 
% For each node $x\in V(S)$, let $A'_x$ be the set of vertices $v\in V(G)$ mapped to $(a,b)\in V(A\StrongProd B)$ with $a\in A_x$. Then $(A'_x)_{x\in V(S)}$ is a tree-decomposition of $G$. Similarly, for each node $y\in V(T)$, let $B'_y$ be the set of vertices $v\in V(G)$ mapped to $(a,b)\in V(A\StrongProd B)$ with $b\in B_y$. Then $(B'_y:y\in V(T))$ is a tree-decomposition of $G$. Say $v\in A'_x\cap B'_y$ where $v$ is mapped to $(a,b)\in V(A\StrongProd B)$. By construction, $a\in A_x$ and $b\in B_y$. Since $|A_x|\leq k+1$ and $|B_y|\leq k+1$, we have $|A'_x\cap B'_y|\leq (k+1)^2$. Thus $G$ has a pair of $(k+1)^2$-orthogonal tree-decompositions. Hence
% $\TwIntTw(G)\leq (k+1)^2$. The other two inequalities are proved analogously. 
% \end{proof}

Note the following connection between $\TwIntTw$-number and tree-treewidth.

\begin{lem}
\label{ttw-otw}
For any graph $G$, 
\begin{equation*}
\begin{aligned}
&\textup{(a)} \quad & \ttw(G)+1 & \leq \TwIntTw(G),\\
&\textup{(b)} \quad & \tpw(G)+1 & \leq \TwIntPw(G).
%\ptw(G)+1 & \leq \TwIntPw(G)
\end{aligned}
\end{equation*}
\end{lem}

\begin{proof}
Let $k:=\TwIntTw(G)$. So $G$ has a pair of $k$-orthogonal tree-decompositions $(B_x:x\in V(S))$ and $(C_y:y\in V(T))$. For each $x\in V(S)$, $(B_x\cap C_y:y\in V(T))$ is a tree-decomposition of $G[B_x]$ of width $k-1$. Hence $\ttw(G)\leq k-1$.  The other inequality is proved analogously. 
\end{proof}

Note the following connection between $\twtw$-number and tree-treewidth, proved by combining
\cref{TwIntTW-twtw,ttw-otw}.

\begin{lem}
\label{ttwProduct}
For any graphs $H_1$ and $H_2$ and any graph $G\subsetsim H_1\StrongProd H_2$,
\begin{equation*}
\begin{aligned}
\textup{(a)}& \quad & \ttw(G)+1 & \leq (\tw(H_1)+1)(\tw(H_2)+1),\\
\textup{(b)}& \quad & \tpw(G)+1 & \leq (\tw(H_1)+1)(\pw(H_2)+1). 
%    \ptw(G)+1 & \leq (\pw(H_1)+1)(\tw(H_2)+1)\\
%     \ppw(G)+1 & \leq (\pw(H_1)+1)(\pw(H_2)+1).
\end{aligned}
\end{equation*}
In particular, 
%\end{lem}
%\cref{ttwProduct} (or \cref{TwIntTW-twtw,ttw-otw}) implies the next lemma. 
%
%\begin{cor}
\label{ttw-twtw}
for any graph $G$, 
\begin{equation*}
\begin{aligned}
&\textup{(c)} \quad &     \ttw(G)+1 & \leq (\twtw(G)+1)^2, \hspace*{11mm}\\
&\textup{(d)} \quad &    \tpw(G)+1 & \leq ((\tw\StrongProd\pw)(G)+1)^2. 
%   \ptw(G)+1 & \leq ( (\tw\StrongProd \pw)(G)+1)^2\\
%    \ppw(G)+1 & \leq (\pwpw(G)+1)^2.
\end{aligned}
\end{equation*}
\end{lem}

%\begin{proof} We first prove (a). We may assume that $G\subseteq H_1\StrongProd H_2$. Let $(B_x:x\in V(T))$ be a tree-decomposition of $H_1$ with width $\tw(H_1)$. Let $\widehat{B}_x:= \{(v,w)\in V(G): v\in B_x,w\in V(H_1)\}$ for each $x\in V(T)$. Observe that $(\widehat{B}_x)_{x\in V(T)}$ is a tree-decomposition of $G$. Since $G[\widehat{B}_x] \subsetsim K_{|B_x|} \StrongProd H_2$, we have $\tw(G[\widehat{B}_x]) \leq |B_x|\,(\tw(H_2)+1)-1 \leq (\tw(H_1)+1)(\tw(H_2)+1)-1$. Hence $\ttw(G)\leq (\tw(H_1)+1)(\tw(H_2)+1)-1$. Part (b) is proved analogously, using a path-decomposition of $H_2$ with width $\pw(H_2)$, and noting that $\pw(G[\widehat{B}_x]) \leq |B_x|\,(\pw(H_2)+1)-1$. 
%The other claims are proved analogously. For (c) and (d) start with a path-decomposition of $H_1$ with width $\pw(H_1)$.  For (b) and (d) use a path-decomposition of $H_2$ with width $\pw(H_2)$, and note that $\pw(G[\widehat{B}_x]) \leq |B_x|\,(\pw(H_2)+1)-1$. 
%\end{proof}

% \begin{prop}
% \label{ttwNeighbourhoodLowerBound}
%  For any graph $G$, 
%  $\ttw(G)\geq \min\{ \tw(G[N_G[v]]): v\in V(G) \}$.   
% \end{prop}

% \begin{proof}
% Consider any tree-decomposition of $G$. Let $B_x$ be a leaf bag. Let $B_y$ be the neighbouring bag. If $B_x\subseteq B_y$ then delete $x$. Now assume there exists $v\in B_x\setminus B_y$. Thus $N_G[v]\subseteq B_x$ and $\tw(G[B_x]) \geq \tw(G[N_G[v]])$. The claim follows. 
% \end{proof}

Every graph $G$ has treewidth at least its minimum degree $\delta(G)$. This fact is generalised as follows. Here a graph parameter $f$ is \defn{hereditary} if $f(G-v)\leq f(G)$ for every graph $G$ and vertex $v$ of $G$. 

\begin{prop}
\label{NeighbourhoodLowerBound}
For any hereditary parameter $f$ and any graph $G$, 
$$\treef(G)\geq \min\{ f(G[N_G[v]]): v\in V(G)\}.$$
\end{prop}

\begin{proof}
Consider a tree-decomposition $(B_x:x\in V(T))$ of $G$ with $f(G[B_x]) \leq\treef(G)$ for each node $x\in V(T)$, and subject to this condition $|V(T)|$ is minimal. When $|V(T)|=1$, $f(G[N_G[v]]) \leq f(G[B_x]) \leq \treef(G)$ for every $v \in V(G)$ since $f$ is hereditary. So we may assume $|V(T)| \geq 2$. Let $x$ be a leaf node in $T$. Let $y$ be the neighbour of $x$ in $T$. If $B_x\subseteq B_y$ then the tree-decomposition obtained by deleting $x$ contradicts the choice of tree-decomposition. Thus, there exists a vertex $v\in B_x\setminus B_y$. Hence  $N_G[v]\subseteq B_x$. Since $f$ is hereditary, 
$f(G[N_G[v]]) \leq f(G[B_x]) \leq \treef(G)$, as claimed. 
\end{proof}

\cref{NeighbourhoodLowerBound} implies that every graph $G$ with $\text{tree-}\Delta(G)\leq k$ has minimum degree $\delta(G)\leq k$, and is thus $k$-degenerate. For example, $\text{tree-}\Delta(K_{n,n})\geq \delta(K_{n,n})=n$, in contrast to 
\cref{Bipartite-ttd-ttw-tpw} which says that bipartite graphs have tree-treewidth 1. \cref{NeighbourhoodLowerBound} also gives lower bounds on tree-treewidth. For example, it shows that $\ttw(K_{n,n,n})\geq \tw(K_{1,n,n}) \geq \delta(K_{1,n,n}) = n+1$, which separates $\ttw(G)$ and $\tchi$, since $\tchi(K_{n,n,n})\leq\chi(K_{n,n,n})=3$. 

As another lower bound, for every clique $C$ of a graph $G$, in every tree-decomposition of $G$ 
there is a bag that contains $C$ (see \citep[Corollary~12.3.5]{Diestel5}). Since $\tw(K_n)=n-1$, we have $\ttw(G)\geq\omega(G)-1$.

%%%%%%%%%%%%%%%%%%%%%%%%%%%%%%%%%%%%%
\section{\boldmath Upper Bounds}
\label{UpperBounds}

This section shows that $\twtw$-number  is bounded for graphs excluding a fixed minor or a fixed odd minor, plus various related results.

%%%%%%%%%%%%%%%%%%%%%%%%%%%%%%%
\subsection{\boldmath Joins and Low-Treewidth Colourings}
\label{Joins}

The results for bipartite graphs in \cref{Introduction} generalise as follows. For graphs $A$ and $B$, let \defn{$A+B$} be the \defn{complete join} of $A$ and $B$, which is obtained from disjoint copies of $A$ and $B$ by adding all edges joining $A$ and $B$. Note that for any graph $G$ and $a\in\NN$, \[\tw(G+K_a)=\tw(G)+a.\]
The next two lemmas connect  strong products and complete joins. 

\begin{lem}
\label{MoveApexVertices}
For all graphs $A$ and $B$, and for any $p,q\in\NN$, 
$$
(A\StrongProd B)+K_{pq} \subsetsim (A+K_p) \StrongProd (B+K_q).$$
%Moreover, the subgraph of $(A+K_p) \StrongProd (B+K_q)$ induced by the image of $(A\StrongProd B)+K_{pq}$ is isomorphic to $(A\StrongProd B)+K_{pq}$.
\end{lem}

\begin{proof}
Let $A':=A+K_p$, where $V(K_p)=\{a_1,\dots,a_p\}$.
Let $B':=B+K_q$, where $V(K_q)=\{b_1,\dots,b_q\}$.
Say $K_{pq}:= \{r_{i,j}:i\in\{1,\dots,p\},j\in\{1,\dots,q\}\}$. 
Map each vertex $(v,w)$ in $A\StrongProd B$ to $(v,w)$ in $A'\StrongProd B'$. Map each vertex $r_{i,j}$ to the vertex $(a_i,b_j)$ in $A'\StrongProd B'$. Since $a_i$ is dominant in $A'$ and $b_j$ is dominant in $B'$, $(a_i,b_j)$ is dominant in $A'\StrongProd B'$. 
Hence $(A\StrongProd B)+K_{pq} \subsetsim A'\StrongProd B'$.
%The `moreover' claim is immediate from the construction. 
\end{proof}

\begin{lem}
\label{JoinProduct}
For all graphs $A$ and $B$, and for all $p,q\in\NN$, 
$$A+B+K_{pq} \subsetsim (A+K_p)\StrongProd (B+K_q).$$
\end{lem}
\begin{proof}
Let $A':=A+K_p$, where $V(K_p)=\{a_1,\dots,a_p\}$.
Let $B':=B+K_q$, where $V(K_q)=\{b_1,\dots,b_q\}$.
Say $V(K_{pq})=\{r_{i,j}: i\in\{1,\dots,p\},j\in\{1,\dots,q\}\}$. 
Map each vertex $x\in V(A)$ to $(x,b_1)\in V(A'\StrongProd B')$. 
Map each vertex $y\in V(B)$ to $(a_1,y) \in V(A' \StrongProd B')$. 
Map each vertex $r_{i,j}$ to $(a_i,b_j) \in V(A'\StrongProd B')$. 
Observe that this mapping is an injective homomorphism from 
$A+B+K_{pq}$ to $A'\StrongProd B'$. The result follows. 
\end{proof}

A \defn{partition} of a graph $G$ is a partition of $V(G)$ into sets, called the \defn{parts} of $\PP$. Let  $\mathdefn{G^+}:=G+K_1$ be the graph obtained from a graph $G$ by adding one new dominant vertex. \cref{JoinProduct} with $p=q=1$ implies the following:

\begin{cor}
\label{JoinApex}
For any partition $\{V_1,V_2\}$ of a graph $G$, 
\[G \subsetsim G^+ \subsetsim G[V_1]^+ \StrongProd G[V_2]^+.\] 
\end{cor}

The next results follow from \cref{JoinApex} since $\tw(G^+)=\tw(G)+1$ and $\pw(G^+)=\pw(G)+1$.

\begin{cor}
\label{JoinTreewidth}
If $\{V_1,V_2\}$ is a partition of a graph $G$ such that $\tw(G[V_i])\leq c$ for each $i\in\{1,2\}$, then $G\subseteq G^+ \subsetsim H_1\StrongProd H_2$ for some graph $H_i$ with $\tw(H_i)\leq c+1$ for each $i\in\{1,2\}$.
\end{cor}

\begin{cor}
\label{JoinPathwidth}
If $\{V_1,V_2\}$ is a partition of a graph $G$ such that $\pw(G[V_i])\leq c$ for each $i\in\{1,2\}$, then $G\subseteq G^+ \subsetsim H_1\StrongProd H_2$ for some graph $H_i$ with $\pw(H_i)\leq c+1$ for each $i\in\{1,2\}$.
\end{cor}

% \begin{proof}
% Let $H_i$ be the graph obtained from $G[V_i]$ by adding one dominant vertex $v_i$. Thus $\tw(H_i)\leq c+1$. Map each vertex $x\in V_1$ to $(x,v_2)$ in $H_1\StrongProd H_2$.  Map each vertex $y\in V_2$ to $(v_1,y)$ in $H_1\StrongProd H_2$. By construction, $G \subsetsim H_1\StrongProd H_2$. 
% \end{proof}

\citet{DDOSRSV04} showed that for any $t\in\NN$ there exists $c\in\NN$ such that every $K_t$-minor-free graph has a partition $\{V_1,V_2\}$ with $\tw(G[V_i])\leq c$ for each $i\in\{1,2\}$ (see \citep{DJMMUW20} for a proof using graph product structure theory). Combining this result with \cref{JoinTreewidth} shows that $\twtw(G)\leq c+1$, which proves \cref{NoKtMinorProduct}. \citet{DHK-SODA10} extended this result of \citet{DDOSRSV04} for odd-$K_t$-minor-free graphs, which with \cref{JoinTreewidth} implies \cref{NoOddKtMinorProduct}.

The same approach has other interesting applications. 

\citet{DOSV98} proved that if $k=k_1 +k_2 +1$ for some $k_1,k_2\in\NN_0$, then every graph of treewidth $k$ has a partition into two induced subgraphs with treewidth at most $k_1$ and $k_2$ respectively. In particular, every graph of treewidth $k$ has a partition into two induced subgraphs each with treewidth at most $\ceil{(k-1)/2}$. \cref{JoinTreewidth} thus implies:

\begin{prop}
\label{TreewidthProductTreewidth}
Every graph $G$ with treewidth at most $k$ is contained in $H_1\StrongProd H_2$ for some graphs $H_1$ and $H_2$ with $\tw(H_i)\leq \ceil{(k+1)/2}$.
\end{prop}

In \cref{LowerBoundTreewidth} we show that the bound on $\tw(H_i)$ in \cref{TreewidthProductTreewidth} is best possible.

As illustrated  in \cref{WindmillsFlowers}(a), a \defn{windmill} is a graph $G$ such that $\Delta(G-a)\leq 1$ for some $a\in V(G)$. Every windmill $G$ has $\td(G)\leq 3$ and $\tw(G)\leq \pw(G)\leq 2$.

\begin{figure}[ht]
    \centering
    (a)\includegraphics[scale=0.8]{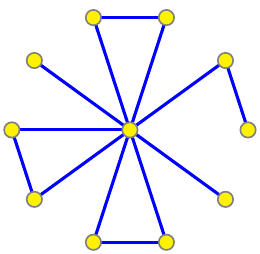}
    \qquad
    (b)\includegraphics[scale=0.8]{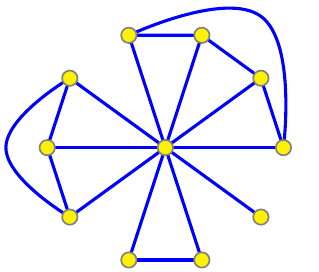}
    \caption{\label{WindmillsFlowers}
(a) windmill, (b) flower}
\end{figure}

\begin{prop}
\label{Delta3}
Every graph $G$ with maximum degree at most 3 is contained in $H_1\StrongProd H_2$ for some windmills $H_1$ and $H_2$. In particular, $\twtw(G)\leq \pwpw(G)\leq 2$ and $\tdtd(G)\leq 3$. 
\end{prop}

\begin{proof}
Let $\{V_1,V_2\}$ be a partition of $G$ maximising the number of edges between $V_1$ and $V_2$. If $v\in V_1$ has at least two neighbours in $V_1$, then $v$ has at most one neighbour in $V_2$, implying that $\{V_1\setminus\{v\},V_2\cup\{v\}\}$ contradicts the choice of partition. Hence $G[V_1]$, and by symmetry $G[V_2]$, has maximum degree at most $1$. Thus $G[V_i]^+$ is a windmill. By \cref{JoinApex}, $G\subseteq G^+ \subsetsim G[V_1]^+ \StrongProd G[V_2]^+$. 
\end{proof}

As illustrated  in \cref{WindmillsFlowers}(b), a  \defn{flower} is a graph $G$ such that $\Delta(G-a)\leq 2$ for some $a\in V(G)$.  Every flower $G$ has $\tw(G)\leq \pw(G)\leq 3$.

\begin{prop}
\label{Delta4}
Every graph $G$ with maximum degree at most 4 is contained in $H_1\StrongProd H_2$ for some flowers $H_1$ and $H_2$. In particular, $\twtw(G)\leq \pwpw(G)\leq 3$.
\end{prop}

\begin{proof}
Let $\{V_1,V_2\}$ be a partition of $G$ maximising the number of edges between $V_1$ and $V_2$. If $v\in V_1$ has at least three neighbours in $V_1$, then $v$ has at most one neighbour in $V_2$, implying that $\{V_1\setminus\{v\},V_2\cup\{v\}\}$ contradicts the choice of partition. Hence $G[V_1]$, and by symmetry $G[V_2]$, has maximum degree at most $2$. Thus $G[V_i]^+$ is a flower. By \cref{JoinApex}, $G\subseteq G^+ \subsetsim G[V_1]^+ \StrongProd G[V_2]^+$. 
\end{proof}

%\begin{prop}
%\label{Delta5}
%There exists $c\in\NN$ such that every graph $G$ with maximum degree at most 5 is contained in $S\StrongProd H$, where $S$ is a star and $H$ is a graph with $\td(H)\leq c$. In particular, $$\twtw(G)+1 \leq \pwpw(G)+1 \leq (\td\StrongProd\td)(G)\leq c.$$
%\end{prop}

%\begin{proof}
%\citet{HST03} showed that for some constant $c$, every graph $G$ with maximum degree at most 5 is (improperly) 2-colourable such that each monochromatic component has at most $c$ vertices, implying $G\subsetsim H\StrongProd K_c$ where $H$ is bipartite. By \cref{JoinTreewidth}, $G\subsetsim K_{1,n} \StrongProd K_{1,n} \StrongProd K_c \subsetsim S \StrongProd H$ where $S:=K_{1,n}$ is a star, and $H := K_{1,n} \StrongProd K_c$ which has treedepth at most $2c$.
%\end{proof}

\begin{prop}
\label{Delta5}
If $\GG$ is the class of graphs with maximum degree at most 5, then 
$$\twtwa(\GG) \leq \pwpwa(\GG) \leq 1,$$
and there exists $c\in\NN$ such that for every $G\in\GG$, 
$$\twtw(G) \leq \pwpw(G)\leq c.$$
\end{prop}

\begin{proof}
\citet{HST03} showed that for some $c\in\NN$, every graph $G$ with maximum degree at most 5 is (improperly) 2-colourable such that each monochromatic component has at most $c$ vertices, implying $G\subsetsim H\StrongProd K_c$ where $H$ is bipartite. By \cref{JoinApex}, $G\subsetsim K_{1,n} \StrongProd K_{1,n} \StrongProd K_c$. 
Then $\twtwa(\GG) \leq \pwpwa(\GG) \leq 1$ since $\tw(K_{1.n})\leq\pw(K_{1,n})\leq 1$.
For the final claim, note that $G\subsetsim K_{1,n} \StrongProd K_{1,n} \StrongProd K_c \subsetsim S \StrongProd H$ where $S:=K_{1,n}$ is a star which has pathwidth 1, and $H := K_{1,n} \StrongProd K_c$ which pathwidth at most $2c-1$. 
\end{proof}

\subsection{Using Underlying Treewidth}
\label{UsingUTW}

This section proves our 3-term product structure results for $X$-minor-free graphs, as introduced in \cref{Introduction}. The following definition, due to \citet{UTW}, is a key tool. The \defn{underlying treewidth} of a graph class $\mathcal{G}$, denoted by \defn{$\utw(\GG)$}, is the minimum integer such that, for some function $f:\mathbb{N}\to\mathbb{N}$, for every graph $G\in \GG$ there is a graph $H$ of treewidth at most $\utw(\GG)$ such that $G\subsetsim H \StrongProd K_{f(\tw(G))}$. Here, $f$ is called the \defn{treewidth binding function}. \citet{UTW} showed that the underlying treewidth of the class of planar graphs equals 3, and the same holds for any fixed surface. More generally, \citet{UTW} showed that $\utw(\GG_{K_t})=t-2$ and $\utw(\GG_{K_{s,t}})=s$ (for $t\geq\max\{s,3\})$. In these results, the treewidth binding function is quadratic. \citet{ISW} reproved these results with a linear treewidth binding function (see \cref{JJstMain} below).  \citet{DHHJLMMRW} showed that $\utw(\GG_X)$ and $\td(X)$ are tied:
 \begin{equation}
 \label{utw-td}
      \td(X)-2\leq \utw(\GG_X) \leq 2^{\td(X)+1}-4.
 \end{equation}
Moreover, in the upper bound, the treewidth binding function is linear. \citet{DHHJLMMRW} used \cref{utw-td} to prove a more precise version of \cref{ApexMinorFreeProduct} for $X$-minor-free graphs where $X$ is an apex graph\footnote{Let $X$ be an apex graph. Let $g(X)$ be the minimum integer such that there exists $c\in\NN$ such that every $X$-minor-free graph $G$ is contained in $H\boxtimes P\boxtimes K_c$, where $H$ is a graph with $\tw(H)\leq g(X)$ and $P$ is a path. \citet{DJMMUW20} proved that $g(X)$ is well-defined. \citet{DHHJLMMRW}  showed that $g(X)$ is tied to $\td(X)$; in particular,  $\td(X)-2\leq g(X) \leq 2^{\td(X)+1}-1$.}. We now use \cref{utw-td} to prove the following upper bound on $\twtwa(\GG_X)$ where $X$ is any graph. 
  
\begin{thm}
\label{XMinorFree-ThreeDimProduct-Treedepth}
For any graph $X$, 
\[\twtwa(\GG_X)
\leq \utw(\GG_X)+1 \leq 2^{\td(X)+1}-3.\]
\end{thm}

\begin{proof}
Let $G$ be any $X$-minor-free graph. By the above-mentioned result of \citet{DDOSRSV04}, there exists $k=k(X)\in\NN$ such that $G$ has a partition $\{V_1,V_2\}$ with $\tw(G[V_1])\leq k$ and $\tw(G[V_2])\leq k$. Of course, $G[V_i]$ is $X$-minor-free. By the definition of underlying treewidth, $G[V_i]\subsetsim H_i\boxtimes K_c$, where $\tw(H_i)\leq\utw(\GG_X)$ and $c=c(X)\in\NN$. Now,
\[G\subsetsim G[V_1]+G[V_2] 
\subsetsim (H_1\boxtimes K_c)+(H_2\boxtimes K_c).\]
Recall $H^+_i:=H_i+K_1$. By \cref{utw-td},
\[\tw(H^+_i)\leq\tw(H_i)+1\leq \utw(\GG_X)+1
\leq 2^{\td(X)+1}-3.\]
By \cref{JoinProduct} with $A=H_1\boxtimes K_c$ and $B=H_2\boxtimes K_c$ and $p=q=1$, 
\begin{align*}
G & \subsetsim 
( (H_1\boxtimes K_c)+K_1) \StrongProd
( (H_2\boxtimes K_c)+K_2) \\
& \subsetsim 
( H^+_1 \boxtimes K_c ) \StrongProd ( H^+_2\boxtimes K_c )\\
& \subsetsim 
H^+_1 \StrongProd H^+_2\boxtimes K_{c^2} ,
\end{align*}
as desired. 
\end{proof}

We now set out to improve the bounds in 
\cref{XMinorFree-ThreeDimProduct-Treedepth} for particular excluded minors. The following definitions  provide a useful way to think about graph products. For a partition $\PP$ of a graph $G$, let \defn{$G / \PP$} be the graph obtained from $G$ by identifying the vertices in each non-empty part of $\PP$ to a single vertex; that is, $V(G/\PP)$ is the set of non-empty parts in $\PP$, where distinct non-empty parts $P_1,P_2 \in \PP$ are adjacent in $G / \PP$ if and only if  there exist $v_1 \in P_1$ and $v_2 \in P_2$ such that $v_1v_2 \in E(G)$. The next observation characterises when a graph is contained in $H_1 \StrongProd H_2 \StrongProd K_c$.
	
\begin{obs}
\label{PartitionProduct}
For any graphs $H_1,H_2$ and any $c\in\NN$, a graph $G$ is contained $H_1\StrongProd H_2 \StrongProd K_c$ if and only if $G$ has partitions $\PP_1$ and $\PP_2$ such that $G/\PP_i \subsetsim H_i$ for each $i\in\{1,2\}$, and $|A_1\cap A_2|\leq c$ for each $A_1\in\PP_1$ and $A_2\in\PP_2$. 
\end{obs}

\begin{proof}
$(\Rightarrow)$ Assume $G$ is contained $H_1\StrongProd H_2 \StrongProd K_c$. That is, there is an isomorphism $\phi$ from $G$ to a subgraph of $H_1\StrongProd H_2 \StrongProd K_c$. For each vertex $x\in V(H_1)$, let $A_x:=\{v\in V(G): \phi(v)=(x,y,z), y\in V(H_2), z\in V(K_c)\}$. Similarly, for each vertex $y\in V(H_2)$, let $B_y:=\{v\in V(G): \phi(v)=(x,y,z), x\in V(H_1), z\in V(K_c)\}$. Let $\PP_1:=\{A_x:x\in V(H_1)\}$ and $\PP_2:=\{B_y:y\in V(H_2)\}$, which are partitions of $G$. By construction, $G/\PP_i\subsetsim H_i$ for $i\in\{1,2\}$. For $A_x\in \PP_1$ and $B_y\in \PP_2$, if $v\in A_x\cap B_y$ then $\phi(v)=(x,y,z)$ where $z\in V(K_c)$. Thus $|A_x\cap B_y|\leq c$. 

$(\Leftarrow)$ Assume $G$ has partitions $\PP_1$ and $\PP_2$ such that $G/\PP_i\subsetsim H_i$ for each $i\in\{1,2\}$, and $|A_1\cap A_2|\leq c$ for each $A_1\in\PP_1$ and $A_2\in\PP_2$. Let $\phi_i$ be an isomorphism from $G/\PP_i$ to a subgraph of $H_1$. For each $A_1\in\PP_1$ and $A_2\in\PP_2$, enumerate the at most $c$ vertices in $A_1\cap A_2$, and map the $i$-th such vertex to $(\phi_1(A_1),\phi_2(A_2),i)$. This defines an isomorphism from $G$ to a subgraph of $H_1\StrongProd H_2 \StrongProd K_c$, as desired. 
\end{proof}

The following definition allows $K_t$-minor-free and $K_{s,t}$-minor-free graphs to be considered simultaneously. Let \defn{$\JJ_{s, t}$} be the class of graphs $G$ whose vertex-set has a partition $\{A, B\}$, where $|A| = s$ and $|B| = t$, $A$ is a clique, every vertex in $A$ is adjacent to every vertex in $B$, and $G[B]$ is connected. A graph is \defn{$\JJ_{s,t}$-minor-free} if it contains no graph in $\JJ_{s,t}$ as a minor. 

\citet{ISW} introduced the following definition, which measures the `complexity' of a set of vertices with respect to a tree-decomposition $\DD=(B_x:x\in V(T))$ of a graph $G$. Define the \defn{$\DD$-width} of a set $S\subseteq V(G)$ to be the minimum number of bags of $\DD$ whose union contains $S$. The \defn{$\DD$-width} of a partition $\PP$ of $G$ is the maximum $\DD$-width of one of the parts of $\PP$. \citet[Theorem~12]{ISW} proved the following result, which describes how to convert a tree-decomposition of a graph excluding a minor into a partition. 
	
\begin{thm}[\citep{ISW}]
\label{JJstMain}
For any integers $s, t \geq 2$, 
for any $\JJ_{s, t}$-minor-free graph $G$, for any tree-decomposition $\DD$ of $G$, $G$ has a partition $\PP$ of $\DD$-width at most $t-1$, where $\tw(G/\PP)\leq s$.
\end{thm}

Note that \cref{JJstMain} implies that the class of $\JJ_{s, t}$-minor-free graphs has underlying treewidth at most $s$. We now use \cref{JJstMain} to prove 3-term product structure results.

\begin{lem}
\label{JstConstructThreeDimProduct}
For any $s,t,k\in\NN$, if $G$ is a $\JJ_{s, t}$-minor-free graph and $\TwIntTw(G)\leq k$, then $G$ is contained in $H_1\StrongProd H_2 \StrongProd K_{(t-1)^2k}$ where $\tw(H_1)\leq s$ and $\tw(H_2)\leq s$. 
\end{lem}

\begin{proof}
Since $\TwIntTw(G)\leq k$, there are tree-decompositions $\DD_1=(B^1_x:x\in V(T_1))$ and $\DD_2=(B^2_y:y\in V(T_2))$ of $G$, such that $|B^1_x\cap B^2_y|\leq k$ for each $x\in V(T_1)$ and $y\in V(T_2)$. For each $i\in\{1,2\}$, by \cref{JJstMain} applied to $\DD_i$, $G$ has a partition $\PP_i$ of $\DD_i$-width at most $t-1$, where $\tw(G /\PP_i)\leq s$. Let $A_i\in \PP_i$ for each $i\in\{1,2\}$. Since $\PP_i$ has $\DD_i$-width at most $t-1$, there is a set $S_i$ of at most $t-1$ nodes of $T_i$, such that $A_i\subseteq \bigcup\{B^i_x:x\in S_i\}$. For each $x\in S_1$ and $y\in S_2$, we have $|B^1_x\cap B^2_y|\leq k$. Thus $|A_1\cap A_2|\leq (t-1)^2k$. The result follows from \cref{PartitionProduct}. 
\end{proof}

\begin{cor}
\label{twtw-TwIntTw-ExcludedMinor}
For any graph $X$, for every $X$-minor-free graph $G$, 
    \[\twtw(G) \leq (|V(X)|-1)\, \TwIntTw(G).\]
\end{cor}

\begin{proof}
If $|V(X)|\leq 4$ then $\twtw(G)\leq \tw(G)\leq|V(X)|-2$, and we are done. Now assume that $|V(X)|\geq 5$. Let $s:=|V(X)|-2$ and $t:=2$. So  $\JJ_{s,t}=\{K_{|V(X)|}\}$ and every $X$-minor-free graph $G$ is $\JJ_{s,t}$-minor-free. Let $k:=\TwIntTw(G)$. By \cref{JstConstructThreeDimProduct}, $G$ is contained in $H_1\StrongProd (H_2\boxtimes K_k)$ where $\tw(H_1)\leq s$ and $\tw(H_2)\leq s$, implying $\tw(H_2 \boxtimes K_k) \leq (s+1)k-1 < (|V(X)|-1)k $. 
\end{proof}

\cref{twtw-TwIntTw-ExcludedMinor} is in contrast to the result mentioned in the introduction, saying that $\TwIntTw$-number and $\twtw$-number are separated for Burling graphs. There is no contradiction, since Burling graphs contain arbitrarily large complete graph minors.

As mentioned in the introduction, \citet{DJMNW18} showed that $\TwIntPw$-number is bounded for $K_t$-minor-free graphs and thus in any proper minor-closed class (see \cref{OddMinorFree-TreePathDecompsOrtho} for the extension to odd minors). Hence
$\TwIntTw$-number is also bounded, and \cref{JstConstructThreeDimProduct}  implies:

\begin{thm}
\label{JstMinorFreeThreeDimProduct}
For any $s,t\in\NN$ there exists $c\in\NN$, such that every $\JJ_{s, t}$-minor-free graph is contained in $H_1\StrongProd H_2 \StrongProd K_c$ where $\tw(H_1)\leq s$ and $\tw(H_2)\leq s$. That is, if $\GG_{\JJ_{s,t}}$ is the class of $\JJ_{s,t}$-minor-free graphs, then $\twtwa(\GG_{\JJ_{s,t}})\leq s$. 
\end{thm}

Since $\JJ_{t-2,2}=\{K_t\}$, \cref{JstMinorFreeThreeDimProduct} implies the next result stated in the introduction. 

\KtMinorFreeThreeDimProduct*

Recall that $\tau(X)$ is the vertex-cover number of a graph $X$. Observe that $X$ is contained in every graph in $\JJ_{\tau(X),|V(X)|-\tau(X)}$. Thus \cref{JstMinorFreeThreeDimProduct} implies:
\begin{equation}
\label{XMinorFreeThreeDimProduct}
\twtwa(\GG_X)\leq\tau(X).
\end{equation}
For example, 
\begin{equation}
\label{KstMinorFreeThreeDimProduct}
\twtwa(\GG_{K_{s,t}})\leq \min\{s,t\}.
\end{equation}

\subsection{Sparse Directed Products}
\label{SparseDirectedProducts}

Here we prove \cref{DirectedProduct} regarding sparse directed products. The starting point is the following `Graph Minor Product Structure Theorem' by \citet{DJMMUW20}. 

\begin{thm}[\citep{DJMMUW20}]  \label{ProductStructureXMinorFree} 
For any graph $X$ there exists $c\in\NN$ and $a\in\NN_0$ such that every $X$-minor-free graph $G$ has a tree-decomposition $(B_x:x\in V(T))$ such that for every $x\in V(T)$,  
$$G\langle{B_x}\rangle \subsetsim (H\StrongProd P)+K_a$$ 
for some graph $H$ with $\tw(H)\leq c$ and for some path $P$. 
\end{thm}

Note that in \cref{ProductStructureXMinorFree}, $a\leq\max\{ a(X)-1,0\}$, although we will not use this property. We first prove a similar result without the $+K_a$ term.

\begin{lem}
\label{NoXMinor-DirectedProductStructureNoApexVertices}
For every graph $X$ there exist $c,d,h\in\NN$ such that every $X$-minor-free graph $G$ has a tree-decomposition $(B_x:x\in V(T))$ with adhesion at most $h$, such that for every $x\in V(T)$, the torso $G\langle{B_x}\rangle$ is contained in $J\DirectedStrongProd F$, for some digraph $J$ with $\tw(\undirected{J})\leq c$ and $\indeg(J)\leq d$, and for some digraph $F$ with $\pw(\undirected{F})\leq 2$ and $\indeg(F)\leq 3$. 
\end{lem}

\begin{proof}
Consider the tree-decomposition of $G$ from \cref{ProductStructureXMinorFree}. 
Each torso $G\langle B_x\rangle$ is contained in $(H\StrongProd P)+K_a$ for some graph $H$ with $\tw(H)\leq w$ and for some path $P$. 
The adhesion is at most $\omega((H\StrongProd P)+K_a)\leq 2w+a+2$. 

Consider a torso $G\langle B_x\rangle$. 
Let $A_x$ be a set of vertices in $B_x$, such that $|A_x|\leq a$ and $G\langle B_x\rangle-A_x$ is contained in $H\StrongProd P$. By \cref{MoveApexVertices}, $G\langle B_x\rangle$ is contained in $\undirected{J}\StrongProd \undirected{F}$, where $\undirected{J}:=H+K_a$ and $\undirected{F}:=P+K_1$ is a fan. Note that $\tw(\undirected{J})\leq w+a$ and $\pw(\undirected{F})\leq 2$. 
Let $x_1,\dots,x_a$ be the vertices in $\undirected{J}-V(H)$, and let $y$ be the vertex in $\undirected{F}-V(P)$, as illustrated in \cref{ConstructF}.

\begin{figure}[ht]
\centering
\includegraphics{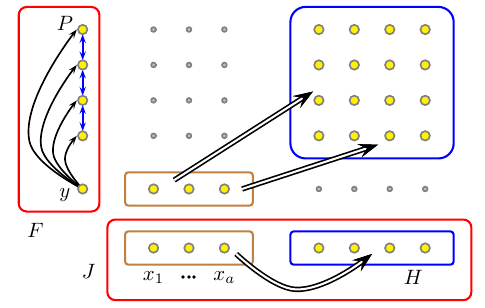}
\caption{\label{ConstructF}
Construction of $J\DirectedStrongProd F$ in the proof of \cref{NoXMinor-DirectedProductStructureNoApexVertices}.}
\end{figure}

Fix an acyclic orientation of $H$ with indegree at most $w$, which exists since $\tw(H)\leq w$, implying $H$ is $w$-degenerate. Orient each edge $x_ix_j$ of $\undirected{J}$ arbitrarily. 
Orient each edge $x_iz\in E(\undirected{J})$ with $z\in V(H)$ from $x_i$ to $z$. It defines an orientation of $\undirected{J}$ and defines a digraph $J$ with indegree at most $w+a$. Replace each undirected edge $uv$ of $P$ by two directed edges $uv$ and $vu$. Orient every other edge of $F$ away from $y$. So $\indeg(F)\leq 3$. Because of the bidirected edges in $P$,  
$\undirected{H\DirectedStrongProd P}=H\StrongProd P$. Each vertex $(x_i,y)$ is adjacent to every other vertex in $J\DirectedStrongProd F$. Hence $G\langle{B_x}\rangle$ is contained in $J\DirectedStrongProd F$, where $A_x$ is injectively mapped to $(x_1,y),\dots,(x_a,y)$. The result follows with $c=w+a$ and $d=w+a$ and $h=2w+a+2$. 
\end{proof}

\begin{lem}
\label{DirectedProductStructure}
Let $G$ be a graph that has a taut tree-decomposition $(B_t: t\in V(T))$ with adhesion $h$ such that for each node $t\in V(T)$, the subgraph $G[B_t]$ is contained in $J_1\DirectedStrongProd J_2$, for some digraphs $J_1$ and $J_2$ with $\tw(\undirected{J_i})\leq c_i$ and $\indeg(J_i)\leq d_i$ for each $i\in\{1,2\}$. 
%Let $G$ be a graph that has a tree-decomposition $(B_x:x\in V(T))$ with adhesion $h$ such that for every $x\in V(T)$, the torso $G\langle{B_x}\rangle$ is contained in $J_1\DirectedStrongProd J_2$, for some digraphs $J_1$ and $J_2$ with $\tw(\undirected{J_i})\leq c_i$ and $\indeg(J_i)\leq d_i$ for each $i\in\{1,2\}$. 
Then $G$ is contained in $D_1\DirectedStrongProd D_2$, for some digraphs $D_1$ and $D_2$ with $\tw(\undirected{D_i})\leq c_i+h$ and $\indeg(D_i)\leq d_i+h$ for each $i\in\{1,2\}$.
\end{lem}

\begin{proof}
We proceed by induction on $|V(T)|$. The base case $|V(T)|=1$ is trivial. Now assume that $G$ is a graph with a taut tree-decomposition $(B_t: t\in V(T))$ satisfying the assumptions of the lemma, where $|V(T)|\geq 2$. 

Let $x$ be a leaf of $T$, and let $y$ be the neighbour of $x$ in $T$. Let $K:= B_x\cap B_y$, which is a clique of size at most $h$ since the tree-decomposition is taut. By assumption, $G[B_x]$ is contained in $J_1\DirectedStrongProd J_2$ for some digraphs $J_1$ and $J_2$ with $\tw(\undirected{J_i})\leq c_i$ and $\indeg(J_i)\leq d_i$. 

Let $G':=G-(B_x\setminus B_y)$. So $(B_t: t\in V(T-x))$ is a taut tree-decomposition of $G'$ indexed by $T-x$ satisfying the assumptions of the lemma. By induction, $G'$ is contained in $D'_1\DirectedStrongProd D'_2$, for some digraphs $D'_1$ and $D'_2$ with $\tw(\undirected{D'_i})\leq c_i+h$ and $\indeg(D'_2)\leq d_i+h$.

As illustrated in \cref{ConstructH}, for $i\in\{1,2\}$, let $K_i$ be the projection of the image of $K$ in $D'_i$. Let $k_i:=|K_i|$. So $k_i\leq|K|\leq h$, and $K_i$ is a clique in $D'_i$. 
Let $D_i$ be obtained from the disjoint union of $D'_i$ and $J_i$ by adding a directed edge from each vertex in $K_i$ to each vertex in $J_i$. To obtain a tree-decomposition of $\undirected{D_i}$, start with disjoint tree-decompositions of $\undirected{D'_i}$ and $\undirected{J_i}$, add $K_i$ to every bag of the tree-decomposition of $\undirected{J_i}$, and add an edge to the indexing tree between any bag of $\undirected{J_i}$ and a bag of $\undirected{D'_i}$ that contains $K_i$ (which exists since $K_i$ is a clique). Thus $\tw(\undirected{D_i}) \leq \max\{\tw(\undirected{D'_i}), \tw(\undirected{J_i})+|K_i|\} \leq c_i+h$. 

\begin{figure}
\centering
\includegraphics{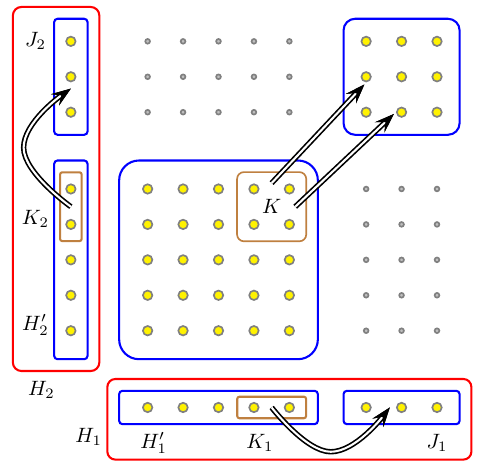}
\caption{\label{ConstructH}
Construction of $H_1 \DirectedStrongProd H_2$ in the proof of \cref{DirectedProductStructure}.}
\end{figure}

Map each vertex of $G'$ to its existing image in $D'_1\DirectedStrongProd D'_2$. Map each vertex of $B_x\setminus B_y$ to its existing image in $J_1\DirectedStrongProd J_2$. Since $K_i$ is complete to $J_i$, $K$ is complete to $J_1\DirectedStrongProd J_2$ in $D_1\DirectedStrongProd D_2$. Hence, this mapping shows that $G$ is contained in $D_1\DirectedStrongProd D_2$. 
\end{proof}

We now prove the main result of this subsection. 

\DirectedProduct*

\begin{proof}
Let $G$ be a $K_t$-minor-free graph. 
By \cref{NoXMinor-DirectedProductStructureNoApexVertices}, for some $c,d,h \in \mathbb{N}$ only depending on $t$, $G$ has a tree-decomposition $\TT=(B_x:x\in V(T))$ with adhesion at most $h$, such that for every $x\in V(T)$, the torso $G\langle{B_x}\rangle$ is contained in $J\DirectedStrongProd F$, for some digraph $J$ with $\tw(\undirected{J})\leq c$ and $\indeg(J)\leq d$, and for some digraph $F$ with $\pw(\undirected{F})\leq 2$ and $\indeg(F)\leq 3$. Let $G':=\bigcup_{x\in V(T)}G\langle{B_x}\rangle$. So $\TT$ is a taut tree-decomposition of $G'$. By \cref{DirectedProductStructure}, $G'$ and thus $G$ is contained in $D_1\DirectedStrongProd D_2$, for some digraphs $D_1$ and $D_2$ with $\tw(\undirected{D_i})\leq c+h$ and $\indeg(D_i)\leq d+h$ for each $i\in\{1,2\}$. 
\end{proof}

We emphasize that in \cref{DirectedProduct}, $D_1\DirectedStrongProd D_2$ has bounded indegree and thus has bounded degeneracy.

We finish this subsection with an open problem.

\begin{open} 
Can \cref{KtMinorFreeThreeDimProduct,DirectedProduct} be combined to show that every $K_t$-minor-free graph $G$ is contained in $\undirected{D_1\DirectedStrongProd D_2} \StrongProd K_{c(t)}$, for some digraphs $D_1$ and $D_2$ with $\tw(\undirected{D_i})\leq t-2$ and $\indeg(D_i)\leq d(t)$ for each $i\in\{1,2\}$?
\end{open}

%%%%%%%%%%%%%%%%%%%%%%%%%%%%%%%%%%%%%
\subsection{Excluding an Odd Minor}
\label{NoOddktMinor}

This section first proves a product structure theorem for graphs excluding an odd minor (see \cref{NoOddMinor-TreeDecompProduct}), which leads to the proof of \cref{OddMinorFree-TreePathDecompsOrtho} showing that $\TwIntPw$-number is bounded for graphs excluding an odd minor. We use the following structure theorem, which is a special case of a folklore refinement of a result by \citet{GRRSV09} explicitly proved in \cite{Liu24}. See \citep{DHK-SODA10} for a similar structure theorem.

\begin{thm}[\citep{Liu24}]
\label{OddMinorFreeTorsoStructure}
For any graph $X$, there exist $r,a\in\NN$ such that every odd-$X$-minor-free graph $G$ has a tree-decomposition $(B_x:x\in V(T))$ such that for every $x\in V(T)$, either: 
\begin{enumerate}[(a)]
\item $G\langle{B_x}\rangle$ is $K_r$-minor-free, or 
\item $G\langle{B_x}\rangle \subsetsim H+K_a$ for some bipartite graph $H$.
\end{enumerate}
\end{thm}

\begin{lem}
\label{NoOddMinor-TreeDecompProduct}
For every graph $X$ there exists $c\in\NN$ such that every odd-$X$-minor-free graph $G$ has a tree-decomposition $(B_x:x\in V(T))$, such that for every $x\in V(T)$, the torso $G\langle B_x\rangle$ is contained in $H_1\StrongProd H_2$ for some graphs $H_1$ and $H_2$ with $\tw(H_1)\leq c$ and $\pw(H_2)\leq 2$.   
\end{lem}

\begin{proof}
Let $X$ be a graph.
Let $r$ and $a$ be the positive integers stated in \cref{OddMinorFreeTorsoStructure}.
Consider the tree-decomposition of $G$ from 
\cref{OddMinorFreeTorsoStructure}.

In case (a), apply \cref{NoXMinor-DirectedProductStructureNoApexVertices} 
to $G\langle{B_x}\rangle$ ignoring the direction of edges in the resulting product. We obtain a tree-decomposition $(C_z:z\in V(U))$ of $G\langle B_x\rangle$ such that for every $z\in V(U)$, 
$G\langle{C_z}\rangle \subsetsim H\StrongProd F$ for some graph $H$ with $\tw(H)\leq c_r$ and for some fan graph $F$, where $c_r$ depends only on $r$. We now replace $x$ in $T$ by $U$. More precisely, first let $T'$ be the disjoint union of $T-x$ and $U$. So $T'$ is a forest. For each edge $xy$ of $T$, since $B_x\cap B_y$ is a clique in $G\langle B_x\rangle$, there is a node $z\in V(U)$ with $B_x\cap B_y\subseteq C_z$; add the edge $yz$ to $T'$ and let $B_z:=C_z$. Now $T'$ is a tree, and $(B_x:x\in V(T'))$ is a tree-decomposition of $G$. Replace $T$ by $T'$, and repeat this step whenever case (a) applies for $x$. For each node $z$ created, $G\langle{B_z}\rangle \subsetsim H\StrongProd F$ for some graph $H$ with $\tw(H)\leq c_r$, and for some fan graph $F$ which has $\pw(F)\leq 2$. 

In case (b), 
$G\langle{B_x}\rangle \subsetsim H+ K_a$ where $H$ is bipartite, implying $H\subsetsim A+B$, for some graphs $A$ and $B$ with no edges. 
By \cref{JoinProduct}, 
$G\langle{B_x}\rangle \subsetsim A+B+ K_a
\subsetsim (A+K_a)\StrongProd (B+K_1)$.
Here $\tw(A+K_a)=a$ and $\pw(B+K_1)=1$. 

In both cases, $G\langle{B_x}\rangle$ is contained in $H_1\StrongProd H_2$, where $\tw(H_1)\leq \max\{c_r,a\}$ and $\pw(H_2)\leq 2$. 
\end{proof}

\begin{lem}
\label{TwIntPw-torsos}
Let $G$ be a graph that has a taut tree-decomposition  $(B_x:x\in V(T))$ such that $\TwIntPw(G[B_x])\leq k$ for each $x\in V(T)$. Then $\TwIntPw(G)\leq k$.
\end{lem}

\begin{proof}
We proceed by induction on $V(T)$. The base case $|V(T)|=1$ is trivial. Now assume $|V(T)|\geq 2$. Let $x$ be a leaf of $T$. Let $y$ be the neighbour of $x$ in $T$. Let $G':=G-(B_x\setminus B_y)$. So $(B_z:z\in V(T-x))$ is a taut tree-decomposition of $G'$, where $\TwIntPw(G[B_z])\leq k$ for each $z\in V(T-x)$. By induction, $\TwIntPw(G')\leq k$. That is, $G'$ has a tree-decomposition $(C_a:a\in V(R))$ and a path-decomposition $(A_i:i\in\ZZ)$ that are $k$-orthogonal. 
Since $B_x\cap B_y$ is a clique in $G'$, there exists $a^*\in V(T)$ and $i^*\in\ZZ$ such that $B_x\cap B_y\subseteq C_{a^*} \cap A_{i^*}$. 

By assumption, 
$\TwIntPw(G\langle{B_x}\rangle)\leq k$. That is, 
$G\langle{B_x}\rangle$ has a tree-decomposition $(C_p:p\in V(S))$ and a path-decomposition $(E_j:j\in\ZZ)$ that are $k$-orthogonal. 
Since $B_x\cap B_y$ is a clique in $G\langle{B_x}\rangle$, there exists $p^*\in V(S)$ and $j^*\in\ZZ$ such that $B_x\cap B_y\subseteq C_{p^*} \cap E_{j^*}$. 

Let $T'$ be the tree obtained from the disjoint union of $R$ and $S$ by adding the edge $a^*p^*$. Since $V(G') \cap V(G\langle{B_x}\rangle) = B_x\cap B_y\subseteq C_{a^*} \cap C_{p^*}$, it follows that $(C_z:z\in V(T'))$ is a tree-decomposition of $G$. 

%As illustrated in \cref{MergePathDecomps}, l
Let $A'_i := A_i \cup E_{i-i^*+j^*}$ for each $i\in\ZZ$. Since $V(G') \cap V(G\langle{B_x}\rangle) = B_x\cap B_y \subseteq A_{i^*} \cap E_{j^*}$, it follows that $(A'_i:i\in\ZZ)$ is a path-decomposition of $G$. 

For $z\in V(R)$ and $i\in\ZZ$,
\begin{align*}
A'_i \cap C_z 
& = A'_i \cap (C_z \cap V(G'))  \hspace*{47mm} \text{(since $C_z\subseteq V(G')$)}\\
& = (A'_i \cap V(G')) \cap C_z\\
& = \big( (A_i \cap V(G')) \cup (E_{i-i^*+j} \cap V(G') ) \big) \cap C_z \hspace*{10mm}  \text{(by the definition of $A'_i$)}\\
& \subseteq A_i \cup (B_x \cap B_y) \cap C_z \\
& =  A_i \cap C_z  \hspace*{65mm} \text{(since $B_x\cap B_y \subseteq A_{i^*}$)}.
\end{align*}
A symmetric argument shows that $A'_i\cap C_z=E_{i-i^*+j^*}\cap C_z$ for $z\in V(S)$ and $i\in\ZZ$.
Hence, $|A'_i\cap C_z|\leq k$ for each $z\in V(T')$ and $i\in\ZZ$. Therefore $\TwIntPw(G)\leq k$. 
\end{proof}

We now prove the following theorem from \cref{Introduction}.

\MainTWIntPW*

\begin{proof}
By \cref{NoOddMinor-TreeDecompProduct} there exists $c=c(X)$ such that every odd-$X$-minor-free graph $G$ has a tree-decomposition $\TT=(B_x:x\in V(T))$, such that 
$\TwIntPw( G\langle B_x\rangle ) \leq ((\tw \boxtimes \pw)( G\langle B_x\rangle )+1)^2\leq (c+1)^2$ for each $x\in V(T)$. Let $G':=\bigcup_{x\in V(T)}G\langle{B_x}\rangle$. 
So $\TT$ is a taut tree-decomposition of $G'$, and
$\TwIntPw( G'[B_x] ) = \TwIntPw( G\langle B_x\rangle )\leq (c+1)^2$ for each $x\in V(T)$. 
By \cref{TwIntPw-torsos},  
$\TwIntPw(G)\leq \TwIntPw(G')\leq (c+1)^2$.
\end{proof}

\subsection{Layered Treewidth and Tree-Bandwidth}

This section uses the notion of layered treewidth to prove upper bounds on the main parameters studied in this paper. 

A \defn{layering} of a graph $G$ is an ordered partition $(L_0,L_1,\dots)$ of $V(G)$ such that for every edge $vw\in E(G)$, if $v\in L_i$ and $w\in L_j$ then $|i-j|\leq 1$. For example, if $r$ is a vertex in a connected graph $G$ and $L_i:=\{v\in V(G):\dist_G(r,v)=i\}$ for $i\in\NN_0$, then $(L_0,L_1,\dots)$ is a layering of $G$, called  a \defn{BFS-layering} of $G$. The \defn{layered treewidth} of a graph $G$ is the minimum $k\in\NN_0$ such that $G$ has a layering $(L_0,L_1,\dots)$ and a tree-decomposition $(B_x:x\in V(T))$ such that $|L_i\cap B_x|\leq k$ for each $i\in\NN_0$ and each $x\in V(T)$. This definition was independently introduced by 
\citet{DMW17} and \citet{Shahrokhi13}, and has applications to asymptotic dimension~\citep{BBEGLPS24} and clustered colouring~\citep{LW4,DEMWW22,LW1} for example.

Layered treewidth was a precursor to the development of row treewidth (introduced in \cref{Introduction}). It is straightforward to see and well-known that $\ltw(G)\leq \rtw(G)$. However, \citet{BDJMW22} showed that $\ltw$ and $\rtw$ are separated: for any $k\in\NN$ there exists a graph $G$ with $\ltw(G)=1$ and $\rtw(G)>k$. On the other hand, within a minor-closed class, $\ltw$ and $\rtw$ are tied. In particular, using \cref{JJstMain}, \citet[Corollary~25]{ISW} showed that $\ltw(G)\leq (t-1)\rtw(G)-1$ for every $K_t$-minor-free graph $G$.

The following discussion re-uses the notation from the definition of layered treewidth. Note that  $(L_0\cup L_1,L_1\cup L_2,\dots)$ is a path-decomposition of $G$, where $|(L_{i-1}\cup L_{i})\cap B_x|\leq 2\ltw(G)$ for each $i\in\NN$ and $x\in V(T)$. Thus 
\begin{equation}
\TwIntTw(G)\leq\TwIntPw(G)\leq 2\ltw(G).
\end{equation}
By \cref{ttw-otw},
\begin{equation}
\label{ttw-ltw}
\ttw(G)\leq %\ptw(G)\leq 
2\ltw(G)-1.
\end{equation}
Similarly, since $|L_i\cap B_x|\leq k$ for each $i\in\NN_0$ and $x\in V(T)$, the subgraph $G[B_x]$ is contained in $P \StrongProd K_k$ for some path $P$, and thus $\bw(G[B_x])\leq 2k-1$. Hence 
\begin{equation}
\label{tbw-ltw}
\ttw(G)\leq \tpw(G)\leq \tbw(G)\leq2\ltw(G)-1.
\end{equation}
Finally, let $X:=\bigcup\{L_{2i-1}:i\in\NN\}$ and $Y:=\bigcup\{L_{2i}:i\in\NN\}$. So $\{X,Y\}$ is a partition of $G$, where $G[X]:=G[L_1]\cup G[L_3]\cup\dots$ and $G[Y]:=G[L_2]\cup G[L_4]\cup\dots$. So each of $G[X]$ and $G[Y]$ is a disjoint union of graphs of treewidth at most $k-1$. Hence $G[X]$ and $G[Y]$ have treewidth at most $k-1$. By \cref{JoinTreewidth},
\begin{equation}
\label{ltw-product}
\twtw(G)\leq \ltw(G).
\end{equation}

The \defn{layered pathwidth} of a graph $G$, denoted by \defn{$\lpw(G)$}, is the minimum $k\in\NN_0$ such that $G$ has a layering $(L_0,L_1,\dots)$ and a path-decomposition $(B_1,B_2,\dots)$ such that $|L_i\cap B_j|\leq k$ for each $i,j$. This definition is due to \citet{BDDEW19}. Defining $X$ and $Y$ as above, $G[X]$ and $G[Y]$ have pathwidth at most $\lpw(G)-1$. By \cref{JoinPathwidth}, 
\begin{equation}
\label{pwpw-lpw}
    \pwpw(G) \leq \lpw(G).
\end{equation}
\citet{DEJMW20} showed that a minor-closed class $\GG$ has bounded layered pathwidth if and only if some apex-forest is not in $\GG$ (see \citep{HLMR} for a simpler proof with better bounds). With \cref{pwpw-lpw} this implies:

\begin{cor}
\label{NoApexForest-pwpw}
    Graphs excluding a fixed apex-forest minor have bounded $\pwpw$-number. 
\end{cor}

We now consider specific minor-closed classes, such as  planar graphs and graphs embeddable on a fixed surface. \citet{DMW17} showed that every planar graph $G$ has \begin{equation}
\label{ltw-planar}
    \ltw(G)\leq 3,
\end{equation}
implying $\ttw(G)\leq \tpw(G)\leq \tbw(G)\leq 5$ by \eqref{tbw-ltw}. Using a variation of the proof method of \citet{DMW17} we establish the following best-possible upper bound for planar graphs.

\begin{thm}
\label{ttw-tpw-planar}
Every planar graph $G$ has an ordering $\preceq$ of $V(G)$ and a tree-decomposition $(B_v:v\in V(T))$ such that for each $v\in V(T)$,  $G[B_v]$ has bandwidth at most 3 with respect to $\preceq$ (restricted to $B_v$). In particular, $\ttw(G)\leq \tpw(G) \leq \tbw(G) \leq 3$. 
\end{thm}

\begin{proof}
We may assume $G$ is a plane triangulation. Let $r$ be a vertex on the outerface. For $i\in\NN_0$, let $L_i:=\{v\in V(G):\dist(r,v)=i\}$. So $(L_0,L_1,\dots)$ is a layering of $G$. Let $T$ be the \defn{Lex-BFS} spanning tree of $G$ rooted at $r$. This means that each layer $L_i$ is equipped with a linear ordering $\preceq_i$, such that each vertex $v\in L_{i+1}$ is adjacent in $T$ to the neighbour $u$ of $v$ in $L_i$ that is minimal with respect to $\preceq_i$, and $u$ is the parent of $v$ in $T$. Then $\preceq_{i+1}$ is an ordering of $L_{i+1}$ defined by the order of the parents in $\preceq_i$, breaking ties arbitrarily. Let $\preceq$ be the ordering $\preceq_0,\preceq_1,\preceq_2,\ldots$ of $V(G)$.

For each vertex $x\in V(T)$, let $P_x$ be the $xr$-path in $T$. Let $T^*$ be the spanning tree of the dual $G^*$ not crossing $T$. (This idea, now called the \defn{tree-cotree} method,  dates to the work of \citet{vonStaudt} from 1847.)\ For each vertex $v$ of $T^*$, corresponding to face $xyz$ of $G$, let $B_v:= P_x\cup P_y\cup P_z$. \citet{Eppstein99} showed that $(B_v:v\in V(T^*))$ is a tree-decomposition of $G$ (for any spanning tree $T$). 

Consider a node $v\in V(T^*)$ corresponding to face $xyz$ of $G$. We now show that $G[B_v]$ has bandwidth at most 3 with respect to $\preceq$ (restricted to $B_v$). For each $i\in\NN_0$, let $C_i:=L_i\cap B_v$. Each of $P_x,P_y,P_z$ contribute at most one vertex to $C_i$. So $|C_i|\leq 3$. 
Consider each set $C_i$ to be ordered by $\preceq_i$. 
Note that $\preceq$ restricted to $B_v$ is the ordering $C_0,C_1,\dots$.

We now show that each edge of $G[B_v]$ has span at most 3 in the ordering $C_0,C_1,\dots$ of $G[B_v]$. Since $(L_0,L_1,\dots)$ is a layering of $G$, $(C_0,C_1,\dots)$ is a layering of $G[B_v]$. So it suffices to show that for each $i\in\NN_0$, each edge with both endpoints in $C_i\cup C_{i+1}$ has span at most 3. Call this property ($\star$). 

By construction, each vertex in $C_i$ (except $x,y,z$) has at least one neighbour in $C_{i+1}$. Also, $C_i$ is precisely the set of parents in $T$ of vertices in $C_{i+1}$. Each vertex in $C_{i+1}$ has exactly one parent in $T$, so $|C_i|\leq|C_{i+1}|$.

Case 1. $|C_i|+|C_{i+1}|\leq 4$. Then $(\star)$ holds trivially. 

Case 2. $|C_i|=2$ and $|C_{i+1}|=3$: Say $C_i=\{a,b\}$ with  $a\prec_i b$. 

Case 2a. $C_{i+1}=\{a',a'',b'\}$ with $aa',aa'',bb'\in E(T)$, where $a'\prec_{i+1} a'' \prec_{i+1} b'$ by the Lex-BFS property. 
So $C_i\cup C_{i+1}$ is ordered $a,b,a',a'',b'$, and $(\star)$ holds since $ab'$ is not an edge by the Lex-BFS property. 

Case 2b. $C_{i+1}=\{a',b',b''\}$ with $aa',bb',bb''\in E(T)$, where $a'\prec_{i+1} b' \prec_{i+1} b''$ by the Lex-BFS property. So $C_i\cup C_{i+1}$ is ordered $a,b,a',b',b''$, and $(\star)$ holds since neither $ab'$ nor $ab''$ are edges by the Lex-BFS property. 

Case 3. $|C_i|=3$ and $|C_{i+1}|=3$: Then $C_i$ is matched to $C_{i+1}$ by edges in $T$. Say $C_i=\{a,b,c\}$ and $C_{i+1}=\{a',b',c'\}$ with $aa',bb',cc'$ matched in $T$. Say $a\prec_i b \prec_i c$. Then $a'\prec_{i+1} b' \prec_{i+1} c'$ by the Lex-BFS property. So $C_i\cup C_{i+1}$ is ordered $a,b,c,a',b',c'$, 
and $(\star)$ holds since neither $ab'$, $ac'$ nor $bc'$ are edges by the Lex-BFS property.

This shows that $G[B_v]$ has bandwidth at most 3 with respect to $\preceq$ (restricted to $B_v$). In particular, 
$\ttw(G)\leq \tpw(G) \leq \tbw(G) \leq 3$.
\end{proof}

\cref{ttw-tpw-planar} is best possible, since $\ttw(K_4)=\tpw(K_4)=\tbw(K_4)=3$. It also strengthens and implies a result of \citet{HRWY21}, who showed $\tchi(G)\leq 4$ for every planar graph $G$ (without using the 4-Colour-Theorem). Our proof of \cref{ttw-tpw-planar} uses a similar tree-decomposition as the proof of \citet{HRWY21} (who used an arbitrary BFS spanning tree, not necessarily Lex-BFS). 

The next lemma enables \cref{ttw-tpw-planar} to be extended to the setting of $K_5$-minor-free graphs. 

\begin{lem}
\label{ttw-tpw-torsos}
For any graph parameter $f$, if a graph $G$ has a tree-decomposition $(B_x:x\in V(T))$ such that $\treef(G\langle B_x\rangle)\leq k$ for each $x\in V(T)$, then $\treef(G) \leq k$.
\end{lem}

\begin{proof}
Since the adhesion sets are cliques, we can glue the tree-decompositions of different torsos together to get a tree-decomposition of $G$ in which each torso is a torso of one of the original tree-decompositions. The result follows.
\end{proof}

\begin{cor}
\label{ttw-tpw-K5MinorFree}
For every $K_5$-minor-free graph $G$, 
\[\ttw(G)\leq \tpw(G) \leq \tbw(G) \leq 3.\]
\end{cor}

\begin{proof}
\citet{Wagner37} proved that $G$ has a tree-decomposition with adhesion at most 3 and in which each torso is planar or $V_8$, where $V_8$ is the graph obtained from an 8-cycle $(1,2,3,4,1',2',3',4')$ by adding each edge $ii'$. Consider the path-decomposition $\{1,2,1',4,4'\},\{2,1',2',4,4'\},\{2,3,2',4,4'\},\{3,2',3',4,4'\}$ of $V_8$. The subgraph induced by each bag has bandwidth at most $3$ using the stated orderings. So $\tbw(V_8)\leq 3$. The result now follows from \cref{ttw-tpw-torsos,ttw-tpw-planar}.
\end{proof}

Now consider graphs embeddable on surfaces.  \citet{DMW17} showed that 
$\ltw(G)\leq 2g+3$ for every graph $G$  with Euler genus $g$. 
%By \eqref{ttw-ltw}, $$\ttw(G)\leq %\ptw(G)\leq  2\ltw(G)-1 \leq 2(2g+3)-1=4g+5$$ 
By \eqref{tbw-ltw},
\[\ttw(G) \leq \tpw(G) \leq \tbw(G)\leq 2\ltw(G)-1 \leq 2(2g+3)-1=4g+5.\]
These bounds are tight up to  a constant factor:
Take a random $256$-regular graph $G$ on $n$ vertices. By \cref{ttw-RandomRegularGraph} below,   $\ttw(G)+1 \geq \frac{n}{4}$ asymptotically almost surely. On the other hand, if $g$ is the Euler genus of $G$, then $g\leq |E(G)|=128n$. Thus, $\ttw(G)+1 \geq \frac{n}{4} \geq \frac{g}{512}$ asymptotically almost surely.

We now show that apex graphs (and thus $K_6$-minor-free graphs) have unbounded tree-bandwidth. In fact, they have unbounded tree-$\Delta$. 

\begin{lem}
\label{Pyramid}
If $G:=(P_n\CartProd P_n)+K_1$ then tree-$\Delta(G)\geq n+1$ and $\tbw(G)\geq \ceil{\frac{n+1}{2}}$. 
\end{lem}

\begin{proof}
Let $v$ be the dominant vertex in $G$. Consider the tree-decomposition of $G$ with minimum width, and subject to this, with minimum number of bags. Then every leaf bag contains a vertex that does not belong to any other bag.
Since $v$ is adjacent to all vertices in $G$, $v$ belongs to every leaf bag.
So $v$ is in every bag. Since $\tw(P_n\CartProd P_n)=n$, some bag has at least $n+1$ vertices of $P_n\CartProd P_n$. Thus this bag contains a star $K_{1,n+1}$, implying tree-$\Delta(G)\geq n+1$ and $\tbw(G)\geq \ceil{\frac{n+1}{2}}$. 
\end{proof}

More generally, the following result characterizes when a minor-closed class has bounded tree-bandwidth. 

\begin{thm}
\label{tbw-characterisation}
    The following are equivalent for a minor-closed class $\GG$:
    \begin{enumerate}[(1)] 
        \item some apex graph is not in $\GG$,
        \item $\GG$ has bounded layered treewidth,
        \item $\GG$ has bounded row treewidth,
        \item $\GG$ has bounded local treewidth (that is, there exists $f$ such that $\tw(G)\leq f(\diam(G))$ for every connected graph $G\in\GG$),
        \item $\GG$ has bounded tree-bandwidth,
        \item $\GG$ has bounded tree-$\Delta$.
    \end{enumerate} 
\end{thm}

\begin{proof}
\citet{Eppstein-Algo00}  proved that (1) and (4) are equivalent (see \citep{DH-Algo04,DH-SODA04} for alternative proofs with improved bounds). \citet{DMW17}  proved that (1) and (2) are equivalent. \cref{ApexMinorFreeProduct} by \citet{DJMMUW20} says that (1) and (3) are equivalent. (Note that these results in \citep{DMW17,DJMMUW20} depend on a version of the Graph Minor Structure Theorem for apex-minor-free graphs due to \citet{DvoTho}.)\ 
This shows that (1), (2), (3) and (4) are equivalent. 
Property (2) implies (5) since $\tbw(G)\leq2\ltw(G)-1$. 
Property (5) implies (6) since $\Delta(G)\leq 2\bw(G)$. 
Property (6) implies (1) by \cref{Pyramid}.
\end{proof}

%%%%%%%%%%%%%%%%%%%%%%%%%%%%%%%%%%%
\subsection{Tree-Treedepth}
\label{Tree-Treedepth}

\cref{OddMinorFree-TreePathDecompsOrtho,ttw-otw} imply that every proper minor-closed class has bounded $\tpw$. We now  show this result cannot be strengthened with $\ttd$ instead of $\tpw$, even for planar graphs. 

An \defn{$n\times n$ Hex graph} is any planar graph $G$ obtained from the $n\times n$ grid by triangulating each internal face. Let $B_i:=\{i,i+1\}\times\{1,\dots,n\}$ for $i\in\{1,\dots,n-1\}$. Then $(B_1,\dots,B_{n-1})$ is a path-decomposition of $G$ where each bag $B_i$ induces a subgraph with bandwidth at most 2. Thus $\ttw(G)\leq \tpw(G) \leq \tbw(G) \leq 2$. 
% and $\ptw(G)\leq \ppw(G)\leq 2$. 
On the other hand, the next lemma shows that Hex graphs have large $\ttd$. A \defn{bramble} in a graph $G$ is a set $\BB$ of connected subgraphs of $G$, such that $G[V(A)\cup V(B)]$ is connected for all $A,B\in\BB$.

\begin{lem}
\label{ApplyHex}
For any $n\times n$ Hex graph $G$, every tree-decomposition of $G$ has a bag that contains an $n$-vertex path of $G$. 
In particular, $\ttd(G)\geq \ceil{\log_2(n+1)}$. 
\end{lem}

\begin{proof}
Let $P,P',Q,Q'$ be the left-most, right-most, top and bottom paths of $G$ respectively. Let $\BB$ be the collection of all connected subgraphs of $G$ that intersect each of $P,P',Q,Q'$. Then $\BB$ is a bramble (since every path from $P$ to $P'$ intersects every path from $Q$ to $Q'$). Let $(B_x:x \in V(T))$ be a tree-decomposition of $G$. \citet{ST93} showed there exist $x \in V(T)$ such that $B_x$ intersects every element of $\BB$. Suppose that $G[B_x]$ does not contain a path of length $n$. Then $G[B_x] $ does not contain a path from $P$ to $P'$, and so by the Hex Lemma~\citep{Gale79,HT19}, $G - B_x$ contains a path $R_1$ from $Q$ to $Q'$. Similarly, $G - B_x$ contains a path $R_2$ from $P$ to $P'$. Thus $R_1 \cup R_2 \in \BB$ and $(R_1 \cup R_2) \cap B_x = \emptyset$, which is a contradiction. Hence, $G[B_x]$ contain a path of length $n$. 
The final claim follows, since the $n$-vertex path has treedepth $\ceil{\log_2(n+1)}$ (see \citep[(6.2)]{Sparsity}). 
\end{proof}

We now characterize the minor-closed classes with bounded tree-treedepth.

\begin{thm}
\label{ttd-Characterisation}
The following are equivalent for a minor-closed class $\GG$:
    \begin{enumerate}[(1)]
    \item $\GG$ has bounded tree-treedepth,
    \item $\GG$ has bounded treewidth,
    \item some planar graph is not in $\GG$. 
    \end{enumerate}
\end{thm}

\begin{proof}
\citet{RS-V} proved that (2) and (3) are equivalent. Property 
(1) implies (3) since if $\GG$ contains every planar graph, then it contains every Hex graph, and $\GG$ has unbounded $\ttd$ by \cref{ApplyHex}. Property (2) implies (1) since $\td(G)\leq|V(G)|$ for every graph $G$.
\end{proof}

%%%%%%%%%%%%%%%%%%%%%%%%%%%%%%%%%%
\section{\boldmath Lower Bounds}
\label{LowerBounds}

This section proves lower bounds on the parameters introduced in \cref{Introduction}. We start by proving lower bounds on $\twtw$-number, $\twpw$-number and $\pwpw$-number. Most of these results actually hold in the 3-term product setting. We then show lower bounds on $\TwIntTw$-number and tree-treewidth. 

\subsection{\boldmath Lower Bound for 
\texorpdfstring{$\twtw$}{(tw⛝tw)}-Number of Bounded Treedepth Graphs}
\label{LowerBoundTreewidth}

Our first lower bound is for the $\twtw$-number of bounded treedepth graphs. The following notation will be helpful. Consider a graph $G$ contained in $H_1\StrongProd H_2 \StrongProd K_c$. So there is an isomorphism $\phi_0$ from $G$ to a subgraph of $H_1\StrongProd H_2\StrongProd K_c$. For each $v\in V(G)$, if $\phi_0(v)=(x,y,z)\in V(H_1\StrongProd H_2 \StrongProd K_c)$, then let $\phi(v):=(x,y)\in V(H_1\StrongProd H_2)$. For a clique $C$ of $G$, let $C_1 := \{x \in V(H_1): \phi(v)=(x,y),\,v\in C\}$ and $C_2 := \{y \in V(H_2): \phi(v)=(x,y),\,v\in C\}$. Note that $C_i$ is a clique in $H_i$. %(defined with respect to $\phi_0$). 
We use this notation throughout this subsection.

\begin{lem}
\label{td-LowerBound-induction}
For any $k\in\NN_0$ and $c\in\NN$ there exists a graph $G_k$ with $\td(G_k)\leq k+1$, such that if 
$G_k \subsetsim H_1\StrongProd H_2 \StrongProd K_c$ then 
there exists a $(k+1)$-clique $C$ in $G_k$ with $|C_1|+|C_2|\geq k+2$.
\end{lem}

\begin{proof}
We proceed by induction on $k$. In the base case, $k=0$, let $G_0:=K_{1}$. If $G\subsetsim H_1\StrongProd H_2 \StrongProd K_c$, then 
$C:=V(G_0)$ is a $1$-clique with $|C_1|= 1$ and $|C_2|= 1$, implying $|C_1|+|C_2|= 2=k+2$. Now assume that $k\geq 1$ and the desired graph $G_{k-1}$ exists. Let $G_k$ be obtained from $G_{k-1}$ by adding, for each $k$-clique $C$ in $G_{k-1}$, a set $V_C$ of $k^2c+1$ vertices adjacent to $C$. Observe that $\td(G_k)\leq \td(G_{k-1})+1\leq k+1$. Say 
$G_k \subsetsim H_1\StrongProd H_2 \StrongProd K_c$. 
Since $G_{k-1}\subseteq G_k$, by induction, there exists a $k$-clique $C$ in $G_{k-1}$ such that $|C_1|+|C_2|\geq k+1$. Note that $|C|\leq\omega(G_{k-1})\leq \td(G_{k-1})\leq k$. Since $|C_1\times C_2 \times V(K_c)|
\leq |C|^2 c 
\leq k^2c
< |V_{C}|$, 
there exists a vertex $v\in V_{C}$ such that $\phi(v)\not\in C_1\times C_2$. 
Let $C':=C\cup\{v\}$, which is a $(k+1)$-clique in $G_k$  by construction. 
Say $\phi(v)=(p,q) \in V(H_1\StrongProd H_2)$. 
So $p\not\in C_1$ or $q\not\in C_2$, and
$C_1'=C_1\cup\{p\}$ and
$C_2'=C_2\cup\{q\}$. 
Thus $|C_1'|+|C_2'|\geq |C_1|+|C_2|+1\geq k+2$, and we are done.
\end{proof}

\cref{td-LowerBound-induction} implies:

\begin{thm}
\label{td-lowerbound}
For any $k\in\NN_0$ and $c\in\NN$ there exists a graph $G$ with $\td(G)\leq k+1$, such that if $G\subsetsim H_1\StrongProd H_2 \StrongProd K_c$ then $\omega(H_1)+\omega(H_2)\geq k+2$, 
implying $\tw(H_1)+\tw(H_2)\geq k$. 
\end{thm}

\cref{td-lowerbound} leads to the following proof of the lower bound in \eqref{kX-td}.

\begin{cor}
\label{XMinorFree-twtw-td}
 For any $c\in\NN$ and any graph $X$ with $\td(X)\geq 2$, there exists an $X$-minor-free graph $G$ such that if $G\subsetsim H_1\StrongProd H_2 \StrongProd K_c$, then $\tw(H_1)+\tw(H_2)\geq \td(X)-2$. In particular, 
 $\twtwa(\GG_X)\geq \ceil{\tfrac12(\td(X)-2)}$.
\end{cor}

\begin{proof}
Let $G$ be the graph from \cref{td-lowerbound} with $k:= \td(X)-2$. 
So $\td(G)\leq k+1 <\td(X)$. 
Since treedepth is minor-monotone, $G$ is $X$-minor-free. 
By \cref{td-lowerbound}, if $G\subsetsim H_1\StrongProd H_2 \StrongProd K_c$ then 
$\tw(H_1)+\tw(H_2)\geq k=\td(X)-2$.
\end{proof}

\cref{XMinorFree-twtw-td} completes the proof that $\twtwa(\GG_X)$ is tied to the treedepth of $X$. 
It would be interesting to prove tight  bounds here. 

\begin{open}
    What is the minimum function $f$ such that $\twtwa(\GG_X)\leq f(\td(X))$ for every graph $X$?
\end{open}

For graphs of treewidth $k$, we now push the above proof method a little further. We need the following simple lemma. For finite sets $A$ and $B$, a \defn{column} of $A\times B$ is any set $\{a\}\times B$ with $a\in A$, and a \defn{row} of $A\times B$ is any set $A\times\{b\}$ with $b\in B$. 

\begin{lem}
\label{Fact}
Let $A$ and $B$ be finite sets, and let $X$ be a subset of $A\times B$, such that each element of $X$ is the only element of $X$ in some row or column of $A\times B$. Then $|X| \leq |A|+|B|-1$. Moreover, if $|A|\geq 2$ and $|B|\geq 2$, then $|X|\leq|A|+|B|-2$. 
\end{lem}

\begin{proof}
We proceed by induction on $|A|+|B|$. If $|A|=1$ then $|X|\leq|B|=|A|+|B|-1$, as desired. If $|B|=1$ then $|X|\leq|A|=|A|+|B|-1$, as desired. Now assume that $|A|\geq 2 $ and $|B|\geq 2$. Suppose that $|A|=2$. If two elements of $X$ are in the same row, then $|X|\leq 2\leq |A|+|B|-2$, as desired. Otherwise, no two elements of $X$ are in the same row, implying $|X|\leq |B|=|A|+|B|-2$, as desired. The case $|B|=2$ is analogous. Now assume that $|A|\geq 3$ and $|B|\geq 3$. Let $v\in X$. Thus $v$ is the only element in $X$ in some row or column of $A\times B$. Without loss of generality, $v$ is the only element in $X$ in some row $A\times\{b\}$. Hence $X\setminus\{v\}$ is a subset of $A\times (B\setminus\{b\})$, such that each element of $X$ is the only element of $X$ in some row or column. By induction, $|X\setminus\{v\}|\leq |A|+(|B|-1)-2$, implying $|X|\leq |A|+|B|-2$. 
\end{proof}

\begin{lem}
\label{tw-LowerBound-tweak}
For any $c,k\in\NN$, there exists a graph $G$ with $\tw(G)\leq k$, such that for any graphs $H_1$ and $H_2$ with $\omega(H_1)\leq k$ and $\omega(H_2)\leq k$, if 
$G \subsetsim H_1\StrongProd H_2 \StrongProd K_c$ then there exists a $(k+1)$-clique $C$ in $G$ with $|C_1|+|C_2|\geq k+3$.
\end{lem}

\begin{proof}
Let $G_k$ be the graph from \cref{td-LowerBound-induction}.
So $\tw(G_k)\leq \td(G_k)-1\leq  k$. Let $G$ be obtained from $G_k$ by adding, for each $k$-clique $C$ in $G_k$, a set $V_C$ of $k^2c+1$ vertices adjacent to $C$. Observe that $\tw(G)\leq k$.

Let $H_1$ and $H_2$ be graphs with $\omega(H_1)\leq k$ and $\omega(H_2)\leq k$, and assume that $G \subsetsim H_1\StrongProd H_2 \StrongProd K_c$. Define $\phi$ and $C_1,C_2$ as above. Since $G_k\subseteq G$, by \cref{td-LowerBound-induction}, there exists a $(k+1)$-clique $C$ in $G_k$ with $|C_1|+|C_2|\geq k+2$. Note that $\phi(C)\subseteq C_1\times C_2$. 

First suppose there exists a subset $C'$ of $C$ with size $k$ such that $\phi(C')$ intersects every row and column of $C_1\times C_2$. 
Since $|C_1\times C_2 \times V(K_c)| 
\leq \omega(H_1)\,\omega(H_2)\,c
\leq k^2c < |V_{C'}|$, 
there exists a vertex $v\in V_{C'}$ such that $\phi(v)\not\in C_1\times C_2$. 
Let $C'':=C'\cup\{v\}$, which is a $(k+1)$-clique in $G$  by construction. 
Say $\phi(v)=(p,q) \in V(H_1\StrongProd H_2)$. 
So $p\not\in C_1$ or $q\not\in C_2$, and
$C_1''=C_1\cup\{p\}$ and
$C_2''=C_2\cup\{q\}$. 
Thus $|C_1''|+|C_2''|\geq |C_1|+|C_2|+1\geq k+3$, as desired.

Now assume there is no subset $C'$ of $C$ with size $k$ such that $\phi(C')$ intersects every row and column of $C_1\times C_2$. Since $|C|=k+1$, for each vertex $v\in C$, $\phi(C\setminus\{v\})$ avoids some row or column of $C_1\times C_2$. Hence $v$ is the only vertex of $C$ such that $\phi(v)$ is in this row or column  of $C_1\times C_2$. Thus 
$|\phi(C)|=|C|=k+1$. If $|C_1|=1$ then $|C_2|=|C|=k+1$, which contradicts the assumption that $\omega(H_2)\leq k$. So $|C_1|\geq 2$. Similarly, $|C_2|\geq 2$. By \cref{Fact}, 
$k+1 =|\phi(C)| \leq |C_1|+|C_2|-2$. 
Hence $|C_1|+|C_2| \geq k+3$, as desired. 
\end{proof}

\cref{tw-LowerBound-tweak} leads to the following result:

\begin{thm}
\label{tw-lowerbound-tweaked}
If $\TT_k$ is the class of graphs with treewidth at most $k$, then 
\[\twtwa(\TT_k)= \ceil{\tfrac12(k+1)}.\]
\end{thm}

\begin{proof}
\cref{TreewidthProductTreewidth} implies the upper bound, 
$\twtwa(\TT_k)\leq \ceil{\tfrac12(k+1)}$. For the lower bound, suppose that $\twtwa(\TT_k)\leq \ceil{\tfrac12(k+1)}-1$. Since $\TT_1$ contains arbitrarily large trees, $\twtwa(\TT_k) \geq 1$. So $k\geq 2$. Thus, for some $c\in\NN$ every graph $G\in\TT_k$ is contained in $H_1\StrongProd H_2 \StrongProd K_c$, where $\tw(H_i)\leq \ceil{\tfrac12(k+1)}-1 \leq k-1$ for each $i\in\{1,2\}$. Thus $\omega(H_i)\leq k$. By \cref{tw-LowerBound-tweak}, there exists a graph $G\in\TT_k$ such that %(with respect to the above containment in $H_1\StrongProd H_2 \StrongProd K_c$) 
there exists a $(k+1)$-clique $C$ in $G$ with $|C_1|+|C_2|\geq k+3$. Hence
$k+3\leq |C_1|+|C_2|\leq \omega(H_1)+\omega(H_2)\leq \tw(H_1)+\tw(H_2)+2$. Therefore $\tw(H_1)+\tw(H_2)\geq k+1$, and $\tw(H_i)\geq\ceil{\tfrac12(k+1)}$ for some $i\in\{1,2\}$, a contradiction. Hence $\twtwa(\TT_k)\geq \ceil{\frac12(k+1)}$. 
\end{proof}

%%%%%%%%%%%%%%%%%%%%%%%%%%%%%%%%%%%%%%%%%%
\subsection{\boldmath Lower Bounds on \texorpdfstring{$\twtw$}{(tw⛝tw)}-Number} 
\label{ProductLowerBounds}
\label{LowerBounds-twtw}

This section proves lower bounds on $\twtw(G)$ for certain $X$-minor-free graphs. These lower bounds are expressed in terms of the apex-number of the excluded minor. For $k\in\NN_0$, let \defn{$\AAA_k$} be the class of graphs $X$ with apex-number $a(X)\leq k$; that is, $X\in \AAA_k$ if and only if there exists $S\subseteq V(X)$ such that $|S|\leq k$ and $X-S$ is planar. For example, $\AAA_0$ is the class of planar graphs, and $\AAA_1$ is the class of apex graphs.

For a graph class $\GG$ and $k\in\NN_0$, define \defn{$\GG^{(k)}$} as follows. Every graph in $\GG$ is in $\GG^{(k)}$. If $G_1\in \GG^{(k)}$ and $G_2\in\GG$, then a graph obtained by pasting $G_1$ and $G_2$ on a clique of at most $k$ vertices is in $\GG^{(k)}$. 

\begin{lem}
\label{ForceH}
For any $k,c,\ell\in\NN$, for any graph $H$ with $\tw(H)\leq \ell$, there exists a graph $G\in\AAA^{(\ell)}_\ell$,  such that for any graphs $H_1$ and $H_2$, if $\tw(H_2)\leq k$ and $G$ is contained in $H_1 \StrongProd H_2\StrongProd K_c$, then $H$ is contained in $H_1$. 
\end{lem}

\begin{proof}
%Since $\tw(H_2\StrongProd K_c)\leq (\tw(H_2)+1)c-1$, the result follows from the $c=1$ case by adjusting the value of $k$. Now assume that $c=1$. 
Let $h:=|V(H)|$. We may assume that $H$ is edge-maximal subject to $\tw(H)\leq \ell$. So $H$ is chordal with no $K_{\ell+2}$ subgraph. Let $(v_1,\dots,v_h)$ be an elimination ordering of $H$. That is, for each $i\in\{1,\dots,h\}$, if $N_i:= N_H(v_i) \cap\{v_1,\dots,v_{i-1}\}$ then $N_i$ is a clique of size at most $\ell$. 

Let $n:= (h-1)(k+1)c$. Recursively define graphs $G_1,G_2,\dots,G_h\in\AAA^{(\ell)}_\ell$ as follows. Let $G_1:=K_1 \in\AAA^{(\ell)}_\ell$. For $i\geq 2$, let $G_i$ be obtained from $G_{i-1}$ by adding, for each clique $S$ in $G_{i-1}$ of size at most $\ell$, an $n\times n$ grid $D^i_S$ complete to $S$ and disjoint from $G_{i-1}$. Note that $G_i[D^i_S\cup S]$ has $\ell$ vertices (namely, the vertices in $S$)  whose deletion leaves an $n\times n$ grid. So $a(G_i[D^i_S\cup S])\leq\ell$. Thus, the addition of $D^i_S$ can be described by pasting a graph in $\AAA^{(\ell)}_\ell$ and a graph in $\AAA_\ell$ on the clique $S$. Therefore $G_i$ is in $\AAA^{(\ell)}_\ell$.

We consider $G_1\subseteq G_2\subseteq\dots\subseteq G_h$. Let $G:=G_h$. Suppose that $G$ is contained in $H_1 \StrongProd H_2 \StrongProd K_c$, where  $\tw(H_2)\leq k$. For ease of notation, assume $G\subseteq H_1\StrongProd H_2 \StrongProd K_c$. Let $p$ be the maximum integer such that $H[\{v_1,\dots,v_p\}]$ is contained in $G$, where for $i\in\{1,\dots,p\}$, vertex $v_i$ is mapped to $(x_i,y_i,z_i)\in V(G_i)$, and $x_1,\dots,x_p$ are distinct vertices of $H_1$. 

Suppose for the sake of contradiction that $p\leq h-1$. Note that  $v_1,\dots,v_p$ are mapped to vertices in $G_p$, and thus $N_{p+1}$ is mapped to a clique in $G_p$. By construction, $D^{p+1}_{N_{p+1}}$ is an $n\times n$ grid graph complete to $N_{p+1}$ in $G_{p+1}$. If there is a vertex $(x,y,z)$ in the image of  $D^{p+1}_{N_{p+1}}$ with $x\not\in \{x_1,\dots,x_p\}$, then by mapping $v_{p+1}$ to $(x,y,z)$, we find that 
$H[\{v_1,\dots,v_{p+1}\}]$ is contained in $G$, where for $i\in\{1,\dots,p+1\}$, vertex $v_i$ is mapped to $(x_i,y_i,z_i)\in V(G_i)$, and $x_1,\dots,x_{p+1}$ are distinct vertices of $H_1$, thus contradicting the maximality of $p$. Hence $D^{p+1}_{N_{p+1}}$ is mapped to $\{x_1,\dots,x_p\} \times V(H_2) \times V(K_c)$ in $G$, implying  
\[(h-1)(k+1)c=n= \tw(D^{p+1}_{N_{p+1}})\leq p (\tw(H_2 \StrongProd K_c)+1)-1 \leq (h-1)((k+1)c-1)-1,\]
which is contradiction. Therefore $p=h$, implying $H$ is contained in $H_1$, as desired. 
\end{proof}

We use the following folklore lemma. 

\begin{lem}
\label{Pasting}
Fix $\ell\in\NN$ and let $X$ be any $\ell$-connected graph. Let $G_1$ and $G_2$ be any $X$-minor-free graphs. Let $C_i$ be a clique in $G_i$, with $|C_1|=|C_2|\leq \ell-1$. Let $G$ be a graph obtained by pasting $G_1$ and $G_2$ on $C_1$ and $C_2$. Then $G$ is $X$-minor-free.
\end{lem}

\begin{proof}
Suppose for the sake of contradiction that $X$ is a minor of $G$. So $G$ contains a model $\phi$ of $X$. Suppose that $V(\phi(x))\subseteq V(G_1)\setminus V(G_2)$, and $V(\phi(y))\subseteq V(G_2)\setminus V(G_1)$ for some distinct $x,y\in V(X)$. 
For each $z\in V(X)$, contract $\phi(z)$ to a vertex $z'$, and delete vertices not in $\bigcup_{z \in V(X)}V(\phi(z))$. We obtain a supergraph of $X$ in which the contraction of $C_1$ separates $x'$ and $y'$. Since 
$|C_1|\leq \ell-1$, this contradicts the assumption that $X$ is $\ell$-connected. Now we may assume $\phi(x)$ is not contained in $V(G_2)\setminus V(G_1)$, for each $x\in V(X)$. For each $x\in V(X)$, contract any edge of $\phi(x)$ with at least one end in $V(G_2)\setminus V(G_1)$. We obtain a model of $X$ contained in $G_1$, since $C_1$ is a clique. This contradicts the assumption that $G_1$ is $X$-minor-free. So $G$ is $X$-minor-free. 
\end{proof}

We now give several corollaries of \cref{ForceH}. 

\begin{cor}
\label{ForceH-Minor}
For any $k,c,\ell\in\NN$, for any $\ell$-connected graph $X$ with $a(X)=\ell$, for any graph $H$ with $\tw(H)\leq \ell-1$, there exists an $X$-minor-free graph $G$ such that for any graphs $H_1$ and $H_2$, if $\tw(H_2)\leq k$ and $G$ is contained in $H_1 \StrongProd H_2\StrongProd K_c$, then $H$ is contained in $H_1$. 
\end{cor}

\begin{proof}
Since $a$ is minor-monotone and $a(X)=\ell$, every graph in $\AAA_{\ell-1}$ is $X$-minor-free. By \cref{Pasting}, every  graph in $\AAA^{(\ell-1)}_{\ell-1}$ is $X$-minor-free. The result thus follows from \cref{ForceH}.
\end{proof}

\cref{ForceH-Minor} implies:

\begin{cor}
\label{ForceH-Symmetric}
For any $k,\ell,c\in\mathbb{N}$, for any $\ell$-connected graph $X$ with $a(X)=\ell$, for any graph $H$ with $\tw(H)=\ell-1$, there exists an $X$-minor-free graph $G$ such that for any graphs $H_1$ and $H_2$ with  $\tw(H_1)\leq k$ and $\tw(H_2)\leq k$, if $G$ is contained in $H_1 \StrongProd H_2 \StrongProd K_c$, then $H$ is contained in both $H_1$ and $H_2$, implying $\tw(H_1)\geq\ell-1$ and $\tw(H_2)\geq\ell-1$. 
\end{cor}

\begin{cor}
\label{twtw-LowerBound}
For any $\ell,c\in\mathbb{N}$, for any $\ell$-connected graph $X$ with $a(X)=\ell$, there exists an $X$-minor-free graph $G$ such that for any graphs $H_1$ and $H_2$, if $G$ is contained in $H_1 \StrongProd H_2 \StrongProd K_c$, then $\tw(H_1)\geq\ell-1$ or $\tw(H_2)\geq \ell-1$. In particular, $\twtwa(\GG_X)\geq \ell-1$. 
\end{cor}

\begin{proof}
The $\ell \leq 2$ case holds, since every connected graph on more than $c$ vertices is not contained in $H_1 \StrongProd H_2 \StrongProd K_c$ for graphs $H_1$ and $H_2$ with $\tw(H_1) \leq \ell-2$ and $\tw(H_2) \leq \ell-2$. Now assume that $\ell\geq 3$. Since $\tw(K_\ell)=\ell-1$, by \cref{ForceH-Symmetric} with $k=\ell-2$, there exists an $X$-minor-free graph $G$ such that for any graphs $G_1$ and $G_2$ with $\tw(G_1)\leq \ell-2$ and $\tw(G_2)\leq \ell-2$, if $G$ is contained in $G_1 \StrongProd G_2 \StrongProd K_{c}$, then $K_\ell$ is contained in both $G_1$ and $G_2$.
Suppose $G$ is contained in $H_1 \StrongProd H_2 \StrongProd K_c$ for some graphs $H_1$ and $H_2$. If $\tw(H_1) \leq \ell-2$ and $\tw(H_2)\leq \ell-2$, then $K_\ell$ is contained in both $H_1$ and $H_2$, so $\ell-2 \geq \tw(H_1) \geq \tw(K_\ell) = \ell-1$, which is a contradiction. Thus 
$\tw(H_1) \geq \ell-1$ or $\tw(H_2)\geq \ell-1$. This proves the lemma.
\end{proof}

When the excluded minor $X=K_t$, the above results are applicable with $\ell=a(K_t)=t-4$ (since $K_4$ is planar, and $K_t$ is $(t-4)$-connected). In particular, \cref{ForceH-Symmetric} implies:

\begin{cor}
\label{ForceH-Kt}
For any $c,k,t\in\NN$ with $t\geq 5$, and for any graph $H$ with $\tw(H)=t-5$, there exists a $K_t$-minor-free graph $G$ such that for any graphs $H_1$ and $H_2$ with  $\tw(H_1)\leq k$ and $\tw(H_2)\leq k$, if $G$ is contained in $H_1 \StrongProd H_2 \StrongProd K_c$, then $H$ is contained in both $H_1$ and $H_2$, implying $\tw(H_1)\geq t-5$ and $\tw(H_2)\geq t-5$. 
\end{cor}

And \cref{twtw-LowerBound} implies:

\begin{cor}
\label{twtw-LowerBound-Kt}
For any $c,t\in\NN$ with $t\geq 5$, there exists a $K_t$-minor-free graph $G$ such that for any graphs $H_1$ and $H_2$, if $G$ is contained in $H_1 \StrongProd H_2 \StrongProd K_c$, then $\tw(H_1)\geq t-5$ or $\tw(H_2)\geq t-5$. That is,
$\twtwa(\GG_{K_{t}})\geq t-5$. 
\end{cor}

\cref{KtMinorFreeThreeDimProduct,twtw-LowerBound-Kt} together show that $\twtwa(\GG_{K_t})\in\{t-5,t-4,t-3,t-2\}$. 

\begin{open}
    What is $\twtwa(\GG_{K_t})$?
\end{open}

When the excluded minor $X=K_{s,t}$ with $t\geq s\geq 3$, the above results are applicable with $\ell=a(K_{s,t})=s-2$ (since $K_{2,t}$ is planar, and since $K_{s,t}$ is $(s-2)$-connected). In particular, \cref{ForceH-Symmetric} implies:

\begin{cor}
\label{ForceH-Kst}
For any $c,k,t,s\in\NN$ with $t\geq s\geq 3$, and for any graph $H$ with $\tw(H)= s-3$, there exists a $K_{s,t}$-minor-free graph $G$ such that for any graphs $H_1$ and $H_2$ with  $\tw(H_1)\leq k$ and $\tw(H_2)\leq k$, if $G$ is contained in $H_1 \StrongProd H_2 \StrongProd K_c$, then $H$ is contained in both $H_1$ and $H_2$, implying $\tw(H_1)\geq s-3$ and  $\tw(H_2)\geq s-3$.
\end{cor}

And \cref{twtw-LowerBound} implies:

\begin{cor}
\label{twtw-LowerBound-Kst}
For any $c,t,s\in\NN$ with $t\geq s\geq 3$,  there exists a $K_{s,t}$-minor-free graph $G$ such that for any graphs $H_1$ and $H_2$, if $G$ is contained in $H_1 \StrongProd H_2 \StrongProd K_c$, then $\tw(H_1)\geq s-3$ or $\tw(H_2)\geq s-3$. That is, $\twtwa(\GG_{K_{s,t}})\geq s-3$. 
\end{cor}

\cref{KstMinorFreeThreeDimProduct,twtw-LowerBound-Kst} together show that $\twtwa(\GG_{K_{s,t}})\in\{s-3,s-2,s-1,s\}$ for $t\geq s\geq 3$. It is desirable to close the gaps in these bounds.

\begin{open}
What is $\twtwa(\GG_{K_{s,t}})$?
\end{open}

Here is a related question. 

\begin{open}
What is $\twtwa(\LL)$ where $\LL$ is the class of planar graphs? 
\end{open}

\cref{PGPST}(c) implies $\twtwa(\LL)\leq 3$, and in fact $\twtw(G)\leq 3$ for every planar graph $G$ by \cref{ltw-product} and \cref{ltw-planar}. For a lower bound, \cref{tw-lowerbound-tweaked} implies $\twtwa(\LL)\geq 2$ since every graph with treewidth 2 is planar.
 
\subsection{\boldmath Lower Bounds on \texorpdfstring{$\twpw$}{(tw⛝pw)}-Number} 

\cref{ForceH} leads to the following lower bound on $\twpw$-number. 

\begin{cor}
\label{twpw-LowerBound} 
$\twpwa(\AAA^{(1)}_1)$ is unbounded.
\end{cor}

\begin{proof}
Suppose for the sake of contradiction that 
$\twpwa(\AAA^{(1)}_1)\leq k$. 
That is, for some $c\in\NN$, every graph $G\in \AAA^{(1)}_1$ is contained in $H_1 \StrongProd H_2 \StrongProd K_c$, for some graphs $H_1$ and $H_2$ with $\pw(H_1)\leq k$ and $\tw(H_2)\leq k$. Consider this statement for the graph  $G$  from  \cref{ForceH}, applied with $H$ being the complete ternary tree of height $k+1$ and with $\ell=1=\tw(H)$. Hence $H$ is contained in $H_1$, implying $\pw(H)\leq\pw(H_1)\leq k$, which is a contradiction since $\pw(H)=k+1$ (see \citep{EST-IC94}).
\end{proof}

Every apex graph is $K_6$-minor-free. So every graph in $\AAA^{(1)}_1$ is $K_6$-minor-free by  \cref{Pasting}. 
Thus \cref{twpw-LowerBound} implies that 
$\twpwa(\GG_{K_6})$ is unbounded, 
and $\twpw$-number is unbounded on $K_6$-minor-free graphs. So `bounded $\twtw$-number' in  \cref{NoKtMinorProduct} cannot be replaced by `bounded $\twpw$-number'  even for $K_6$-minor-free graphs and even for 3-term products. This separates $\twpw$-number and $\TwIntPw$-number even in the class of $K_6$-minor-free graphs, since $\TwIntPw$-number is bounded for $K_6$-minor-free graphs by \cref{OddMinorFree-TreePathDecompsOrtho}. 

\begin{open}
   Which minor-closed classes $\GG$ have bounded $\twpw$-number?
\end{open}

By \cref{ApexMinorFreeProduct}, the answer is ``yes'' if some apex graph is not in $\GG$ (where the bounded pathwidth graph is a path). By \cref{twpw-LowerBound}, the answer is ``no'' if every graph in $\AAA_1^{(1)}$ is in $\GG$. Perhaps the answer is ``yes'' if and only if some graph in $\AAA^{(1)}_1$ is not in $\GG$.

\subsection{\boldmath Lower Bounds on \texorpdfstring{$\pwpw$}{(pw⛝pw)}-Number} 

We now establish lower bounds on $\pwpw$-number. Here the following graph parameter plays a role analogous to that of apex-number from \cref{LowerBounds-twtw}. For a graph $X$, let
\[\mathdefn{\fvn(X)}:=\min\{|S|:S\subseteq V(X), X-S\text { is acyclic}\}.\]
Here $S$ is called a \defn{feedback vertex set}, and $
\fvn(X)$ is called the \defn{feedback vertex number} of $X$. Let \defn{$\FF_k$} be the class of graphs $X$ with $\fvn(X)\leq k$. Note that $\FF_0$ is the class of forests, and $\FF_1$ is the class of apex-forests.

A proof similar to that of \cref{ForceH} (replacing the grid subgraph by a complete ternary tree) shows the following:

\begin{lem}
\label{pw-ForceH} 
For any $k,\ell,c\in\NN$, for any graph $H$ with $\tw(H)\leq \ell$, there exists a graph $G\in\FF^{(\ell)}_\ell$ such that for any graphs $H_1$ and $H_2$, if $\pw(H_2)\leq k$ and $G$ is contained in $H_1 \StrongProd H_2 \StrongProd K_c$, then $H$ is contained in $H_1$. 
\end{lem}

\begin{proof}
%Since $\pw(H_2\StrongProd K_c)\leq (\pw(H_2)+1)c-1$, the result follows from the $c=1$ case by adjusting the value of $k$. Now assume that $c=1$. 
Let $h:=|V(H)|$. We may assume that $H$ is edge-maximal subject to $\tw(H)\leq \ell$. So $H$ is chordal with no $K_{\ell+2}$ subgraph. Let $(v_1,\dots,v_h)$ be an elimination ordering of $H$. That is, for each $i\in\{1,\dots,h\}$, if $N_i:= N_H(v_i) \cap\{v_1,\dots,v_{i-1}\}$ then $N_i$ is a clique of size at most $\ell$. 

Let $n:= (h-1)(k+1)c$. Let $T$ be the ternary tree of height $n$, which has pathwidth $n$  (see \citep{EST-IC94}). Recursively define graphs $G_1,G_2,\dots,G_h\in\FF^{(\ell)}_\ell$ as follows. Let $G_1:=K_1$. For $i\geq 2$, let $G_i$ be obtained from $G_{i-1}$ by adding, for each clique $S$ in $G_{i-1}$ of at most $\ell$ vertices, a copy $D^i_S$ of $T$ complete to $S$ and disjoint from $G_{i-1}$. Note that 
$G_i[D^i_S\cup S]$ has at most $\ell$ vertices (namely, $S$)  whose deletion leaves a tree. So $\fvn(G_i[D^i_S\cup S])\leq\ell$. Thus, the addition of $D^i_S$ can be described by pasting a graph in $\FF^{(\ell)}_\ell$ and a graph in $\FF_\ell$ on a clique of at most $\ell$ vertices. Therefore, $G_i\in \FF_\ell^{(\ell)}$.

We consider $G_1\subseteq G_2\subseteq\dots\subseteq G_h$. Let $G:=G_h$. Suppose that $G$ is contained in $H_1 \StrongProd H_2 \StrongProd K_c$, where  $\pw(H_2)\leq k$. For ease of notation, assume $G\subseteq H_1\StrongProd H_2 \StrongProd K_c$. Let $p$ be the maximum integer such that $H[\{v_1,\dots,v_p\}]$ is contained in $G$, where for $i\in\{1,\dots,p\}$, vertex $v_i$ is mapped to $(x_i,y_i,z_i)\in V(G_i)$, and $x_1,\dots,x_p$ are distinct vertices of $H_1$. 

Suppose for the sake of contradiction
that $p\leq h-1$. Note that  $v_1,\dots,v_p$ are mapped to vertices in $G_p$, and thus $N_{p+1}$ is mapped to a clique in $G_p$. By construction, $D^{p+1}_{N_{p+1}}$ is a copy of $T$ complete to $N_{p+1}$ in $G_{p+1}$. If there is a vertex $(x,y,z)$ in the image of  $D^{p+1}_{N_{p+1}}$ with $x\not\in \{x_1,\dots,x_p\}$, then mapping $v_{p+1}$ to $(x,y,z)$, we find that 
$H[\{v_1,\dots,v_{p+1}\}]$ is contained in $G$, where for $i\in\{1,\dots,p+1\}$, vertex $v_i$ is mapped to $(x_i,y_i,z_i)\in V(H_i)$, and $x_1,\dots,x_{p+1}$ are distinct vertices of $H_1$, thus contradicting the maximality of $p$. Hence $D^{p+1}_{N_{p+1}}$ is mapped to $\{x_1,\dots,x_p\} \times V(H_2) \times K_c$ in $G$, implying  $$(h-1)(k+1)c=n= \pw(D^{p+1}_{N_{p+1}})\leq p (\pw(H_2)+1)c-1 \leq (h-1)(k+1)c-1,$$
which is contradiction. Therefore $p=h$, implying $H$ is contained in $H_1$, as desired. 
\end{proof}

We now give some corollaries of \cref{pw-ForceH}. 

\begin{cor}
\label{pw-ForceH-Minor} 
For any $k,c,\ell\in\mathbb{N}$, for any $\ell$-connected graph $X$ with $\fvn(X)=\ell$, for any graph $H$ with $\tw(H)\leq \ell-1$, there exists an $X$-minor-free graph $G$ such that for any graphs $H_1$ and $H_2$, if $\pw(H_2)\leq k$ and $G$ is contained in $H_1 \StrongProd H_2\StrongProd K_c$, then $H$ is contained in $H_1$. 
\end{cor}

\begin{proof}
Since $\fvn$ is minor-monotone and $\fvn(X)=\ell$, every graph in $\FF_{\ell-1}$ is $X$-minor-free. By \cref{Pasting}, every  graph in $\FF^{(\ell-1)}_{\ell-1}$ is $X$-minor-free. The result thus follows from \cref{pw-ForceH}.
\end{proof}

\begin{cor}
\label{pwpw-LowerBound} 
$\pwpwa(\FF^{(1)}_1)$ is unbounded. 
\end{cor}

\begin{proof}
Suppose for the sake of contradiction that 
$\pwpwa(\FF^{(1)}_1)\leq k$. That is, for some $c\in\NN$, every graph $G\in \FF^{(1)}_1$ is contained in $H_1 \StrongProd H_2 \StrongProd K_c$, for some graphs $H_1$ and $H_2$ with $\pw(H_1)\leq k$ and $\pw(H_2)\leq k$. Apply this result where  $G$ is the graph from  \cref{pw-ForceH}, applied with $H$ being the complete ternary tree of height $k+1$ and with $\ell=1=\tw(H)$. Hence $H$ is contained in $H_1$, implying $\pw(H)\leq\pw(H_1)\leq k$, which is a contradiction since $\pw(H)=k+1$ (see \citep{EST-IC94}).
\end{proof}

Every graph in $\FF_1$ is $K_4$-minor-free. So every graph in $\FF_1^{(1)}$ is $K_4$-minor-free by \cref{Pasting}. 
Thus \cref{pwpw-LowerBound} implies that 
$\pwpwa(\GG_{K_4})$ is unbounded, and $\pwpw$-number is unbounded on $K_4$-minor-free graphs. 
\citet[Theorem~29]{DJMNW18} showed the stronger result (by \cref{PwIntPW-pwpw}(c)) that $K_4$-minor-free graphs have unbounded $\PwIntPw$-number. 

\begin{open}
Which minor-closed classes $\GG$ have bounded $\pwpw$-number?
\end{open}

As shown in \cref{NoApexForest-pwpw}, the answer is ``yes'' if some apex-forest is not in $\GG$. By \cref{pwpw-LowerBound}, the answer is ``no'' if every graph in  $\FF_1^{(1)}$ is in $\GG$. Perhaps the answer is ``yes'' if and only if some graph in $\FF_1^{(1)}$ is not in $\GG$.

%%%%%%%%%%%%%%%%%%%%%%%%%%
\subsection{\boldmath Separating Tree-Treewidth and \texorpdfstring{$\TwIntTw$}{(tw∩tw)}-Number}

The following result separates tree-treewidth and $\TwIntTw$-number, and thus separates tree-treewidth and $\twtw$-number.

\begin{prop}
\label{Separating}
For any $c\in\NN$, there exists a graph $G$ with 
$\ttd(G)\leq 4$ and $\ttw(G)\leq \tpw(G)\leq 3$ and $\TwIntTw(G)>c$.
\end{prop}

\begin{proof}
Let $S$ be a set of $c+1$ isolated vertices. Let $G$ be the graph obtained from $S$ by adding, for each pair of distinct vertices $v,w\in S$, a copy of the $(c \times c)$-grid graph $G_{v,w}$ complete to $\{v,w\}$.

We first show that $\ttd(G) \leq 4$. 
By \cref{Bipartite-ttd-ttw-tpw}, since $G_{v,w}$ is bipartite, it has a path-decomposition where each bag induces a star plus some isolated vertices, which has treedepth at most 2. Thus, the graph obtained by adding two vertices $v,w$ complete to $G_{v,w}$ has a path-decomposition with $v$ and $w$ in every bag, such that the subgraph induced by each bag has treedepth at most 4. Therefore, $G$ has a tree-decomposition indexed by a subdivided star, where the root bag equals $S$, and the subgraph induced by any of the other bags has treedepth at most $4$. Since $\td(G[S])=1$, we have $\ttd(G) \leq 4$, which implies $\ttw(G)\leq \tpw(G)\leq 3$.

We now show that $\TwIntTw(G)>c$. Suppose to the contrary that $G$ has two tree-decompositions $\TT_1=(B^1_x:x\in V(T_1))$ and $\TT_2=(B^2_x:x\in V(T_2))$, such that $|B^1_x\cap B^2_y|\leq c$ for all $x\in V(T_1)$ and $y\in V(T_2)$.  Suppose there exist $v,w\in S$  in no common bag of $\TT_1$. Let $x$ and $y$ be the closest nodes in $T_1$ with $v\in B^1_x$ and $w\in B^1_y$. So $x\neq y$. Let $z$ be any node in the $xy$-path in $T_1$. Since every vertex of $G_{v,w}$ is a common neighbour of $v$ and $w$, $V(G_{v,w})  \subseteq B^1_z$. Hence $\tw(G[B^1_z]) \geq \tw(G_{v,w})=c$. On the other hand, $(B^1_z \cap B^2_y:y\in V(T_2))$ is a tree-decomposition of $G[B^1_z]$ with width at most $c-1$. This contradiction shows that each pair $v,w\in S$  are in a common bag of $\TT_1$. By the Helly Property, there exists $x\in V(T_1)$ such that $S\subseteq B^1_x$. By the same argument, 
there exists $y\in V(T_2)$ such that $S\subseteq B^2_y$. Hence $|B^1_x\cap B^2_y|\geq|S|=c+1$, which contradicts the $c$-orthogonality assumption. 
\end{proof}

%%%%%%%%%%%%%%%%%%%%%%%%%%%%%%%%%%%
\subsection{\boldmath Triangulated 3-Dimensional Grids}

The main result of this section, \cref{3Dgrid} below,   shows that triangulated 3-dimensional grids have unbounded $\TwIntTw$-number and hence unbounded $\twtw$-number. The proof is topological and implements the strategy of \citet{EHMNSW24}, who used similar methods to show that such graphs have unbounded stack-number. One major difference is that \citet{EHMNSW24} used Gromov's topological overlap theorem, while we use a result of \citet{Norin24} instead, since to the best of our knowledge Gromov's result does not directly extend to contractible simplicial complexes.

Specifically, we need the following version of the result from~\cite{Norin24} specialized to $2$-dimensional simplicial complexes. Here, a \defn{$d$-dimensional simplex} $\Delta$ is a convex hull of a set $S$ of $d+1$ affinely independent points in $\RR^{N}$ for some $N \geq d$. 
An $i$-\defn{face} of $\Delta$  is the convex hull of a subset of $S$ of size $i+1$.
A \defn{simplicial complex} $X$ is a collection of simplices, closed under taking faces, such that any non-empty  intersection of two simplices in $X$ is a face of both.(Informally speaking, a simplicial complex is a collection of simplices glued along faces.)\ Let $X^{(i)}$ denote the set of $i$-dimensional simplices of  $X$. A simplicial complex is \defn{contractible} if it is homotopically equivalent to a point; that is, it can be continuously shrunk to a point. A simplicial complex $X$ is \defn{$d$-dimensional} if each simplex in $X$ has dimension at most $d$. A $d$-dimensional simplicial complex $X$ is \defn{pure} if every simplex in $X$ is a face of some $d$-dimensional simplex in $X$.   A ($2$-dimensional) \defn{bramble} $\mc{B}$ on a  $2$-dimensional simplicial complex $X$ is a collection of simply-connected subcomplexes of $X$ closed under taking unions.  A bramble is \defn{pure} if every element of it is pure.

\begin{thm}[{\citep[Theorem~1]{Norin24}}]\label{t:bramble}
Let $X$ and $Y$ be $2$-dimensional simplicial complexes, such that $Y$ is contractible. For any bramble $\mc{B}$ on $X$ and any continuous function $f: X \to Y$,
$$\bigcap_{Z \in \mc{B}} f(Z) \neq \emptyset.$$  
\end{thm}

We use \cref{t:bramble} in the next lemma, which lower bounds $\TwIntTw(G)$ in terms of the size of a hitting set of a bramble on a triangle complex of $G$.  Let \defn{$\hat{G}$} denote any $1$-dimensional simplicial complex with $0$-faces corresponding to vertices of $G$ and $1$-faces corresponding to edges of $G$. \footnote{Any two such complexes are homeomorphic, and since we are only concerned with their topological invariants, a particular choice of a complex denoted by $\hat{G}$ does not matter.} Let \defn{$T(G)$} denote the  set of triangles (cliques of size 3) of $G$, and let \defn{$\widehat{T(G)}$} denote a $2$-dimensional simplicial complex with $0$-faces corresponding to vertices of $G$, $1$-faces corresponding to $E(G)$, and  $2$-faces corresponding to $T(G)$. \footnote{Thus $\hat{G}$ is (homeomorphic to) the $1$-skeleton of $\widehat{T(G)}$.}

\begin{lem}\label{l:brambleToTwIntTw} 
Let $G$ be a graph and  let $\mc{B}$ be a pure bramble on $\widehat{T(G)}$. Then there exists $W  \subseteq V(G)$ such that 
$|W| \leq  \TwIntTw(G)$ and $W \cap B^{(0)} \neq \emptyset$ for every $B \in \mc{B}$. 
\end{lem}

\begin{proof} Let $w:= \TwIntTw(G)$. By the definition of $\TwIntTw(G)$, for $i\in\{1,2\}$ there exists a tree $T_i$ and a $T_i$-decomposition $(B_x^i:x\in V(T_i))$  of $G$ such that $|B_x^1\cap B_y^2|\leq w$ for every $x\in V(T_1)$ and $y\in V(T_2)$.  We use these tree-decompositions to construct a continuous map $\widehat{T(G)} \to \widehat{T}_1 \times \widehat{T}_2$ to which we will apply \cref{t:bramble}. Note that for every clique $C$ of $G$, and for each $i\in\{1,2\}$, there exists a node $x\in V(T_i)$ such that $C\subseteq B^i_x$.

We will work with the first barycentric subdivision $X$ of $\widehat{T(G)}$ which has a different structure to $\widehat{T(G)}$ as a simplicial complex, but is equal to $\widehat{T(G)}$ as a topological space. Moreover, every  subcomplex of $\widehat{T(G)}$ also corresponds to a subcomplex of $X$, so $\mc{B}$ can be considered as a bramble on $X$. 
Explicitly we consider a $2$-dimensional simplicial complex $X$ such that the set $X^{(0)}$ of vertices of $X$  is $V(G) \cup E(G) \cup T(G)$, the set $X^{(1)}$ of edges  of $X$ correspond to pairs $\{S_1,S_2\} \subseteq X^{(0)}$ such that $S_1 \subsetneqq S_2$ (that is, $S_1=\{v\}$ and $S_2 \in E(G)$ is an edge with an end $v$, or $|S_1| \leq 2$ and $S_2 \in T(G)$ is a triangle including all elements of $S_1$), and the set of $2$-faces $X^{(2)}$ of $X$ corresponds to triples $\{\{v\}, \{u,v\}, \{u,v,w\}\}$ such that $v \in V(G),$ $uv \in E(G)$, $\{u,v,w\} \in T(G)$. Let $\brm{v}(\Delta) = v$ for every $\Delta \in X^{(2)}$ as above; that is,   $\brm{v}(\Delta) $ is the unique vertex of $G$ that is contained in every vertex of $\Delta$.

Our next goal is to define a continuous map $f_i: X \to \widehat{T_i}$ for $i\in\{1,2\}$. We start by defining $f_i$ on $X^{(0)}$. Every $S \in X^{(0)}$ is a clique in $G$ , so there exists a node $x\in V(T_i)$ such that $S\subseteq B^i_{x}$; set $f_i(S):=x$. We now extend this map to $1$-faces and $2$-faces of $X$.  For a $1$-face $\Delta$ corresponding to $\{S_1,S_2\} \in X^{(1)}$ such that $S_1 \subsetneqq S_2$ we have $S_1 \subseteq B_{f_i(S_1)}$ and $S_1 \subseteq S_2 \subseteq B^i_{f_i(S_2)}$. It follows that $S_1 \subseteq B^i_x$ for all $x \in V(T_i)$ on a path from $f_i(S_1)$ to $f_i(S_2)$ in $T_i$. Map $\Delta$ to this path. Similarly, for $\Delta \in X^{(2)}$ the boundary of $\Delta$ is mapped to a contractible subcomplex of $\widehat{T_i}$ such that $\brm{v}(\Delta) \in B^i_x$ for every vertex $x$ of this subcomplex. Extend $f_i$ from the boundary of $\Delta$ to its interior to preserve this property. To summarize: $f_i: X \to \widehat{T_i}$ is a continuous  map such that:
\begin{itemize} 
\item For  every $\Delta\in X^{(2)}$ every  $p \in f_i(\Delta)$ and every minimal face $\Delta'$ of $\widehat{T_i}$ containing $p$,  $\brm{v}(\Delta) \in B^i_x$ for every vertex $x$ of $\Delta'$.
\end{itemize}

Let $Y := \hat{T_1} \times \hat{T_2}$. It is a $2$-dimensional cubical complex which can be considered as a simplicial complex by triangulating its $2$-faces. Since $Y$ is a product of two contractible spaces, it is contractible. 
Let $f: X \to Y$ be a continuous map defined by $f:=f_1 \times f_2$; that is, $f(p):=(f_1(p),f_2(p))$ for every $p \in X$. By \cref{t:bramble}, there exists  $ p  \in \bigcap_{Z \in \mc{B}} f(Z).$ Let $p=(p_1,p_2)$  where $p_i \in \hat{T_i}$, and let $x_i \in V(T_i)$  be a vertex of a minimal face of $\hat{T}_i$ containing $p_i$. 

Since $\mc{B}$ is pure for every $B \in \mc{B}$, there exists $\Delta_B \in B^{(2)} \subseteq X^{(2)}$ such that $p \in f(\Delta_B)$; that is, $p_i \in f_i(\Delta_B)$. By the property of $f_i$ noted above, $\brm{v}(\Delta_B) \in B^{i}_{x_i}$ for every $B \in \mc{B}$. Let $W := 
\{\brm{v}(\Delta_B)\}_{B \in \mc{B}}$. Then $W  \subseteq B^{1}_{x_1} \cap   B^{2}_{x_2}$ and so $|W| \leq w = \TwIntTw(G)$ by the choice of the tree-decompositions $(B_x^i:x\in V(T_i))$.  Moreover, $\brm{v}(\Delta_B) \in B^{(0)}$ for every  $B \in \mc{B}$. Thus $W$ is as desired.  
\end{proof}

Finally, we apply \cref{l:brambleToTwIntTw} to derive the aforementioned lower bound on the $\TwIntTw$-number of triangulations of 3-dimensional grids. For graphs $G_1$ and $G_2$, a \defn{triangulation} of $G_1\CartProd G_2$ is any graph obtained from $G_1 \CartProd G_2$ by adding the edge $(x,y)(x',y')$ or $(x,y')(x',y)$ for each $xx'\in E(G_1)$ and $yy'\in E(G_2)$. For example, $n\times n$ Hex graphs are the triangulations of $P_n\CartProd P_n$. 
As illustrated in \cref{TriangulatedGrid}, a \defn{triangulation} of $G_1 \CartProd G_2 \CartProd G_3$ is any graph obtained by triangulating all subgraphs induced by sets of the form $\{v_1\} \times V(G_2) \times V(G_3)$, $V(G_1) \times \{v_2\} \times V(G_3)$ and $V(G_1)\times V(G_2) \times \{v_3\}$ with $v_i \in V(G_i)$. 

\begin{figure}[ht]
\centering
\includegraphics{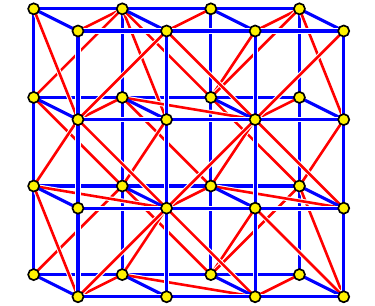}
\caption{\label{TriangulatedGrid} Triangulation of $P_4\CartProd P_4\CartProd P_2$.}
\end{figure}

\begin{thm}\label{3Dgrid}
For any connected graphs $H_1$, $H_2$ and $H_3$, and for any triangulation $G$ of  $H_1 \CartProd H_2 \CartProd H_3$, 
\[ \TwIntTw(G) \geq \min\{|V(H_1)|,|V(H_2)|,|V(H_3)|\}.\]
\end{thm}

\begin{proof} 
Let $n:=\min\{|V(H_1)|,|V(H_2)|,|V(H_3)|\}$. 
By taking a suitable subtree, we may assume without loss of generality  that  $H_1$, $H_2$ and $H_3$ are trees, each with $n$ vertices. We apply \cref{l:brambleToTwIntTw} to the following bramble $\mc{B}$ borrowed from \cite{EHMNSW24}. 
For $u \in V(H_1)$, let $G(u)$ be the subgraph of $G$ induced by all the vertices of $G$ with the first coordinate $u$. That is, $V(G(u))=\{(u,u_2,u_3) : u_2 \in V(H_2),u_3 \in V(H_3)\}$. So $G(u)$ is isomorphic to a triangulation of $H_2 \CartProd H_3$. Let $X(u) = \widehat{T(G(u))}$ be the subcomplex of $\widehat{T(G)}$ induced by the vertices of $G(u)$. 
As a topological space, $X(u)$ is homeomorphic to $\widehat{H}_2 \times \widehat{H}_3$. Define $G(u)$ and $X(u)$ for $u \in V(H_2) \cup V(H_3)$ analogously.  Let $\mc{B} := \{ X(u_1) \cup X(u_2) \cup X(u_3) : (u_1,u_2,u_3) \in V(G) \}$.  It is shown in~\cite[Proof of Lemma~1.4]{EHMNSW24} that $\mc{B}$ is indeed a bramble.\footnote{The definition of a bramble in~\cite{EHMNSW24} is slightly different, but in~\cite[Proof of Lemma~1.4]{EHMNSW24} it is shown that $\mc{B}$ satisfies the conditions of~\cite[Corollary 3.5]{EHMNSW24} and it is not hard to see that this in turn implies that  $\mc{B}$ is a bramble under our definition.}

By  \cref{l:brambleToTwIntTw} there exists $W \subseteq V(G)$ such that $|W| \leq  \TwIntTw(G)$ and $W \cap B^{(0)} \neq \emptyset$ for every $B \in \mc{B}$. By the second condition, there exists $i \in \{1,2,3\}$ such that $W \cap V(G(u)) \neq \emptyset$ for every $ u  \in V(H_i)$. Since the graphs $\{G(u)\}_{u \in V(H_i)}$ are pairwise vertex-disjoint, it follows that $|W| \geq |V(H_i)|=n$, implying the theorem.
\end{proof}

Note that the bound in \cref{3Dgrid} approximates $\TwIntTw(G)$ within a constant factor whenever $H_1$, $H_2$ and $H_3$ have bounded treewidth. To see this, assume that $\min\{|V(H_1)|,|V(H_2)|,|V(H_3)|\}=|V(H_3)|$. 
Let $G:=H_1\StrongProd H_2 \StrongProd H_3$, which 
contains every triangulation of  $H_1 \CartProd H_2 \CartProd H_3$. Let $(B_x:x\in V(T_1))$ be a tree-decomposition of $H_1$ with width $\tw(H_1)$. Let $(C_y:y\in V(T_2))$ be a tree-decomposition of $H_2$ with width $\tw(H_2)$. For each $x\in V(T_1)$, let $B'_x:=B_x\times V(H_2)\times V(H_3)$. For each $y\in V(T_2)$, let $C'_y:=V(H_1) \times C_y\times V(H_3)$. Observe that  $(B'_x:x\in V(T_1))$ and $(C'_y:y\in V(T_2))$ are tree-decompositions of $G$. Moreover, 
$B'_x\cap C'_y = B_x \times C_y \times V(H_3)$ for each $x\in V(T_1)$ and $y\in V(T_2)$. Thus $\TwIntTw(G)\leq (\tw(H_1)+1)(\tw(H_2)+1)|V(H_3)|$. 

The following connection further motivates 
\cref{3Dgrid}. A hereditary graph class $\GG$ \defn{admits strongly sublinear separators} if there exist $c,\epsilon>0$ such that every graph $G\in\GG$ has a balanced separator of order at most $c|V(G)|^{1-\epsilon}$. The class $\GG$ of all subgraphs of 3-dimensional triangulated grids admits strongly sublinear separators with $\epsilon=\frac13$ (see \citep{MTTV97}). In this sense, $\GG$ is a `well-behaved' class. However, by \cref{3Dgrid}, $\GG$ has unbounded $\TwIntTw$-number and unbounded $\twtw$-number by \cref{TwIntPW-twpw}.

\begin{open}
Do 3-dimensional triangulated grids have bounded $\ttw$?
\end{open}

\citet[Section~6]{EHMNSW24} construct a family of graphs $G_n$ with maximum degree $7$ (based on tesselations of $\bb{R}^3$ with truncated octahedra) such that $\widehat{T(G_n)}$ admits brambles very similar to those used in the proof of \cref{3Dgrid}. More explicitly, these brambles have the form $\mc{B}= \{ X^{i_1}_1 \cup X^{i_2}_2 \cup X^{i_3}_3 : 1 \leq i_1,i_2,i_3 \leq n \}$, where $X^{i}_j$ is a subcomplex of   $\widehat{T(G_n)}$  for each $i\in\{1,\dots,n\}$ and $j \in \{1,2,3\}$. Moreover, the subcomplexes $\{X^{i}_j\}_{i\in\{1,\dots,n\}}$ are pairwise vertex-disjoint for  every $j \in \{1,2,3\}$. Thus, as in the proof of \cref{3Dgrid}, it follows that $\TwIntTw(G_n) \geq n,$ implying the following. 

\begin{thm}
\label{MaxDeg7}
There exists a family of graphs with maximum degree $7$ with unbounded  $\TwIntTw$-number and thus with unbounded $\twtw$-number. 
\end{thm}

By \cref{Delta5}, graphs with maximum degree $5$ have bounded $\twtw$-number and thus have bounded $\TwIntTw$-number by \cref{TwIntTW-twtw}. 

\begin{open}
Is $\TwIntTw$-number or $\twtw$-number bounded for graphs with  maximum degree $6$?
\end{open}

%%%%%%%%%%%%%%%%%%%%%%%%%%%%%%%%%%%%%%%%%%%%%%%%%%%%
\subsection{\boldmath Random Graphs}

The main result of this section shows that random regular graphs have large tree-treewidth. To show this, we first prove the following bound on the treewidth of random regular graphs. It is a direct consequence of the Expander Mixing Lemma of \citet{AC98} combined with the upper bound on the eigenvalues of random regular graphs due to \citet{Friedman08} (also see \citep{Bord20}).

\begin{lem}
\label{Isoperimetry} 
Fix an even positive integer $d$, and let $G$ be a random $d$-regular graph on $n$ vertices. Then asymptotically almost surely, for all pairs of disjoint $S,T \subseteq V(G)$ such that $|S|,|T| \geq  2n/\sqrt{d}$ there exist an edge of $G$ with one end in $S$ and the other in $T$.
\end{lem}

\begin{proof}
Let $d=\lambda_1 \geq \lambda_2 \geq \cdots \geq \lambda_{n}$ be the eigenvalues of the adjacency matrix of $G$, and let $\lambda :=\max\{\lambda_2,|\lambda_n|\}$. \citet{Friedman08} famously proved that asymptotically almost surely $\lambda < 2\sqrt{d-1}+\eps$ for any $\eps > 0$. In particular, $\lambda < 2\sqrt{d}$ asymptotically almost surely. The Expander Mixing Lemma of \citet{AC98} states that for any pair of disjoint $S,T \subseteq V(G)$,
$$\left|e(S,T) - \frac{d|S||T|}{n}\right| \leq\lambda\sqrt{|S||T|}.$$
In particular, the conclusion of the lemma holds as long as $\frac{d|S||T|}{n} > \lambda\sqrt{|S||T|}$. Since this inequality holds whenever $|S|, |T| \geq 2n/\sqrt{d}$  and $\lambda < 2\sqrt{d}$,  the lemma follows.
 \end{proof}

\begin{cor}
\label{TreewidthRandomRegular}
Fix an even positive integer $d$, and let $G$ be a random $d$-regular graph on $n$ vertices. Then asymptotically almost surely,  
$\tw(G)+1 \geq (1-\frac{6}{\sqrt{d}})n$. 
\end{cor}

\begin{proof}
By \cref{Isoperimetry} we may assume that for all pairs of disjoint $S,T \subseteq V(G)$ such that $|S|,|T| \geq  \frac{2n}{\sqrt{d}}$ there exist an edge of $G$ with one end in $S$ and another in $T$. By the separator lemma of \citet[(2.5)]{RS-II}, there exists $X \subseteq V(G)$ with $|X| \leq \tw(G)+1$ such that every component of $G-X$ has at most $\frac{n-|X|}{2}$ vertices. Grouping the components of $G-X$ we obtain disjoint  $S,T \subseteq V(G)$ such that $|S|,|T| \geq \frac{n-|X|}{3}$ and no edge of $G$ has an end in $S$ and another in $T$. By our earlier assumption, $\frac{2n}{\sqrt{d}} > \frac{n-|X|}{3} \geq \frac{n-(\tw(G)+1)}{3}$; that is, $\tw(G)+1 \geq (1-\frac{6}{\sqrt{d}})n$, as desired. 
\end{proof}

Results similar to \cref{Isoperimetry,TreewidthRandomRegular} are known in the literature~\citep{BFK16,Shang22,WLCX11,Gao12}. We include these proofs for precision and completeness. 

\cref{TreewidthRandomRegular} says that random $d$-regular $n$-vertex graphs have treewidth very close to $n$. On the other hand, we now show that graphs with small tree-treewidth have treewidth not much more than $\frac{n}{2}$.

\begin{lem}
\label{ttwUpperBound}
For every graph $G$,
    $$\tw(G) + 1\leq 
    \ceil{\half(\ttw(G) + |V(G)|-1)}.$$
\end{lem}

\begin{proof}
Let $k:=\ceil{\half(\ttw(G) + |V(G)|-1)}$. Let $(B_x:x\in V(T))$ be a tree-decomposition of $G$ such that $\tw(G[B_x])\leq\ttw(G)$ for each node $x\in V(T)$. If $|B_x|\leq k$ for each $x\in V(T)$, then we are done. Now assume that $|B_x|\geq k+1$ for some $x\in V(T)$. To construct a tree-decomposition of $G$, start with a tree-decomposition of $G[B_x]$ with width $\tw(G[B_x])$, and add $V(G)\setminus B_x$ to every bag. Thus $\tw(G) \leq \tw(G[B_x]) + |V(G)|-|B_x| \leq\ttw(G) + |V(G)|-(k+1) \leq k$, as desired. 
\end{proof}

\cref{ttwUpperBound} and \cref{ttwProduct}(a) imply:

\begin{cor}
\label{TreewidthUpperBound}
For all graphs $H_1$ and $H_2$, for every $n$-vertex subgraph $G$ of $H_1\StrongProd H_2$, 
$$\tw(G)+1 \leq \ceil{\half(n+(\tw(H_1)+1)(\tw(H_2)+1))}-1.$$
\end{cor}

While the upper bounds in \cref{ttwUpperBound,TreewidthUpperBound} may appear to be naive, we now show they are tight up to the second-order additive term. Let $p,q,n\in\NN$ such that $n-pq$ is even and $m:= \half(n-pq) \geq pq$. Let $H_1:=K_{p,m}$ and $H_2:= K_{q,m}$. Note that $\tw(H_1) \leq p$ and $\tw(H_2) \leq q$. Let $G:= K_{pq,m,m}$ which has $n$ vertices. Note that $G \subsetsim H_1\StrongProd H_2$ and $\tw(G) \geq \delta(G) = m+pq = \half (n+pq) \geq \half(n + \tw(H_1) \tw(H_2))$. This shows that \cref{TreewidthUpperBound}, and thus \cref{ttwUpperBound}, is tight up to an additive $O(\sqrt{|V(G)|})$ term.

We now combine the above results to conclude that a random $d$-regular graph has large $\ttw$.

\begin{thm}
\label{ttw-RandomRegularGraph}
\label{ProductRandomRegular}
Fix an even integer $d\geq 146$, and let $G$ be a random $d$-regular graph on $n$ vertices. Then asymptotically almost surely,  
$$\TwIntTw(G) \geq \ttw(G)+1 \geq (1 -\tfrac{12}{\sqrt{d}})n +1$$ 
and $$\twtw(G) \geq \sqrt{(1-\tfrac{12}{\sqrt{d}})n+1}-1.$$
\end{thm}

\begin{proof}
By \cref{TreewidthRandomRegular,ttwUpperBound},
\begin{equation*}
    (1-\tfrac{6}{\sqrt{d}})n \leq \tw(G)+1 \leq 
    \ceil{\half(\ttw(G) + n-1)} \leq
    \half(\ttw(G) + n)    .
\end{equation*}
By \cref{ttw-otw},
\begin{equation*}
\TwIntTw(G) \geq
\ttw(G)+1 \geq 
(1 -\tfrac{12}{\sqrt{d}})n +1.
\end{equation*}
By \cref{ttw-twtw},
$$  ( \twtw(G)+1)^2 \geq  \ttw(G)+1  \geq
(1 -\tfrac{12}{\sqrt{d}})n +1.$$ 
The result follows.
\end{proof}

In the above results, we assume that $d$ is even  simply so that $d$-regular $n$-vertex graphs exist for all $n\gg d$. Analogous results hold for odd $d$ and even $n\gg d$. 

We finish with one more open problem. \cref{OddMinorFree-TreePathDecompsOrtho,ttw-otw}  imply that any proper minor-closed class has bounded tree-treewidth. It would be interesting to prove this without using the Graph Minor Structure Theorem and with better bounds. In particular, what is the maximum of $\ttw(G)$ for a $K_t$-minor-free graph $G$? There is a lower bound of $t-2$ since $K_{t-1}$ is $K_t$-minor-free and $\ttw(K_{t-1})=t-2$. In the spirit of Hadwiger's Conjecture, it is tempting to conjecture that $\ttw(G)\leq t-2$ for every $K_t$-minor-free graph $G$. This is true for $t\leq 5$ by \cref{ttw-tpw-K5MinorFree} and since $\ttw(G)\leq\tw(G)\leq t-2$ for $t\in\{2,3,4\}$. However, random graphs provide a counterexample. By \cref{ttw-RandomRegularGraph}, for large $n$, there is a $256$-regular $n$-vertex graph $G$ with $\ttw(G)\geq \frac{n}{4}$. On the other hand, if $K_t$ is a minor of $G$, then $\frac{t^2}{4}\leq \binom{t}{2}\leq |E(G)|\leq 128 n$, implying $t\leq \sqrt{512 n}$. 
Let $t:= \floor{\sqrt{600 n}}$. 
Hence $G$ is $K_t$-minor-free, and 
$\ttw(G)\geq \frac{n}{4} 
\geq \frac{t^{2}}{2400}$.
In short, there are $K_t$-minor-free graphs $G$ with $\ttw(G)\in\Omega(t^2)$. The following natural question arises: 
\begin{open}
Do $K_t$-minor-free graphs $G$ have $\ttw(G)\in O(t^2)$? 
\end{open}

\subsection*{Acknowledgements} 

This research was initiated at the 2024 Graph Theory Workshop held at the Bellairs Research Institute. Thanks to the organisers and participants for creating a healthy working environment. Special thanks to Raphael Steiner for asking a good question and for helpful discussions. 

{
\fontsize{10pt}{11pt}
\selectfont
\bibliographystyle{DavidNatbibStyle}
\bibliography{DavidBibliography}

\def\soft#1{\leavevmode\setbox0=\hbox{h}\dimen7=\ht0\advance \dimen7 by-1ex\relax\if t#1\relax\rlap{\raise.6\dimen7 \hbox{\kern.3ex\char'47}}#1\relax\else\if T#1\relax \rlap{\raise.5\dimen7\hbox{\kern1.3ex\char'47}}#1\relax \else\if d#1\relax\rlap{\raise.5\dimen7\hbox{\kern.9ex \char'47}}#1\relax\else\if D#1\relax\rlap{\raise.5\dimen7 \hbox{\kern1.4ex\char'47}}#1\relax\else\if l#1\relax \rlap{\raise.5\dimen7\hbox{\kern.4ex\char'47}}#1\relax \else\if L#1\relax\rlap{\raise.5\dimen7\hbox{\kern.7ex \char'47}}#1\relax\else\message{accent \string\soft \space #1 not defined!}#1\relax\fi\fi\fi\fi\fi\fi}
\begin{thebibliography}{88}
\providecommand{\natexlab}[1]{#1}
\providecommand{\msn}[1]{MR:\,\href{http://www.ams.org/mathscinet-getitem?mr=MR{#1}}{#1}}
\providecommand{\ZBL}[1]{Zbl:\,\href{https://www.zentralblatt-math.org/zmath/en/search/?q=an:#1}{#1}}
\providecommand{\url}[1]{\texttt{#1}}
\providecommand{\urlprefix}{}
\expandafter\ifx\csname urlstyle\endcsname\relax
  \providecommand{\doi}[1]{doi:\discretionary{}{}{}#1}\else
  \providecommand{\doi}{doi:\discretionary{}{}{}\begingroup \urlstyle{rm}\Url}\fi

\bibitem[{Abrishami et~al.(2024)Abrishami, Alecu, Chudnovsky, Hajebi, Spirkl, and Vušković}]{AACHSV24}
\textsc{Tara Abrishami, Bogdan Alecu, Maria Chudnovsky, Sepehr Hajebi, Sophie Spirkl, and Kristina Vušković}.
\newblock \href{https://doi.org/10.1002/jgt.23055}{Induced subgraphs and tree decompositions {V}. {O}ne neighbor in a hole}.
\newblock \emph{J. Graph Theory}, 105(4):542--561, 2024.

\bibitem[{Alon and Chung(1988)}]{AC98}
\textsc{Noga Alon and Fan~R.K. Chung}.
\newblock \href{https://doi.org/10.1016/j.disc.2006.03.025}{Explicit construction of linear sized tolerant networks}.
\newblock \emph{Discrete Math.}, 72(1-3):15--19, 1988.

\bibitem[{Bannister et~al.(2019)Bannister, Devanny, Dujmovi\'c, Eppstein, and Wood}]{BDDEW19}
\textsc{Michael~J. Bannister, William~E. Devanny, Vida Dujmovi\'c, David Eppstein, and David~R. Wood}.
\newblock \href{https://doi.org/10.1007/s00453-018-0487-5}{Track layouts, layered path decompositions, and leveled planarity}.
\newblock \emph{Algorithmica}, 81(4):1561--1583, 2019.

\bibitem[{Barrera-Cruz et~al.(2019)Barrera-Cruz, Felsner, M\'{e}sz\'{a}ros, Micek, Smith, Taylor, and Trotter}]{BFMMSTT19}
\textsc{Fidel Barrera-Cruz, Stefan Felsner, Tam\'{a}s M\'{e}sz\'{a}ros, Piotr Micek, Heather Smith, Libby Taylor, and William~T. Trotter}.
\newblock \href{https://doi.org/10.1016/j.jctb.2019.02.003}{Separating tree-chromatic number from path-chromatic number}.
\newblock \emph{J. Combin. Theory Ser. B}, 138:206--218, 2019.

\bibitem[{Bekos et~al.(2022)Bekos, Da~Lozzo, Hlin\v{e}n\'{y}, and Kaufmann}]{BDHK22}
\textsc{Michael~A. Bekos, Giordano Da~Lozzo, Petr Hlin\v{e}n\'{y}, and Michael Kaufmann}.
\newblock \href{https://doi.org/10.4230/LIPIcs.ISAAC.2022.23}{Graph product structure for $h$-framed graphs}.
\newblock In \textsc{Sang~Won Bae and Heejin Park}, eds., \emph{Proc. 33rd Int'l Symp.\ on Algorithms and Computation \textup{(ISAAC '22)}}, vol. 248 of \emph{LIPIcs}, pp. 23:1--23. Schloss Dagstuhl, 2022.
\newblock arXiv:2204.11495.

\bibitem[{Berger and Seymour(2024)}]{BS24}
\textsc{Eli Berger and Paul Seymour}.
\newblock \href{https://doi.org/10.1007/s00493-024-00088-1}{Bounded-diameter tree-decompositions}.
\newblock \emph{Combinatorica}, 44(3):659--674, 2024.

\bibitem[{Bl\"{a}sius et~al.(2016)Bl\"{a}sius, Friedrich, and Krohmer}]{BFK16}
\textsc{Thomas Bl\"{a}sius, Tobias Friedrich, and Anton Krohmer}.
\newblock \href{https://doi.org/10.4230/LIPIcs.ESA.2016.15}{Hyperbolic random graphs: separators and treewidth}.
\newblock In \emph{Proc. 24th {A}nnual {E}uropean {S}ymposium on {A}lgorithms \textup{(ESA)}}, vol.~57 of \emph{LIPIcs. Leibniz Int. Proc. Inform.}, pp. 15:1--15.16. Schloss Dagstuhl. Leibniz-Zent. Inform., 2016.

\bibitem[{Bodlaender(1998)}]{Bodlaender98}
\textsc{Hans~L. Bodlaender}.
\newblock \href{https://doi.org/10.1016/S0304-3975(97)00228-4}{A partial $k$-arboretum of graphs with bounded treewidth}.
\newblock \emph{Theoret. Comput. Sci.}, 209(1-2):1--45, 1998.

\bibitem[{Bonamy et~al.(2024)Bonamy, Bousquet, Esperet, Groenland, Liu, Pirot, and Scott}]{BBEGLPS24}
\textsc{Marthe Bonamy, Nicolas Bousquet, Louis Esperet, Carla Groenland, Chun-Hung Liu, François Pirot, and Alex Scott}.
\newblock \href{https://doi.org/10.4171/jems/1341}{Asymptotic dimension of minor-closed families and {A}ssouad--{N}agata dimension of surfaces}.
\newblock \emph{J. European Math. Soc.}, 26:3739--3791, 2024.

\bibitem[{Bonamy et~al.(2022)Bonamy, Gavoille, and Pilipczuk}]{BGP22}
\textsc{Marthe Bonamy, Cyril Gavoille, and Micha{\l} Pilipczuk}.
\newblock \href{https://doi.org/10.1137/20M1330464}{Shorter labeling schemes for planar graphs}.
\newblock \emph{SIAM J. Discrete Math.}, 36(3):2082--2099, 2022.

\bibitem[{Bonnet et~al.(2022)Bonnet, Kwon, and Wood}]{BKW}
\textsc{\'Edouard Bonnet, {O-joung} Kwon, and David~R. Wood}.
\newblock \href{https://arxiv.org/abs/2202.11858}{Reduced bandwidth: a qualitative strengthening of twin-width in minor-closed classes (and beyond)}.
\newblock 2022, arXiv:2202.11858.

\bibitem[{Bordenave(2020)}]{Bord20}
\textsc{Charles Bordenave}.
\newblock \href{https://doi.org/10.24033/asens.2450}{A new proof of {F}riedman's second eigenvalue theorem and its extension to random lifts}.
\newblock \emph{Ann. Sci. \'{E}c. Norm. Sup\'{e}r. (4)}, 53(6):1393--1439, 2020.

\bibitem[{Bose et~al.(2022)Bose, Dujmovi\'c, Javarsineh, Morin, and Wood}]{BDJMW22}
\textsc{Prosenjit Bose, Vida Dujmovi\'c, Mehrnoosh Javarsineh, Pat Morin, and David~R. Wood}.
\newblock \href{https://doi.org/10.46298/dmtcs.7458}{Separating layered treewidth and row treewidth}.
\newblock \emph{Disc. Math. Theor. Comput. Sci.}, 24(1):\#18, 2022.

\bibitem[{Burling(1965)}]{Burling65}
\textsc{James~Perkins Burling}.
\newblock On coloring problems of families of polytopes.
\newblock Ph.D. thesis, University of Colorado, 1965.

\bibitem[{Campbell et~al.(2024)Campbell, Clinch, Distel, Gollin, Hendrey, Hickingbotham, Huynh, Illingworth, Tamitegama, Tan, and Wood}]{UTW}
\textsc{Rutger Campbell, Katie Clinch, Marc Distel, J.~Pascal Gollin, Kevin Hendrey, Robert Hickingbotham, Tony Huynh, Freddie Illingworth, Youri Tamitegama, Jane Tan, and David~R. Wood}.
\newblock \href{https://doi.org/10.1017/S0963548323000457}{Product structure of graph classes with bounded treewidth}.
\newblock \emph{Combin. Probab. Comput.}, 33(3):351--376, 2024.

\bibitem[{Dallard et~al.(2024{\natexlab{a}})Dallard, Fomin, Golovach, Korhonen, and Milanic}]{DFGKM24}
\textsc{Cl{\'{e}}ment Dallard, Fedor~V. Fomin, Petr~A. Golovach, Tuukka Korhonen, and Martin Milanic}.
\newblock \href{https://doi.org/10.4230/LIPIcs.ICALP.2024.51}{Computing tree decompositions with small independence number}.
\newblock In \textsc{Karl Bringmann, Martin Grohe, Gabriele Puppis, and Ola Svensson}, eds., \emph{Proc. 51st International Colloquium on Automata, Languages, and Programming \textup{({ICALP} '24}}, vol. 297 of \emph{LIPIcs}, pp. 51:1--51:18. Schloss Dagstuhl, 2024{\natexlab{a}}.

\bibitem[{Dallard et~al.(2021)Dallard, Milani\v{c}, and \v{S}torgel}]{DMS21}
\textsc{Cl\'{e}ment Dallard, Martin Milani\v{c}, and Kenny \v{S}torgel}.
\newblock \href{https://doi.org/10.1137/20M1352119}{Treewidth versus clique number. {I}. {G}raph classes with a forbidden structure}.
\newblock \emph{SIAM J. Discrete Math.}, 35(4):2618--2646, 2021.

\bibitem[{Dallard et~al.(2024{\natexlab{b}})Dallard, Milani\v{c}, and \v{S}torgel}]{DMS24a}
\textsc{Cl\'{e}ment Dallard, Martin Milani\v{c}, and Kenny \v{S}torgel}.
\newblock \href{https://doi.org/10.1016/j.jctb.2023.10.006}{Treewidth versus clique number. {II}. {T}ree-independence number}.
\newblock \emph{J. Combin. Theory Ser. B}, 164:404--442, 2024{\natexlab{b}}.

\bibitem[{Dallard et~al.(2024{\natexlab{c}})Dallard, Milani\v{c}, and \v{S}torgel}]{DMS24b}
\textsc{Cl\'{e}ment Dallard, Martin Milani\v{c}, and Kenny \v{S}torgel}.
\newblock \href{https://doi.org/10.1016/j.jctb.2024.03.005}{Treewidth versus clique number. {III}. {T}ree-independence number of graphs with a forbidden structure}.
\newblock \emph{J. Combin. Theory Ser. B}, 167:338--391, 2024{\natexlab{c}}.

\bibitem[{Dallard et~al.(2024{\natexlab{d}})Dallard, Krnc, joung Kwon, Milanič, Munaro, Štorgel, and Wiederrecht}]{DKKMMSW24}
\textsc{Clément Dallard, Matjaž Krnc, O~joung Kwon, Martin Milanič, Andrea Munaro, Kenny Štorgel, and Sebastian Wiederrecht}.
\newblock \href{http://arxiv.org/abs/2402.11222}{Treewidth versus clique number. {IV.} {T}ree-independence number of graphs excluding an induced star}.
\newblock 2024{\natexlab{d}}, arXiv:2402.11222.

\bibitem[{Demaine and Hajiaghayi(2004{\natexlab{a}})}]{DH-Algo04}
\textsc{Erik~D. Demaine and MohammadTaghi Hajiaghayi}.
\newblock \href{https://doi.org/10.1007/s00453-004-1106-1}{Diameter and treewidth in minor-closed graph families, revisited}.
\newblock \emph{Algorithmica}, 40(3):211--215, 2004{\natexlab{a}}.

\bibitem[{Demaine and Hajiaghayi(2004{\natexlab{b}})}]{DH-SODA04}
\textsc{Erik~D. Demaine and MohammadTaghi Hajiaghayi}.
\newblock \href{http://dl.acm.org/citation.cfm?id=982792.982919}{Equivalence of local treewidth and linear local treewidth and its algorithmic applications}.
\newblock In \emph{Proc. 15th Annual ACM-SIAM Symp. on Discrete Algorithms \textup{(SODA '04)}}, pp. 840--849. SIAM, 2004{\natexlab{b}}.

\bibitem[{Demaine et~al.(2010)Demaine, Hajiaghayi, and Kawarabayashi}]{DHK-SODA10}
\textsc{Erik~D. Demaine, MohammadTaghi Hajiaghayi, and {Ken-ichi} Kawarabayashi}.
\newblock \href{http://portal.acm.org/citation.cfm?id=1873629}{Decomposition, approximation, and coloring of odd-minor-free graphs}.
\newblock In \emph{Proc. 21st Annual ACM-SIAM Symp. on Discrete Algorithms \textup{(SODA '10)}}, pp. 329--344. ACM Press, 2010.

\bibitem[{DeVos et~al.(2004)DeVos, Ding, Oporowski, Sanders, Reed, Seymour, and Vertigan}]{DDOSRSV04}
\textsc{Matt DeVos, Guoli Ding, Bogdan Oporowski, Daniel~P. Sanders, Bruce Reed, Paul Seymour, and Dirk Vertigan}.
\newblock \href{https://doi.org/10.1016/j.jctb.2003.09.001}{Excluding any graph as a minor allows a low tree-width 2-coloring}.
\newblock \emph{J. Combin. Theory Ser. B}, 91(1):25--41, 2004.

\bibitem[{Diestel(2018)}]{Diestel5}
\textsc{Reinhard Diestel}.
\newblock Graph theory, vol. 173 of \emph{Graduate Texts in Mathematics}.
\newblock Springer, 5th edn., 2018.

\bibitem[{Ding et~al.(1998)Ding, Oporowski, Sanders, and Vertigan}]{DOSV98}
\textsc{Guoli Ding, Bogdan Oporowski, Daniel~P. Sanders, and Dirk Vertigan}.
\newblock \href{https://doi.org/10.1007/s004930050001}{Partitioning graphs of bounded tree-width}.
\newblock \emph{Combinatorica}, 18(1):1--12, 1998.

\bibitem[{Distel et~al.(2024)Distel, Dujmovi\'c, Eppstein, Hickingbotham, Joret, Micek, Morin, Seweryn, and Wood}]{DDEHJMMSW24}
\textsc{Marc Distel, Vida Dujmovi\'c, David Eppstein, Robert Hickingbotham, Gwena\"el Joret, Piotr Micek, Pat Morin, Micha{\l}~T. Seweryn, and David~R. Wood}.
\newblock \href{https://doi.org/10.1137/23M1591773}{Product structure extension of the {A}lon--{S}eymour--{T}homas theorem}.
\newblock \emph{SIAM J. Disc. Math.}, 38(3), 2024.

\bibitem[{Distel et~al.(2022)Distel, Hickingbotham, Huynh, and Wood}]{DHHW22}
\textsc{Marc Distel, Robert Hickingbotham, Tony Huynh, and David~R. Wood}.
\newblock \href{https://doi.org/10.48550/arXiv.2112.10025}{Improved product structure for graphs on surfaces}.
\newblock \emph{Discrete Math. Theor. Comput. Sci.}, 24(2):\#6, 2022.

\bibitem[{Distel et~al.(2023)Distel, Hickingbotham, Seweryn, and Wood}]{DHSW}
\textsc{Marc Distel, Robert Hickingbotham, Michał~T. Seweryn, and David~R. Wood}.
\newblock \href{http://arxiv.org/abs/2308.06995}{Powers of planar graphs, product structure, and blocking partitions}.
\newblock 2023, arXiv:2308.06995.

\bibitem[{D\k{e}bski et~al.(2021)D\k{e}bski, Felsner, Micek, and Schr\"{o}der}]{DFMS21}
\textsc{Micha{\l} D\k{e}bski, Stefan Felsner, Piotr Micek, and Felix Schr\"{o}der}.
\newblock \href{https://doi.org/10.19086/aic.27351}{Improved bounds for centered colorings}.
\newblock \emph{Adv. Comb.}, \#8, 2021.

\bibitem[{Dourisboure and Gavoille(2007)}]{DG07}
\textsc{Yon Dourisboure and Cyril Gavoille}.
\newblock \href{https://doi.org/10.1016/j.disc.2005.12.060}{Tree-decompositions with bags of small diameter}.
\newblock \emph{Discret. Math.}, 307(16):2008--2029, 2007.

\bibitem[{Dujmovi\'c et~al.(2020)Dujmovi\'c, Eppstein, Joret, Morin, and Wood}]{DEJMW20}
\textsc{Vida Dujmovi\'c, David Eppstein, Gwena\"el Joret, Pat Morin, and David~R. Wood}.
\newblock \href{https://doi.org/10.1137/18M122162X}{Minor-closed graph classes with bounded layered pathwidth}.
\newblock \emph{SIAM J. Disc. Math.}, 34(3):1693--1709, 2020.

\bibitem[{Dujmovi\'c et~al.(2021)Dujmovi\'c, Esperet, Gavoille, Joret, Micek, and Morin}]{DEGJMM21}
\textsc{Vida Dujmovi\'c, Louis Esperet, Cyril Gavoille, Gwena\"el Joret, Piotr Micek, and Pat Morin}.
\newblock \href{https://doi.org/10.1145/3477542}{Adjacency labelling for planar graphs (and beyond)}.
\newblock \emph{J. ACM}, 68(6):42, 2021.

\bibitem[{Dujmovi{\'c} et~al.(2020{\natexlab{a}})Dujmovi{\'c}, Esperet, Joret, Walczak, and Wood}]{DEJWW20}
\textsc{Vida Dujmovi{\'c}, Louis Esperet, Gwena\"{e}l Joret, Bartosz Walczak, and David~R. Wood}.
\newblock \href{https://doi.org/10.19086/aic.12100}{Planar graphs have bounded nonrepetitive chromatic number}.
\newblock \emph{Adv. Comb.}, \#5, 2020{\natexlab{a}}.

\bibitem[{Dujmovi{\'c} et~al.(2022)Dujmovi{\'c}, Esperet, Morin, Walczak, and Wood}]{DEMWW22}
\textsc{Vida Dujmovi{\'c}, Louis Esperet, Pat Morin, Bartosz Walczak, and David~R. Wood}.
\newblock \href{https://doi.org/10.1017/S0963548321000213}{Clustered 3-colouring graphs of bounded degree}.
\newblock \emph{Combin. Probab. Comput.}, 31(1):123--135, 2022.

\bibitem[{Dujmovi\'c et~al.(2023)Dujmovi\'c, Hickingbotham, Hodor, Joret, La, Micek, Morin, Rambaud, and Wood}]{DHHJLMMRW}
\textsc{Vida Dujmovi\'c, Robert Hickingbotham, Jędrzej Hodor, Gwena\"el Joret, Hoang La, Piotr Micek, Pat Morin, Clément Rambaud, and David~R. Wood}.
\newblock \href{https://doi.org/10.1137/1.9781611977912.48}{The grid-minor theorem revisited}.
\newblock In \emph{Proc. 2024 Annual ACM-SIAM Symposium on Discrete Algorithms \textup{(SODA '24)}}, pp. 1241--1245. 2023.
\newblock arXiv:2307.02816.

\bibitem[{Dujmovi{\'c} et~al.(2020{\natexlab{b}})Dujmovi{\'c}, Joret, Micek, Morin, Ueckerdt, and Wood}]{DJMMUW20}
\textsc{Vida Dujmovi{\'c}, Gwena\"{e}l Joret, Piotr Micek, Pat Morin, Torsten Ueckerdt, and David~R. Wood}.
\newblock \href{https://doi.org/10.1145/3385731}{Planar graphs have bounded queue-number}.
\newblock \emph{J. ACM}, 67(4):\#22, 2020{\natexlab{b}}.

\bibitem[{Dujmovi\'c et~al.(2018)Dujmovi\'c, Joret, Morin, Norin, and Wood}]{DJMNW18}
\textsc{Vida Dujmovi\'c, Gwena\"el Joret, Pat Morin, Sergey Norin, and David~R. Wood}.
\newblock \href{https://doi.org/10.1137/17M1112637}{Orthogonal tree decompositions of graphs}.
\newblock \emph{SIAM J. Discrete Math.}, 32(2):839--863, 2018.

\bibitem[{Dujmovi{\'c} et~al.(2017)Dujmovi{\'c}, Morin, and Wood}]{DMW17}
\textsc{Vida Dujmovi{\'c}, Pat Morin, and David~R. Wood}.
\newblock \href{https://doi.org/10.1016/j.jctb.2017.05.006}{Layered separators in minor-closed graph classes with applications}.
\newblock \emph{J. Combin. Theory Ser. B}, 127:111--147, 2017.

\bibitem[{Dujmovi{\'c} et~al.(2023)Dujmovi{\'c}, Morin, and Wood}]{DMW23}
\textsc{Vida Dujmovi{\'c}, Pat Morin, and David~R. Wood}.
\newblock \href{https://doi.org/10.1016/j.jctb.2023.03.004}{Graph product structure for non-minor-closed classes}.
\newblock \emph{J. Combin. Theory Ser. B}, 162:34--67, 2023.

\bibitem[{Dujmović et~al.(2024{\natexlab{a}})Dujmović, Hickingbotham, Joret, Micek, Morin, and Wood}]{DHJMMW24}
\textsc{Vida Dujmović, Robert Hickingbotham, Gwenaël Joret, Piotr Micek, Pat Morin, and David~R. Wood}.
\newblock \href{https://doi.org/10.1017/S0963548323000275}{The excluded tree minor theorem revisited}.
\newblock \emph{Combin. Probab. Comput.}, 33(1):85--90, 2024{\natexlab{a}}.

\bibitem[{Dujmović et~al.(2024{\natexlab{b}})Dujmović, Joret, Micek, Morin, and Wood}]{DJMMW24}
\textsc{Vida Dujmović, Gwenaël Joret, Piotr Micek, Pat Morin, and David~R. Wood}.
\newblock \href{https://doi.org/10.37236/11712}{Bounded-degree planar graphs do not have bounded-degree product structure}.
\newblock \emph{Electron. J. Combin.}, 31(2), 2024{\natexlab{b}}.

\bibitem[{Dvor{\'{a}}k et~al.(2022)Dvor{\'{a}}k, Gon{\c{c}}alves, Lahiri, Tan, and Ueckerdt}]{DGLTU22}
\textsc{Zdenek Dvor{\'{a}}k, Daniel Gon{\c{c}}alves, Abhiruk Lahiri, Jane Tan, and Torsten Ueckerdt}.
\newblock \href{https://doi.org/10.4230/LIPIcs.SoCG.2022.38}{On comparable box dimension}.
\newblock In \textsc{Xavier Goaoc and Michael Kerber}, eds., \emph{Proc. 38th Int'l Symp. on Computat. Geometry \textup{(SoCG 2022)}}, vol. 224 of \emph{LIPIcs}, pp. 38:1--38:14. Schloss Dagstuhl, 2022.

\bibitem[{Dvo{\v{r}}{\'a}k and Thomas(2014)}]{DvoTho}
\textsc{Zden{\v{e}}k Dvo{\v{r}}{\'a}k and Robin Thomas}.
\newblock \href{http://arxiv.org/abs/1401.1399}{List-coloring apex-minor-free graphs}.
\newblock 2014, arXiv:1401.1399.

\bibitem[{Ellis et~al.(1994)Ellis, Sudborough, and Turner}]{EST-IC94}
\textsc{John~A. Ellis, I.~Hal Sudborough, and Jonathan~S. Turner}.
\newblock \href{https://doi.org/10.1006/inco.1994.1064}{The vertex separation and search number of a graph}.
\newblock \emph{Inform. and Comput.}, 113(1):50--79, 1994.

\bibitem[{Eppstein(1999)}]{Eppstein99}
\textsc{David Eppstein}.
\newblock \href{https://doi.org/10.7155/jgaa.00014}{Subgraph isomorphism in planar graphs and related problems}.
\newblock \emph{J. Graph Algorithms Appl.}, 3(3):1--27, 1999.

\bibitem[{Eppstein(2000)}]{Eppstein-Algo00}
\textsc{David Eppstein}.
\newblock \href{https://doi.org/10.1007/s004530010020}{Diameter and treewidth in minor-closed graph families}.
\newblock \emph{Algorithmica}, 27(3--4):275--291, 2000.

\bibitem[{Eppstein et~al.(2024)Eppstein, Hickingbotham, Merker, Norin, Seweryn, and Wood}]{EHMNSW24}
\textsc{David Eppstein, Robert Hickingbotham, Laura Merker, Sergey Norin, Micha{\l}~T. Seweryn, and David~R. Wood}.
\newblock \href{https://doi.org/10.1007/s00454-022-00478-6}{Three-dimensional graph products with unbounded stack-number}.
\newblock \emph{Discrete Comput. Geom.}, 71:1210--1237, 2024.

\bibitem[{Esperet et~al.(2023)Esperet, Joret, and Morin}]{EJM23}
\textsc{Louis Esperet, Gwena\"{e}l Joret, and Pat Morin}.
\newblock \href{https://doi.org/10.1112/jlms.12781}{Sparse universal graphs for planarity}.
\newblock \emph{J. London Math. Soc.}, 108(4):1333--1357, 2023.

\bibitem[{Felsner et~al.(2018)Felsner, Joret, Micek, Trotter, and Wiechert}]{FJMTW18}
\textsc{Stefan Felsner, Gwena\"el Joret, Piotr Micek, William~T. Trotter, and Veit Wiechert}.
\newblock \href{https://doi.org/10.37236/7052}{Burling graphs, chromatic number, and orthogonal tree-decompositions}.
\newblock \emph{Electron. J. Combin.}, 25(1):\#P1.35, 2018.

\bibitem[{Friedman(2008)}]{Friedman08}
\textsc{Joel Friedman}.
\newblock \href{https://doi.org/10.1090/memo/0910}{A proof of {A}lon's second eigenvalue conjecture and related problems}.
\newblock \emph{Mem. Amer. Math. Soc.}, 195(910), 2008.

\bibitem[{Gale(1979)}]{Gale79}
\textsc{David Gale}.
\newblock \href{https://doi.org/10.2307/2320146}{The game of {H}ex and the {B}rouwer fixed-point theorem}.
\newblock \emph{Amer. Math. Monthly}, 86(10):818--827, 1979.

\bibitem[{Gao(2012)}]{Gao12}
\textsc{Yong Gao}.
\newblock \href{https://doi.org/10.1016/j.dam.2011.10.013}{Treewidth of {E}rd{\H{o}}s-{R}\'{e}nyi random graphs, random intersection graphs, and scale-free random graphs}.
\newblock \emph{Discrete Appl. Math.}, 160(4-5):566--578, 2012.

\bibitem[{Geelen et~al.(2009)Geelen, Gerards, Reed, Seymour, and Vetta}]{GRRSV09}
\textsc{Jim Geelen, Bert Gerards, Bruce~A. Reed, Paul~D. Seymour, and Adrian Vetta}.
\newblock \href{https://doi.org/10.1016/j.jctb.2008.03.006}{On the odd-minor variant of {H}adwiger's conjecture}.
\newblock \emph{J. Comb. Theory, Ser. {B}}, 99(1):20--29, 2009.

\bibitem[{Harvey and Wood(2017)}]{HW17}
\textsc{Daniel~J. Harvey and David~R. Wood}.
\newblock \href{https://doi.org/10.1002/jgt.22030}{Parameters tied to treewidth}.
\newblock \emph{J. Graph Theory}, 84(4):364--385, 2017.

\bibitem[{Haxell et~al.(2003)Haxell, Szab\'o, and Tardos}]{HST03}
\textsc{Penny Haxell, Tibor Szab\'o, and G\'abor Tardos}.
\newblock \href{https://doi.org/10.1016/S0095-8956(03)00031-5}{Bounded size components---partitions and transversals}.
\newblock \emph{J. Combin. Theory Ser. B}, 88(2):281--297, 2003.

\bibitem[{Hayward and Toft(2019)}]{HT19}
\textsc{Ryan~B. Hayward and Bjarne Toft}.
\newblock \href{https://doi.org/10.1201/9780429031960}{Hex, inside and out---the full story}.
\newblock CRC Press, 2019.

\bibitem[{Hickingbotham et~al.(2024)Hickingbotham, Jungeblut, Merker, and Wood}]{HJMW24}
\textsc{Robert Hickingbotham, Paul Jungeblut, Laura Merker, and David~R. Wood}.
\newblock \href{https://doi.org/10.1002/jgt.23008}{The product structure of squaregraphs}.
\newblock \emph{J. Graph Theory}, 105(2):179--191, 2024.

\bibitem[{Hickingbotham and Wood(2024)}]{HW24}
\textsc{Robert Hickingbotham and David~R. Wood}.
\newblock \href{https://doi.org/10.1137/22M1540296}{Shallow minors, graph products and beyond-planar graphs}.
\newblock \emph{SIAM J. Discrete Math.}, 38(1):1057--1089, 2024.

\bibitem[{Hliněný and Jedelský(2024)}]{HJ24}
\textsc{Petr Hliněný and Jan Jedelský}.
\newblock \href{https://drops.dagstuhl.de/entities/document/10.4230/LIPIcs.MFCS.2024.61}{$\mathcal{H}$-clique-width and a hereditary analogue of product structure}.
\newblock In \textsc{Rastislav Kr\'{a}lovi\v{c} and Anton{\'\i}n Ku\v{c}era}, eds., \emph{Proc. 49th Int'l Symp. on Math. Foundations of Comput. Sci. \textup{(MFCS 2024)}}, vol. 306 of \emph{LIPIcs}, pp. 61:1--61:16. Schloss Dagstuhl, 2024.

\bibitem[{Hodor et~al.(2024)Hodor, La, Micek, and Rambaud}]{HLMR}
\textsc{Jędrzej Hodor, Hoang La, Piotr Micek, and Clément Rambaud}.
\newblock \href{http://arxiv.org/abs/2404.17306}{Quickly excluding an apex-forest}.
\newblock 2024, arXiv:2404.17306.

\bibitem[{Huynh and Kim(2017)}]{HK17}
\textsc{Tony Huynh and Ringi Kim}.
\newblock \href{https://doi.org/10.1002/jgt.22121}{Tree-chromatic number is not equal to path-chromatic number}.
\newblock \emph{J. Graph Theory}, 86(2):213--222, 2017.

\bibitem[{Huynh et~al.(2021{\natexlab{a}})Huynh, Mohar, {\v{S}}{\'a}mal, Thomassen, and Wood}]{HMSTW}
\textsc{Tony Huynh, Bojan Mohar, Robert {\v{S}}{\'a}mal, Carsten Thomassen, and David~R. Wood}.
\newblock \href{https://arxiv.org/abs/2109.00327}{Universality in minor-closed graph classes}.
\newblock 2021{\natexlab{a}}, arXiv:2109.00327.

\bibitem[{Huynh et~al.(2021{\natexlab{b}})Huynh, Reed, Wood, and Yepremyan}]{HRWY21}
\textsc{Tony Huynh, Bruce Reed, David~R. Wood, and Liana Yepremyan}.
\newblock \href{https://doi.org/10.1007/978-3-030-62497-2_30}{Notes on tree- and path-chromatic number}.
\newblock In \textsc{David~R. Wood, Jan de~Gier, Cheryl~E. Praeger, and Terence Tao}, eds., \emph{2019-20 MATRIX Annals}, pp. 489--498. Springer, 2021{\natexlab{b}}.

\bibitem[{Illingworth et~al.(2022)Illingworth, Scott, and Wood}]{ISW}
\textsc{Freddie Illingworth, Alex Scott, and David~R. Wood}.
\newblock \href{https://arxiv.org/abs/2104.06627}{Product structure of graphs with an excluded minor}.
\newblock 2022, arXiv:2104.06627.
\newblock \emph{Trans. Amer. Math. Soc.}, to appear.

\bibitem[{Jacob and Pilipczuk(2022)}]{JP22}
\textsc{Hugo Jacob and Marcin Pilipczuk}.
\newblock \href{https://doi.org/10.1007/978-3-031-15914-5\_21}{Bounding twin-width for bounded-treewidth graphs, planar graphs, and bipartite graphs}.
\newblock In \textsc{Michael~A. Bekos and Michael Kaufmann}, eds., \emph{Proc. 48th International Workshop on Graph-Theoretic Concepts in Computer Science \textup{({WG} 2022})}, vol. 13453 of \emph{Lecture Notes in Comput. Sci.}, pp. 287--299. Springer, 2022.

\bibitem[{Liu(2024)}]{Liu24}
\textsc{Chun-Hung Liu}.
\newblock \href{https://doi.org/10.1007/s00493-024-00081-8}{Defective coloring is perfect for minors}.
\newblock \emph{Combinatorica}, 44:467--507, 2024.

\bibitem[{Liu and Wood(2019)}]{LW1}
\textsc{Chun-Hung Liu and David~R. Wood}.
\newblock \href{http://arxiv.org/abs/1905.08969}{Clustered graph coloring and layered treewidth}.
\newblock 2019, arXiv:1905.08969.

\bibitem[{Liu and Wood(2023)}]{LW4}
\textsc{Chun-Hung Liu and David~R. Wood}.
\newblock \href{https://doi.org/10.1016/j.ejc.2023.103730}{Clustered coloring of graphs with bounded layered treewidth and bounded degree}.
\newblock \emph{European J. Combin.}, p. 103730, 2023.

\bibitem[{Milanič and Rzążewski(2022)}]{MR22}
\textsc{Martin Milanič and Paweł Rzążewski}.
\newblock \href{http://arxiv.org/abs/2209.12315}{Tree decompositions with bounded independence number: beyond independent sets}.
\newblock 2022, arXiv:2209.12315.

\bibitem[{Miller et~al.(1997)Miller, Teng, Thurston, and Vavasis}]{MTTV97}
\textsc{Gary~L. Miller, Shang-Hua Teng, William Thurston, and Stephen~A. Vavasis}.
\newblock \href{https://doi.org/10.1145/256292.256294}{Separators for sphere-packings and nearest neighbor graphs}.
\newblock \emph{J. ACM}, 44(1):1--29, 1997.

\bibitem[{Mohar and Thomassen(2001)}]{MoharThom}
\textsc{Bojan Mohar and Carsten Thomassen}.
\newblock Graphs on surfaces.
\newblock Johns Hopkins University Press, 2001.

\bibitem[{Ne{\v{s}}et{\v{r}}il and Ossona~de Mendez(2012)}]{Sparsity}
\textsc{Jaroslav Ne{\v{s}}et{\v{r}}il and Patrice Ossona~de Mendez}.
\newblock \href{https://doi.org/10.1007/978-3-642-27875-4}{Sparsity}, vol.~28 of \emph{Algorithms and Combinatorics}.
\newblock Springer, 2012.

\bibitem[{Norin(2024)}]{Norin24}
\textsc{Sergey Norin}.
\newblock A high dimensional bramble lemma.
\newblock 2024.
\newblock \url{https://www.math.mcgill.ca/snorin/papers/HighDBrambles.pdf}.

\bibitem[{Reed(1997)}]{Reed97}
\textsc{Bruce~A. Reed}.
\newblock \href{https://doi.org/10.1017/CBO9780511662119.006}{Tree width and tangles: a new connectivity measure and some applications}.
\newblock In \textsc{R.~A. Bailey}, ed., \emph{Surveys in Combinatorics}, vol. 241 of \emph{London Math. Soc. Lecture Note Ser.}, pp. 87--162. Cambridge Univ. Press, 1997.

\bibitem[{Robertson and Seymour(1986{\natexlab{a}})}]{RS-II}
\textsc{Neil Robertson and Paul Seymour}.
\newblock \href{https://doi.org/10.1016/0196-6774(86)90023-4}{Graph minors. {II}. {A}lgorithmic aspects of tree-width}.
\newblock \emph{J. Algorithms}, 7(3):309--322, 1986{\natexlab{a}}.

\bibitem[{Robertson and Seymour(1986{\natexlab{b}})}]{RS-V}
\textsc{Neil Robertson and Paul Seymour}.
\newblock \href{https://doi.org/10.1016/0095-8956(86)90030-4}{Graph minors. {V}. {E}xcluding a planar graph}.
\newblock \emph{J. Combin. Theory Ser. B}, 41(1):92--114, 1986{\natexlab{b}}.

\bibitem[{Robertson and Seymour(2003)}]{RS-XVI}
\textsc{Neil Robertson and Paul Seymour}.
\newblock \href{https://doi.org/10.1016/S0095-8956(03)00042-X}{Graph minors. {XVI}. {E}xcluding a non-planar graph}.
\newblock \emph{J. Combin. Theory Ser. B}, 89(1):43--76, 2003.

\bibitem[{Seymour(2016)}]{Seymour16}
\textsc{Paul Seymour}.
\newblock \href{https://doi.org/10.1016/j.jctb.2015.08.002}{Tree-chromatic number}.
\newblock \emph{J. Combin. Theory Series B}, 116:229--237, 2016.

\bibitem[{Seymour and Thomas(1993)}]{ST93}
\textsc{Paul Seymour and Robin Thomas}.
\newblock \href{https://doi.org/10.1006/jctb.1993.1027}{Graph searching and a min-max theorem for tree-width}.
\newblock \emph{J. Combin. Theory Ser. B}, 58(1):22--33, 1993.

\bibitem[{Shahrokhi(2013)}]{Shahrokhi13}
\textsc{Farhad Shahrokhi}.
\newblock \href{http://arxiv.org/abs/1502.06175}{New representation results for planar graphs}.
\newblock In \emph{29th European Workshop on Computational Geometry \textup{(EuroCG 2013)}}, pp. 177--180. 2013.
\newblock arXiv:1502.06175.

\bibitem[{Shang(2022)}]{Shang22}
\textsc{Yilun Shang}.
\newblock \href{https://doi.org/10.3792/pjaa.98.015}{On the tree-depth and tree-width in heterogeneous random graphs}.
\newblock \emph{Proc. Japan Acad. Ser. A Math. Sci.}, 98(9):78--83, 2022.

\bibitem[{Stavropoulos(2015)}]{Stav15}
\textsc{Konstantinos Stavropoulos}.
\newblock \href{http://arxiv.org/abs/1512.01104}{On the medianwidth of graphs}.
\newblock 2015, arXiv:1512.01104.

\bibitem[{Stavropoulos(2016)}]{Stav16}
\textsc{Konstantinos Stavropoulos}.
\newblock \href{http://arxiv.org/abs/1603.06871}{Cops, robber and medianwidth parameters}.
\newblock 2016, arXiv:1603.06871.

\bibitem[{Ueckerdt et~al.(2022)Ueckerdt, Wood, and Yi}]{UWY22}
\textsc{Torsten Ueckerdt, David~R. Wood, and Wendy Yi}.
\newblock \href{https://doi.org/10.37236/10614}{An improved planar graph product structure theorem}.
\newblock \emph{Electron. J. Combin.}, 29:P2.51, 2022.

\bibitem[{von Staudt(1847)}]{vonStaudt}
\textsc{Karl Georg~Christian von Staudt}.
\newblock Geometrie der lage.
\newblock Verlag von Bauer and Rapse 25. Julius Merz, N{\"u}rnberg, 1847.

\bibitem[{Wagner(1937)}]{Wagner37}
\textsc{Klaus Wagner}.
\newblock \href{https://doi.org/10.1007/BF01594196}{{\"U}ber eine {E}igenschaft der ebene {K}omplexe}.
\newblock \emph{Math. Ann.}, 114:570--590, 1937.

\bibitem[{Wang et~al.(2011)Wang, Liu, Cui, and Xu}]{WLCX11}
\textsc{Chaoyi Wang, Tian Liu, Peng Cui, and Ke~Xu}.
\newblock \href{https://doi.org/10.1007/978-3-642-22616-8\_38}{A note on treewidth in random graphs}.
\newblock In \emph{Proc. Int'l Computing and Combinatorics Conference \textup{(COCOON 2006)}}, vol. 6831 of \emph{Lecture Notes in Comput. Sci.}, pp. 491--499. Springer, 2011.

\end{thebibliography}
}

\end{document}